\documentclass{amsart}

\usepackage{amsmath,amscd,xypic,amssymb,combelow,color,enumitem,graphicx,float}

\usepackage[normalem]{ulem}
\usepackage[dvipsnames]{xcolor}
\makeatletter
\def\squiggly{\bgroup \markoverwith{\textcolor{black}{\lower3.5\p@\hbox{\sixly \char58}}}\ULon}

\usepackage{xr-hyper}
\usepackage[
pdftex,
bookmarks=false,
colorlinks=true,
debug=true,
pdfnewwindow=true]{hyperref}

\makeatletter
  \@addtoreset{equation}{section}
\makeatother

\newtheorem{theorem}[subsection]{Theorem}

\newtheorem{proposition}[subsection]{Proposition}
\newtheorem{lemma}[subsection]{Lemma}
\newtheorem{corollary}[subsection]{Corollary}
\newtheorem{conjecture}[subsection]{Conjecture}
\newtheorem{definition}[subsection]{Definition}

\theoremstyle{remark}
\newtheorem{claim}[subsection]{Claim}

\newtheorem{remark}[subsection]{Remark}

\def\fb{{\mathfrak{b}}}
\def\fg{{\mathfrak{g}}}
\def\fh{{\mathfrak{h}}}

\def\fn{{\mathfrak{n}}}

\def\fsl{{\mathfrak{sl}}}

\def\hg{{\widehat{\fg}}}

\def\BA{{\mathbb{A}}}

\def\BN{{\mathbb{N}}}
\def\BF{{\mathbb{F}}}

\def\BQ{{\mathbb{Q}}}
\def\BZ{{\mathbb{Z}}}

\def\CA{{\mathcal{A}}}

\def\CF{{\mathcal{F}}}

\def\CT{{\mathcal{T}}}

\def\CV{{\mathcal{V}}}

\def\hCF{\CF^L}
\def\hCFext{\CF^{L,\text{ext}}}
\def\wPhi{\Phi^L}

\def\e{\varepsilon}
\def\ph{\varphi}

\def\sym{\textrm{Sym}}

\def\uu{U_q(\fg)}
\def\uup{U_q(\fn^+)}
\def\uupm{U_q(\fn^\pm)}
\def\uuo{U_q(\fh)}
\def\uum{U_q(\fn^-)}
\def\uug{U_q(\fb^+)}
\def\uul{U_q(\fb^-)}

\def\UU{U_q(L\fg)}
\def\UUp{U_q(L\fn^+)}
\def\UUpm{U_q(L\fn^\pm)}
\def\UUo{U_q(L\fh)}
\def\UUm{U_q(L\fn^-)}
\def\UUg{U_q(L\fb^+)}
\def\UUl{U_q(L\fb^-)}

\def\VV{U_q(\widehat{\fg})}
\def\VVp{U_q(\widehat{\fn}^+)}
\def\VVo{U_q(\widehat{\fh})}
\def\VVm{U_q(\widehat{\fn}^-)}
\def\VVpm{U_q(\widehat{\fn}^\pm)}
\def\VVg{U_q(\widehat{\fb}^+)}
\def\VVl{U_q(\widehat{\fb}^-)}

\def\bk{{\mathbf{k}}}
\def\bl{{\mathbf{l}}}

\def\sfe{{\mathsf{e}}}

\def\bs{\alpha}

\def\nn{{\mathbb{N}^I}}

\def\bare{e}

\def\loccit{\emph{loc.~cit.}}

\def\slaws{\text{standard Lyndon loop words}}

\newcommand{\iso}{{\stackrel{\sim}{\longrightarrow}}}

\def\hdeg{\text{hdeg}}
\def\vdeg{\text{vdeg }}

\def\wI{\widehat{I}}
\def\wQ{\widehat{Q}}
\def\wW{\widehat{W}}
\def\weW{\widehat{W}^{\mathrm{ext}}}

\def\wmu{\widehat{\mu}}
\def\wDelta{\widehat{\Delta}}

\def\talpha{\tilde{\alpha}}
\def\tbeta{\tilde{\beta}}
\def\tgamma{\tilde{\gamma}}
\def\obeta{\overline{\beta}}
\def\ogamma{\overline{\gamma}}
\def\oalpha{\overline{\alpha}}

\newcommand{\hooklongrightarrow}{\lhook\joinrel\longrightarrow}

\begin{document}

\title[Quantum loop groups and shuffle algebras via Lyndon words]
      {\Large{\textbf{Quantum loop groups and shuffle algebras via Lyndon words}}}
	
\author[Andrei Negu\cb t and Alexander Tsymbaliuk]{Andrei Negu\cb t and Alexander Tsymbaliuk}

\address{A.N.: MIT, Department of Mathematics, Cambridge, MA, USA}
\address{Simion Stoilow Institute of Mathematics, Bucharest, Romania}
\email{andrei.negut@gmail.com}

\address{A.T.: Purdue University, Department of Mathematics, West Lafayette, IN, USA}
\email{sashikts@gmail.com}
	
\maketitle
	
\begin{abstract}
We study PBW bases of the untwisted quantum loop group $\UU$ (in the Drinfeld new presentation) using the combinatorics of
loop words, by generalizing the treatment of~\cite{LR, L, R2} in the finite type case. As an application, we prove that
Enriquez' homomorphism~\cite{E1} from the positive half of the quantum loop group to the trigonometric degeneration of
Feigin-Odesskii's shuffle algebra~\cite{FO} associated to $\fg$ is an isomorphism.
\end{abstract}


\section{Introduction}


\subsection{}
\label{sub:classical pbw}

Let $\fg$ be the Kac-Moody Lie algebra corresponding to a root system of finite type. Associated with a decomposition
of the set of roots $\Delta = \Delta^+ \sqcup \Delta^-$, there exists a triangular decomposition:
\begin{equation}
\label{eqn:decomp intro 1}
  \fg = \fn^+ \oplus \fh \oplus \fn^-
\end{equation}
where:
\begin{equation}
\label{eqn:root vectors intro}
  \fn^+ = \bigoplus_{\alpha \in \Delta^+} \BQ \cdot \bare_\alpha
\end{equation}
and analogously for $\fn^-$. The elements $e_\alpha$ will be called \underline{root vectors}.
Formula~\eqref{eqn:decomp intro 1} induces a triangular decomposition\footnote{Given subalgebras $\{A_k\}_{k=1}^N$
of an algebra $A$, the decomposition $A = A_1\otimes \dots \otimes A_N$ will mean that the multiplication in $A$
induces a vector space isomorphism $m\colon A_1\otimes \dots \otimes A_N\iso A$.} of the universal enveloping algebra:
\begin{equation}
\label{eqn:decomp intro 1.2}
  U(\fg) = U(\fn^+) \otimes U(\fh) \otimes U(\fn^-)
\end{equation}
Then the PBW theorem asserts that a linear basis of $U(\fn^+)$ is given by the products:
\begin{equation}
\label{eqn:pbw intro 1}
  U(\fn^+) \ =
  \bigoplus^{k\in \BN}_{\gamma_1 \geq \dots \geq \gamma_k \in \Delta^+}
    \BQ \cdot \bare_{\gamma_1}\dots \bare_{\gamma_k}
\end{equation}
and analogously for $U(\fn^-)$, for any total order of the set of positive roots~$\Delta^+$ (the set $\BN$ will be assumed to include $0$).
The root vectors~\eqref{eqn:root vectors intro} can be normalized so that we have:
\begin{equation}
\label{eqn:comm intro 1}
  [\bare_{\alpha} , \bare_{\beta}] = \bare_{\alpha} \bare_{\beta} - \bare_{\beta} \bare_{\alpha}
  \in \BZ^* \cdot e_{\alpha + \beta}
\end{equation}
whenever $\alpha, \beta$ and $\alpha+\beta$ are positive roots. Thus we see that formula~\eqref{eqn:comm intro 1}
provides an algorithm for constructing, up to scalar multiple, all the root vectors \eqref{eqn:root vectors intro}
inductively starting from $\bare_i = \bare_{\alpha_i}$, where $\{\alpha_i\}_{i\in I}\subset \Delta^+$ are
the simple roots of $\fg$. The upshot is that all the root vectors $\bare_\alpha$, and with them the PBW basis
\eqref{eqn:pbw intro 1}, can be read off from the combinatorics of the root system.

\medskip


\subsection{}
\label{sub:quantum finite pbw}

The quantum group $\uu$ is a $q$-deformation of the universal enveloping algebra $U(\fg)$,
and we will focus on emulating the features of the previous Subsection. For one thing,
there exists a triangular decomposition analogous to~\eqref{eqn:decomp intro 1.2}:
\begin{equation}
\label{eqn:decomp intro 2}
  \uu = \uup \otimes U_q(\fh) \otimes \uum
\end{equation}
and there exists a PBW basis analogous to~\eqref{eqn:pbw intro 1}:
\begin{equation}
\label{eqn:pbw intro 2}
  \uup \ =
  \bigoplus^{k\in \BN}_{\gamma_1 \geq \dots \geq \gamma_k \in \Delta^+}
    \BQ(q) \cdot e_{\gamma_1} \dots e_{\gamma_k}
\end{equation}
The $q$-deformed root vectors $e_\alpha \in \uup$ are defined via Lusztig's braid group action, which requires one to choose
a reduced decomposition of the longest element in the Weyl group of type $\fg$. It is well-known (\cite{P}) that this choice
precisely ensures that the order $\geq$ on $\Delta^+$ is convex, in the sense of Definition \ref{def:convex}. Moreover, the
$q$-deformed root vectors satisfy the following $q$-analogue of relation~\eqref{eqn:comm intro 1}, where $\alpha, \beta$ and
$\alpha+\beta$ are any positive roots that satisfy $\alpha<\alpha+\beta<\beta$ as well as
the minimality property \eqref{eqn:minimal}:
\begin{equation}
\label{eqn:comm intro 2}
  [\bare_{\alpha} , \bare_{\beta}]_q = e_{\alpha} e_{\beta} - q^{(\alpha, \beta)} e_{\beta} e_{\alpha}
  \in \BZ[q,q^{-1}]^* \cdot e_{\alpha+\beta}
\end{equation}
where $(\cdot, \cdot)$ denotes the scalar product corresponding to the root system of type $\fg$. As in the Lie algebra case,
we conclude that the $q$-deformed root vectors can be defined (up to scalar multiple) as iterated $q$-commutators of
$e_i = e_{\alpha_i}$ (with $i\in I$), using the combinatorics of the root system and the chosen convex order on $\Delta^+$.

\medskip


\subsection{}
\label{sub:shuffle finite intro}

There is a well-known incarnation of $\uup$ due to Green~\cite{G}, Rosso~\cite{R1}, and Schauenburg~\cite{S}
in terms of quantum shuffles:
\begin{equation}
\label{eqn:shuffle intro}
  \uup \stackrel{\Phi}\hooklongrightarrow \CF \ = \bigoplus^{k\in \BN}_{i_1,\dots,i_k \in I} \BQ(q) \cdot [i_1 \dots i_k]
\end{equation}
where the right-hand side is endowed with the quantum shuffle product (see Definition~\ref{def:shuf finite}).
As shown by Lalonde-Ram in~\cite{LR}, there is a one-to-one correspondence between positive roots and
\underline{standard Lyndon} (or Shirshov) words in the alphabet~$I$:
\begin{equation}
\label{eqn:1-to-1 intro}
  \ell \colon \Delta^+ \stackrel{\sim}\longrightarrow \Big\{\text{standard Lyndon words}\Big\}
\end{equation}
To this end, we recall that a word in an ordered finite alphabet $I$ is called \emph{Lyndon} if it is lexicographically
smaller than all of its cyclic permutations (see Definition~\ref{def:lyndon}). These words naturally give rise
to a basis of the free Lie algebra generated by the alphabet $I$ through the standard bracketing (cf.~\eqref{eqn:bracketing lyndon}).
In \cite{LR}, a Gr\"{o}bner basis type approach was used to combinatorially describe a subset of all Lyndon words,
called \emph{standard Lyndon} words, that gives rise to a basis of a Lie algebra generated by $I$ (see Definition~\ref{def:standard}(b)).
Thus, in the context of~\eqref{eqn:1-to-1 intro}, the notion of standard Lyndon words intrinsically depends
on a fixed total order of the indexing set $I$ of simple roots. Furthermore,~\eqref{eqn:1-to-1 intro} gives rise to a total order on the positive roots:
\begin{equation}
\label{eqn:induces}
  \alpha < \beta \quad \Leftrightarrow \quad \ell(\alpha) < \ell(\beta) \text{ lexicographically}
\end{equation}
It was shown in~\cite{R2}, see~\cite[Proposition 26]{L}, that this total order is convex, and hence can be applied
to obtain root vectors
$e_\alpha \in \uup$ for any positive root $\alpha$, as in~\eqref{eqn:comm intro 2}. Moreover, \cite{L} shows that
the root vector $e_\alpha$ is uniquely characterized (up to a scalar multiple) by the property that $\Phi(e_\alpha)$
is an element of $\text{Im }\Phi$ whose leading order term $[i_1 \dots i_k]$ (in the lexicographic order) is precisely
$\ell(\alpha)$. We would also like to mention \cite{CHW} which contains alternative proofs of some of the results
of~\cite{L}, particularly leading into a generalization to quantum supergroups.

\medskip


\subsection{}
\label{sub:shuffle affine intro}

The motivation of the present paper is to extend the discussion of Subsection~\ref{sub:shuffle finite intro}
to affine root systems. This would yield a combinatorial description of PBW bases inside the positive half of
the Drinfeld-Jimbo affine quantum group. But there is an important problem with this program: the root spaces
are no longer one-dimensional in the affine case (because of the imaginary roots), which creates various technical
difficulties. We will therefore not take this route, and instead take an ``orthogonal" approach. We start from
Drinfeld's new presentation of quantum loop groups as:
$$
  \UU = \UUp \otimes \UUo \otimes \UUm
$$
where $\UUp$ is a $q$-deformation of the universal enveloping algebra of $\fn^+[t,t^{-1}]$.
The latter Lie algebra has the property that all its root spaces are one-dimensional,
so we are able to adapt many of the results mentioned in the previous Subsection.
To do so, we introduce the loop version $\hCF$ of the algebra $\CF$ in Sections~\ref{sub:affine shuffle}--\ref{sub:fix}:
  $$\hCF \ =
    \mathop{\mathop{\bigoplus_{i_1,\dots,i_k \in I}}_{d_1,\dots,d_k \in \BZ}}^{k\in \BN}
    \BQ(q) \cdot \left[ i_1^{(d_1)} \dots\, i_k^{(d_k)} \right]$$
The algebra structure on $\hCF$ is defined by the following shuffle product:
\begin{equation*}
\begin{split}
  & \left[i^{(d_1)}_1 \dots\, i^{(d_k)}_k \right] * \left[j^{(e_1)}_1 \dots\, j^{(e_l)}_l \right] =\\
  & \quad \mathop{\sum_{\{1, \dots ,k+l\} = A \sqcup B}}_{|A| = k, |B| = l}
     \left( \mathop{\sum_{\pi_1 + \dots +\pi_{k+l} = 0}}_{\pi_1, \dots, \pi_{k+l} \in \BZ}
     \gamma_{A,B,\pi_1, \dots ,\pi_{k+l}} \cdot
     \left[s^{(t_1+\pi_1)}_1 \dots\, s^{(t_{k+l} + \pi_{k+l})}_{k+l} \right] \right)
\end{split}
\end{equation*}
where if $A = \{a_1< \dots <a_k\}$ and $B = \{b_1 < \dots < b_l\}$, we write:
\begin{equation*}
  s_c = \begin{cases} i_\bullet &\text{if } c = a_\bullet \\ j_\bullet &\text{if } c = b_\bullet \end{cases}, \qquad
  t_c = \begin{cases} d_\bullet &\text{if } c = a_\bullet \\ e_\bullet &\text{if } c = b_\bullet \end{cases}
\end{equation*}
and the coefficients $\gamma_{A,B,\pi_1, \dots,\pi_{k+l}}$ are explicitly given in~\eqref{eqn:power series}.
In fact, one actually needs to work with an appropriate completion above, see~(\ref{eqn:completion fix},~\ref{eqn:infinite sums}),
in order for the shuffle product to be well-defined (as it contains infinitely many summands).

\medskip

\begin{theorem}
\label{thm:main 1}
There exists an injective algebra homomorphism:
$$
  \UUp \stackrel{\wPhi}\hooklongrightarrow \hCF
$$
Fix a total order of $I$, which induces the following total order on the set $\{i^{(d)}\}_{i\in I}^{d\in \BZ}$:
\begin{equation}
\label{eqn:lex affine}
  i^{(d)} < j^{(e)} \quad \text{if} \quad
  \begin{cases} d>e \\ \text{ or } \\ d = e \text{ and } i<j \end{cases}
\end{equation}
This induces the lexicographic order on the words $[i_1^{(d_1)} \dots\, i_k^{(d_k)}]$ with respect to which we may define
the notion of \underline{\slaws} by analogy with~\cite{LR}
(see Subsections~\ref{sub:affine standard}--\ref{sub:standard stability} for details).
Then, there exists a 1-to-1 correspondence:
\begin{equation}
\label{eqn:1-to-1 intro affine}
  \ell \colon \Delta^+ \times \BZ \stackrel{\sim}\longrightarrow \Big\{\slaws\Big\}
\end{equation}
The lexicographic order on the right-hand side induces a convex order on the left-hand side,
with respect to which one can define elements:
\begin{equation}
\label{eqn:root vectors intro 2}
  e_{\ell(\alpha,d)} \in \UUp
\end{equation}
for all $(\alpha, d) \in \Delta^+ \times \BZ$. We have the following analogue of the PBW theorem:
\begin{equation}
\label{eqn:pbw intro loop}
  \UUp \ = \ \bigoplus^{k \in \BN}_{\ell_1 \geq \dots \geq \ell_k \text{ standard Lyndon loop words}}
  \BQ(q)\cdot e_{\ell_1} \dots e_{\ell_k}
\end{equation}
There are also analogues of the constructions above with $+ \leftrightarrow -$ and $e \leftrightarrow f$.
\end{theorem}

\medskip

%

\noindent
By analogy with the previous paragraph, the total order on $\Delta^+ \times \BZ$ given by:
\begin{equation}
\label{eqn:induces affine}
  (\alpha,d) < (\beta,e) \quad \Leftrightarrow \quad \ell(\alpha,-d) < \ell(\beta,-e) \text{ lexicographically}
\end{equation}
is convex; this fact will be proved in Proposition \ref{prop:convex loop}. As such, this order comes from a certain reduced
word in the affine Weyl group associated to $\fg$ (= the Coxeter group associated to $\widehat{\fg}$), in accordance with
Theorem~\ref{thm:weyl to lyndon}. Therefore, the root vectors~\eqref{eqn:root vectors intro 2} exactly match (up to constants)
the classical construction of~\cite{B,D,LSS,Lu}, once we pass it through the ``affine to loop'' isomorphism of
Theorem~\ref{thm:two presentations}.

\medskip

\noindent
We note that our notion of standard Lyndon loop words, as well as the order~(\ref{eqn:induces affine}) on $\Delta^+ \times \BZ$,
are not the same as the similarly named notions of~\cite{HRZ}. In general, our order between $(\alpha, d)$ and $(\beta, e)$ is not
determined by the order between $\alpha$ and $\beta$, as was the case in \loccit

\medskip


\subsection{}
\label{sub:FO-shuffle intro}

There exists another shuffle algebra construction in the theory of quantum loop groups, with its origins
in the elliptic algebras defined by Feigin-Odesskii~\cite{FO}. In the setting at hand, the construction is
due to Enriquez~\cite{E1}, who constructed an algebra homomorphism:
$$
  \UUp \stackrel{\Upsilon}\longrightarrow \CA^+ \subset
  \bigoplus_{\bk = (k_i)_{i \in I} \in \nn} \BQ(q)(\dots,z_{i1},\dots,z_{ik_i},\dots)^{\sym}
$$
where the direct sum is made into an algebra using the multiplication~\eqref{eqn:mult}
(we refer the reader to Definition~\ref{def:shuf} for the precise definition of the inclusion $\subset$ above
in terms of \underline{pole} and \underline{wheel} conditions). In the present paper, we prove that:

\medskip

\begin{theorem}
\label{thm:main 2}
The map $\Upsilon$ is an isomorphism.
\end{theorem}

\medskip

\noindent
In type $A_n$, this result follows immediately from the type $\widehat{A}_n$ case proved in~\cite{N2}
(see also \cite{Ts} for the rational, super, and two-parameter generalizations), but the methods of~\loccit\
are difficult to generalize to our current setup. Instead, we use the framework of the preceding Subsection
to prove Theorem~\ref{thm:main 2}. To this end, in Subsection~\ref{sub:iota}, we construct an algebra homomorphism:
$$
  \CA^+ \stackrel{\iota}\hooklongrightarrow \hCF
$$
given explicitly by~\eqref{eqn:def iota}, such that
$$
  \wPhi = \iota \circ \Upsilon
$$
according to~\eqref{eqn:composition}. Extending all algebras by adding Cartan elements, we obtain:
$$
  \UUg \stackrel{\Upsilon}\longrightarrow \CA^\geq \stackrel{\iota}\hooklongrightarrow \hCFext
$$
see~(\ref{eqn:upsilon geq},~\ref{eqn:iota ext}), which are bialgebra homomorphisms by Propositions~\ref{prop:upsilon geq},~\ref{prop:shuf hom}.
Furthermore, in Proposition~\ref{prop:pair shuf}, we construct a bialgebra pairing
  $$\CA^\geq \otimes \UUl \longrightarrow \BQ(q)$$
which is non-degenerate in the first argument by Proposition~\ref{prop:non-degenerate shuf}.
To establish the surjectivity of the embedding $\Upsilon$,
we filter $\UUp$ by $\UUp_{\leq w}$ and $\CA$ by $\CA_w$, so that
  $$\Upsilon(\UUp_{\leq w})\subset \CA^+_{\leq w} \quad \mathrm{for\ any\ loop\ word}\quad w$$
Using the non-degeneracy of the aforementioned pairing, we then obtain:
\begin{multline*}
  \# \Big\{\text{good loop words } \leq w \Big\} = \dim \UUp_{\leq w} \leq \\
  \dim \CA^+_{\leq w} \leq \dim \UUm^{\leq w} = \# \Big\{\text{standard loop words } \leq w \Big\}
\end{multline*}
with the dimension count understood in the sense of restriction to each $Q^+ \times \BZ$-graded component.
Evoking Proposition~\ref{prop:standard is good}, we then conclude that both inequalities $\leq$ above must be equalities.
This implies the surjectivity of $\Upsilon$ as $\CA^+=\cup_w \CA^+_{\leq w}$.

\medskip

\noindent
The homomorphism $\iota$ can be construed as connecting the two (a priori) different instances of
shuffle algebras that appear in the study of quantum loop groups.

\medskip


\subsection{}
\label{sub:literature overview intro}

Many of the things discussed in the present paper are connected to existing literature. Besides the strong inspiration
from the finite type case studied in~\cite{LR, L, R2} that we already mentioned, we encounter the following concepts:

\medskip

\begin{itemize}[leftmargin=*]

\item
Theorems on convex PBW bases of affine quantum groups~\cite{B,D,KT,LSS} inspired by the constructions of~\cite{KiRe,LS,R0}
for quantum groups of finite type.

\medskip

\item
Shuffle algebra incarnations of quantum groups~\cite{G,R1,S}, which we generalize to the case of quantum loop groups,
thus obtaining the algebra $\hCF$ that features in Theorem \ref{thm:main 1}.

\medskip

\item
Feigin-Odesskii shuffle algebras~\cite{FO} and their trigonometric versions~\cite{E1}, which have recently had numerous
applications to mathematical physics, cf.~survey~\cite{F}.

\end{itemize}

\noindent 
The combinatorics of Lyndon words for finite types was connected with representations of KLR algebras in \cite{KR}. It would be very interesting if the combinatorics of Lyndon loop words developed herein had such an interpretation, although this is not at all clear. A priori, the setting of \emph{loc. cit.} generalizes to affine types, which differs from our point of view by the ``affine to loop" isomorphism of Theorem~\ref{thm:two presentations}.

\medskip


\subsection{}
\label{sub:paper structure intro}

The structure of the present paper is the following:

\medskip

\begin{itemize}[leftmargin=*]

\item
In Section~\ref{sec:lie}, we study the Lie algebras $\fg$ and $L\fg$, recall the notion of standard Lyndon words
for the former, and extend this notion to the latter.

\medskip

\item
In Section~\ref{sec:weyl lyndon}, we show that the order~\eqref{eqn:induces affine} on $\Delta^+ \times \BZ$ corresponds to
a certain reduced decomposition in the extended affine Weyl group of $\fg$.

\medskip

\item
In Section~\ref{sec:quantum}, we study the quantum groups $\uu$ and $\UU$, and their PBW bases defined
with respect to standard Lyndon (loop) words. We construct the objects featuring in Theorem~\ref{thm:main 1}.

\medskip

\item
In Section~\ref{sec:loop affine}, we wrap up the proof of Theorem~\ref{thm:main 1} by tying it in with
the well-known construction of PBW bases of affine quantum groups (\cite{B,D}). Many of the results included
in this Section can be found in \cite{EKP}, especially Proposition~\ref{prop:quarter}, but our treatment yields
an alternative proof of some results of \emph{loc. cit.}

\medskip

\item
In Section~\ref{sec:fo}, we recall the trigonometric degeneration of the Feigin-Odesskii shuffle algebra,
and prove Theorem~\ref{thm:main 2} using the results of Theorem~\ref{thm:main 1}.

\medskip

\item
In the Appendix, we give explicit combinatorial data pertaining to standard Lyndon loop words for all (untwisted)
classical types, corresponding to an order of the simple roots of our choice. For any other order of the simple roots,
as well as for the exceptional types, computer code performing these tasks in reasonable time is available on demand
from the authors.

\end{itemize}

\medskip


\subsection{}
\label{sub:Achowledgement}

We would like to thank Pavel Etingof and Boris Feigin for their help and numerous stimulating discussions over the years. We also thank Alexander Kleshchev and Weiqiang Wang for their interesting remarks on a draft of the present paper. We are indebted to the anonymous referee for useful suggestions on the exposition.

A.N.\ would like to gratefully acknowledge NSF grants DMS-$1760264$ and DMS-$1845034$, as well as support from the Alfred P.\ Sloan Foundation and the MIT Research Support Committee.
A.T.\ would like to gratefully acknowledge NSF grants DMS-$2037602$ and DMS-$2302661$.

\medskip


\section{Lie algebras and Lyndon words}
\label{sec:lie}

It is a classical result that the free Lie algebra on a set of generators $\{e_i\}_{i\in I}$ has a basis
indexed by Lyndon words (see Definition~\ref{def:lyndon}) in the alphabet $I$. If we impose a certain collection
of relations among the $e_i$'s, then~\cite{LR} showed that a basis of the resulting Lie algebra is given by
standard Lyndon words (see Definition~\ref{def:standard}), and determined the latter in the particular case of
the maximal nilpotent subalgebra of a simple Lie algebra. In the present Section, we will extend
the treatment of~\loccit\ to the situation of loops into simple Lie algebras.

\medskip
\noindent
We start with the exposition of the relevant classical results in Subsections~\ref{sub:root system}--\ref{sub:Leclerc's convexity}.


\subsection{}
\label{sub:root system}

Let us consider a root system of finite type:
$$
  \Delta^+ \sqcup \Delta^- \subset Q
$$
(where $Q$ denotes the root lattice) associated to the symmetric~pairing:
$$
  (\cdot, \cdot)\colon Q \otimes Q \longrightarrow \BZ
$$
Let $\{\alpha_i\}_{i\in I}\subset \Delta^+$ denote a choice of \underline{simple roots}. The Cartan matrix
$(a_{ij})_{i,j\in I}$ and the symmetrized Cartan matrix $(d_{ij})_{i,j\in I}$ of this root system are:
\begin{equation}
\label{eqn:cartan matrix}
  a_{ij} = \frac {2(\bs_i, \bs_j)}{(\bs_i, \bs_i)}
    \qquad \mathrm{and} \qquad
  d_{ij} = (\bs_i,\bs_j)
\end{equation}

\medskip

\begin{definition}
\label{def:finite lie}
To the root system above, one associates the Lie algebra:
$$
  \fg = \BQ \Big \langle e_i, f_i, h_i \Big \rangle_{i \in I} \Big/
  \text{relations \eqref{eqn:rel 1 finite lie}--\eqref{eqn:rel 3 finite lie}}
$$
where we impose the following relations for all $i,j \in I$:
\begin{equation}
\label{eqn:rel 1 finite lie}
  \underbrace{[e_i,[e_i,[\dots,[e_i,e_j]\dots]]]}_{1-a_{ij} \text{ Lie brackets}}\, =\, 0, \quad \text{if }i \neq j
\end{equation}
\begin{equation}
\label{eqn:rel 2 finite lie}
  [h_j,e_i] = d_{ji} e_i, \qquad \qquad [h_i,h_j] = 0
\end{equation}
as well as the opposite relations with $e$'s replaced by $f$'s, and finally the relation:
\begin{equation}
\label{eqn:rel 3 finite lie}
  [e_i, f_j] = \delta_i^j h_i
\end{equation}
\end{definition}

\medskip

\noindent
We will consider the triangular decomposition~\eqref{eqn:decomp intro 1}, where $\fn^+$, $\fh$, $\fn^-$
are the Lie subalgebras of $\fg$ generated by the $e_i$, $h_i$, $f_i$, respectively. We will write:
$$
  Q^\pm \subset Q
$$
for the monoids generated by $\pm \alpha_i$. The Lie algebra $\fg$ is graded by $Q$, if we let:
$$
  \deg e_i = \alpha_i,\quad \deg h_i = 0, \quad \deg f_i = -\alpha_i
$$
The subalgebras $\fn^\pm$ are graded by $Q^\pm$ accordingly.

\medskip


\subsection{}
\label{sub:finite words}

We will now recall the construction of~\cite{LR}, which describes positive roots in terms of the combinatorics of words:
\begin{equation}
\label{eqn:finite word}
  \left[ i_1 \dots i_k \right]
\end{equation}
for various $i_1,\dots,i_k \in I$. Let us fix a total order on the set $I$ of simple roots,
which induces the following total lexicographic order on the set of all words:
$$
  [i_1 \dots i_k] < [j_1 \dots j_l] \quad \text{if }
  \begin{cases}
    i_1=j_1, \dots, i_a=j_a, i_{a+1} < j_{a+1} \text{ for some } a \geq 0 \\
      \text{\ or} \\
    i_1=j_1, \dots, i_k=j_k \text{ and } k < l
  \end{cases}
$$

\medskip

\begin{definition}
\label{def:lyndon}
A word $\ell=[i_1\dots i_k]$ is called \underline{Lyndon} (such words were also studied independently by Shirshov)
if it is smaller than all of its cyclic permutations:
$$
  [i_1 \dots i_{a-1} i_a \dots i_k] < [i_a \dots i_k i_1 \dots i_{a-1}]
$$
for all $a \in \{2,\dots,k\}$.
\end{definition}

\medskip

\noindent
The following is an elementary exercise, that we leave to the interested reader.

\medskip

\begin{claim}
\label{claim:lyndon}
If $\ell_1 < \ell_2$ are Lyndon, then $\ell_1\ell_2$ is also Lyndon, and so $\ell_1 \ell_2 < \ell_2 \ell_1$.
\end{claim}

\medskip

\noindent
Given a word $w = [i_1 \dots i_k]$, the subwords:
$$
  w_{a|} =  [i_1 \dots i_a] \qquad \text{and} \qquad w_{|a} = [i_{k-a+1} \dots i_k]
$$
with $0\leq a\leq k$ will be called a prefix and a suffix of $w$, respectively.
Such a prefix or a suffix is called proper if $a \notin \{0,k\}$. It is straightforward to show that
a word $w$ is Lyndon iff it is smaller than all of its proper suffixes, i.e.\ $w < w_{|a}$ for all $0<a<k$.

\medskip

\begin{proposition}
\label{prop:costandard factorization}
(see \cite[\S1]{LR} for a survey)
Any Lyndon word $\ell$ has a factorization:
\begin{equation}
\label{eqn:costandard factorization}
  \ell = \ell_1 \ell_2
\end{equation}
defined by the property that $\ell_2$ is the longest proper suffix of $\ell$ which is also a Lyndon word.
Under these circumstances, $\ell_1$ is also a Lyndon word.
\end{proposition}

\medskip

\begin{proposition}
\label{prop:canonical factorization}
Any word $w$ has a \underline{canonical factorization} as a concatenation:
\begin{equation}
\label{eqn:canonical factorization}
  w = \ell_1 \dots \ell_k
\end{equation}
where $\ell_1 \geq \dots \geq \ell_k$ are all Lyndon words.
\end{proposition}

\medskip


\subsection{}
\label{sub:lyndon}

For any word $w = [i_1 \dots i_k]$, we define:
\begin{equation}
\label{eqn:the word}
  _we = e_{i_1} \dots e_{i_k} \in U(\fn^+)
\end{equation}

\medskip

\noindent
On the other hand, Propositions~\ref{prop:costandard factorization} and~\ref{prop:canonical factorization}
yield the following construction.

\medskip

\begin{definition}
\label{def:bracketing}
For any word $w$, define $e_w \in U(\fn^+)$ inductively by $e_{[i]} = e_i$ and:
\begin{equation}
\label{eqn:bracketing lyndon}
  e_{\ell} = \left[ e_{\ell_1}, e_{\ell_2} \right] \in \fn^+
\end{equation}
if $\ell$ is a Lyndon word with factorization~\eqref{eqn:costandard factorization}, and:
\begin{equation}
\label{eqn:bracketing arbitrary}
  e_w = e_{\ell_1} \dots e_{\ell_k} \in U(\fn^+)
\end{equation}
if $w$ is an arbitrary word with the canonical factorization $\ell_1 \dots \ell_k$, as in \eqref{eqn:canonical factorization}.
\end{definition}

\medskip

\begin{remark}
\label{rem:roots}
Because $[e_\alpha, e_\beta] \in \mathbb{Q}^* \cdot e_{\alpha+\beta}$ for all positive roots $\alpha, \beta$
such that $\alpha + \beta$ is also a root (\cite[Proposition 8.4(d)]{H}), then choosing a different factorization
\eqref{eqn:costandard factorization} for various Lyndon words will in practice produce bracketings
\eqref{eqn:bracketing lyndon} which are non-zero multiples of each other. Thus various choices will
simply lead to PBW bases \eqref{eqn:pbw intro 1} which are renormalizations of each other.
\end{remark}

\medskip

\noindent
It is well-known that the elements~\eqref{eqn:the word} and~\eqref{eqn:bracketing arbitrary} both give rise
to bases of $U(\fn^+)$, and indeed are connected by the following triangularity property:
\begin{equation}
\label{eqn:upper}
  e_w = \sum_{v \geq w} c^v_w \cdot {_ve}
\end{equation}
for various integer coefficients $c^v_w$ such that $c_w^w = 1$.

\medskip


\subsection{}
\label{sub:standard words}

If $\fn^+$ were a free Lie algebra, then it would have a basis given by the elements~\eqref{eqn:bracketing lyndon},
as $\ell$ goes over all Lyndon words (and similarly, $U(\fn^+)$ would have a basis given by the
elements~\eqref{eqn:bracketing arbitrary} as $w$ goes over all words). But since we have to contend with
the relations~\eqref{eqn:rel 1 finite lie} between the generators $e_i \in \fn^+$, we must restrict
the set of Lyndon words which appear. The following definition is due to \cite{LR}.

\medskip

\begin{definition}
\label{def:standard}
(a) A word $w$ is called \underline{standard} if $_we$ cannot be expressed as a linear combination of
$_ve$ for various $v>w$, with $_we$ as in~\eqref{eqn:the word}.

\medskip

\noindent
(b) A Lyndon word $\ell$ is called \underline{standard Lyndon} if $e_{\ell}$ cannot be expressed as a linear
combination of $e_m$ for various Lyndon words $m>\ell$, with $e_{\ell}$ as in~\eqref{eqn:bracketing lyndon}.
\end{definition}

\medskip

\noindent
The following Proposition is non-trivial, and it justifies the above terminology.

\medskip

\begin{proposition}
\label{prop:standard}
(\cite{LR})
A Lyndon word is standard iff it is standard Lyndon.
\end{proposition}

\medskip

\noindent
According to~\cite[\S2.1]{LR}, $\fn^+$ has a basis consisting of the $e_{\ell}$'s, as $\ell$ goes over
all standard Lyndon words. Since the Lie algebra $\fn^+$ is $Q^+$-graded by $\deg e_i = \bs_i$, it is
natural to extend this grading to words as follows:
\begin{equation}
\label{eqn:degree of a word}
  \deg [i_1 \dots i_k] = \bs_{i_1} + \dots + \bs_{i_k}
\end{equation}
Because of the decomposition~\eqref{eqn:root vectors intro} of $\fn^+$, and the fact that the basis vectors
$\bare_\alpha \in \fn^+$ all live in distinct degrees $\alpha \in Q^+$, we conclude that there exists a bijection:
\begin{equation}
\label{eqn:associated word}
  \ell \colon \Delta^+ \ \stackrel{\sim}\longrightarrow \ \Big\{\text{standard Lyndon words}\Big\}
\end{equation}
such that $\deg \ell(\alpha) = \alpha$, for all $\alpha \in \Delta^+$. The interested reader may find
some examples of the bijection~\eqref{eqn:associated word} for the classical finite types in the Appendix.

\medskip


\subsection{}
\label{sub:iterated algorithm}

The following explicit description of the bijection \eqref{eqn:associated word} was proved in \cite[Proposition 25]{L},
and allows one to inductively construct the bijection $\ell$:
\begin{equation}
\label{eqn:inductively}
  \ell(\alpha) \ =
  \mathop{\max_{\gamma_1+\gamma_2 = \alpha, \ \gamma_k \in \Delta^+}}_{\ell(\gamma_1) < \ell(\gamma_2)}
  \Big \{ \text{concatenation } \ell(\gamma_1)\ell(\gamma_2) \Big\}
\end{equation}

\medskip

\noindent
We also have the following simple property of standard words.

\medskip

\begin{proposition}
\label{prop:factor standard}
(\cite[\S2.4]{LR})
Any subword of a standard word is standard.
\end{proposition}

\medskip

\noindent
Combining Propositions~\ref{prop:canonical factorization},~\ref{prop:standard},~\ref{prop:factor standard},
we conclude that any standard word can be uniquely written in the form~\eqref{eqn:canonical factorization},
where $\ell_1 \geq \dots \geq \ell_k$ are all standard Lyndon words. The converse also holds
(by a dimension count argument, see~\cite[\S2.8]{LR}).

\medskip

\begin{proposition}
\label{prop:stand via Lyndonstand}
(\cite{LR})
A word $w$ is standard if and only if it can be written (uniquely) as $w=\ell_1 \dots \ell_k$,
where $\ell_1\geq \dots \geq \ell_k$ are standard Lyndon words.
\end{proposition}

\medskip

\begin{remark}\label{rem:generalization}
The results of Propositions~\ref{prop:standard},~\ref{prop:factor standard},~\ref{prop:stand via Lyndonstand} hold
for any finite dimensional Lie algebra, according to~\cite{LR}. In particular, we shall be applying them to Lie algebras
$L^{(s)}\fn^+$ of~\eqref{eqn:loop filtration}, generalizing $L^{(0)}\fn^+\simeq \fn^+$.
\end{remark}

\medskip

\noindent
Thus we obtain the following reformulation of \eqref{eqn:pbw intro 1}:
\begin{equation}
\label{eqn:pbw lie}
  U(\fn^+) \ =
  \bigoplus^{k\in \BN}_{\ell_1 \geq \dots \geq \ell_k \text{ standard Lyndon words}}
    \BQ \cdot e_{\ell_1} \dots e_{\ell_k}
\end{equation}
By the triangularity property \eqref{eqn:upper}, we could also get a basis of $U(\fn^+)$ by replacing
$e_w = e_{\ell_1} \dots e_{\ell_k}$ in \eqref{eqn:pbw lie} by $_we$, for any standard word $w$.

\medskip


\subsection{}
\label{sub:Leclerc's convexity}

The bijection~\eqref{eqn:associated word} yields a total order on the set of positive roots $\Delta^+$, induced by
the lexicographic order of standard Lyndon words, see~(\ref{eqn:induces}). As observed in~\cite{L, R2}, this order
is convex, in the following sense.

\medskip

\begin{definition}
\label{def:convex}
A total order on the set of positive roots $\Delta^+$ is called \underline{convex} if:
\begin{equation}
\label{eqn:convex}
  \alpha < \alpha+\beta < \beta
\end{equation}
for all $\alpha < \beta \in \Delta^+$ such that $\alpha+\beta$ is also a root.
\end{definition}

\medskip

\noindent
It is well-known (\cite{P}) that convex orders of the positive roots are in 1-to-1 correspondence with reduced decompositions
of the longest element of the Weyl group associated to our root system. We will consider this issue, and its affine version,
in more detail in Section \ref{sec:weyl lyndon}.

\medskip

\begin{proposition}
\label{prop:finite convexity}
(\cite[Proposition 26]{L})
The order~(\ref{eqn:induces}) on $\Delta^+$ is convex.
\end{proposition}

\medskip

\noindent
We will prove the loop version of the Proposition above in Proposition \ref{prop:convex loop}.

\medskip


\subsection{}
\label{sub:loop algebra}

We will now extend the description above to the Lie algebra of loops into $\fg$:
$$
  L\fg = \fg[t,t^{-1}]=\fg \otimes_{\BQ} \BQ [t,t^{-1}]
$$
where the Lie bracket is simply given by:
\begin{equation}
\label{eqn:lie loop}
  [x \otimes t^m, y \otimes t^n] = [x,y] \otimes t^{m+n}
\end{equation}
for all $x,y \in \fg$ and $m,n \in \BZ$. The triangular decomposition~\eqref{eqn:decomp intro 1} extends to a similar
decomposition at the loop level, and our goal is to describe $L\fn^+$ along the lines of Subsections
\ref{sub:standard words}--\ref{sub:iterated algorithm}. To this end, we think of $L\fn^+$ as being generated by:
$$
  e_i^{(d)} = \bare_i \otimes t^d
$$
$\forall\, i \in I, d \in \BZ$. Associate to $e_i^{(d)}$ the \underline{letter} $i^{(d)}$, and call $d$ the
\underline{exponent} of $i^{(d)}$. We fix a total order on $I$, which induces the total order \eqref{eqn:lex affine} on
the letters $\{i^{(d)}\}_{i\in I}^{d\in \BZ}$. Any word in these letters will be called a \underline{loop word}:
\begin{equation}
\label{eqn:loop words}
  \left[ i_1^{(d_1)} \dots\, i_k^{(d_k)} \right]
\end{equation}
We have the total lexicographic order on loop words \eqref{eqn:loop words} induced by~\eqref{eqn:lex affine}.
All the results of Subsection~\ref{sub:finite words} continue to hold in the present setup, so we have a notion
of Lyndon loop words. Since $L\fn^+$ is $Q^+ \times \BZ$-graded by:
$$
  \deg e_i^{(d)} = (\bs_i,d)
$$
it makes sense to extend this grading to loop words as follows:
\begin{equation}
\label{eqn:degree of loop word}
  \deg \left[i_1^{(d_1)} \dots\, i_k^{(d_k)} \right] = (\bs_{i_1}+ \dots +\bs_{i_k}, d_1+ \dots +d_k)
\end{equation}
The obvious generalization of~\eqref{eqn:root vectors intro} is:
\begin{equation}
\label{eqn:decomposition loop}
  L\fn^+ = \bigoplus_{\alpha \in \Delta^+} \bigoplus_{d \in \BZ} \BQ \cdot e_\alpha^{(d)}
\end{equation}
with $e^{(d)}_\alpha=e_\alpha\otimes t^d$.
If $\deg x = (\alpha,d) \in Q^+ \times \BZ$, then we will use the notation:
\begin{equation}
\label{eq:hor and vert}
  \hdeg \ x = \alpha \qquad \text{and} \qquad \vdeg x = d
\end{equation}
and call these two notions the \underline{horizontal} and the \underline{vertical} degree, respectively.
While obviously infinite-dimensional, $L\fn^+$ still has one-dimensional $Q^+ \times \BZ$-graded pieces,
which is essential for the treatment of~\cite{LR} to carry through.

\medskip
\noindent
The aim of Subsections~\ref{sub:affine standard}--\ref{sub:standard stability} is to obtain a notion of
standard (Lyndon) loop words. This is a non-trivial task as the alphabet $\{i^{(d)}\}_{i\in I}^{d\in \BZ}$ is infinite.
To do so, we consider a filtration by finitely generated Lie algebras $L^{(s)}\fn^+$ of~\eqref{eqn:loop filtration},
corresponding to the finite alphabets $\{e_i^{(d)}| i\in I, -s\leq d\leq s\}$. We then establish some basic properties
of the corresponding standard Lyndon loop words for $L^{(s)}\fn^+$ in
Propositions~\ref{prop:classification},~\ref{prop:l1},~\ref{prop:l2},~\ref{prop:coherent}.
The latter result implies that the notion of ``standard Lyndon loop word'' does not depend on
the particular $L^{(s)}\fn^+$ with respect to which it is defined, thus establishing the
loop analogue~\eqref{eqn:associated word loop} of the bijection~\eqref{eqn:associated word}.


\subsection{}
\label{sub:affine standard}

We now wish to extend Definition~\ref{def:standard} in order to obtain a notion of standard (Lyndon) loop words,
but here we must be careful, because the alphabet $\{i^{(d)}\}_{i\in I}^{d\in \BZ}$ is infinite. In particular,
the key assumption ``for any word $v$, there are only finitely many words $u$ of the same length and $>v$ in the
lexicographical order'' of~\cite[\S2]{LR} clearly does not hold. To deal with this issue, we consider the increasing filtration:
$$
  L\fn^+ = \bigcup_{s=0}^\infty L^{(s)}\fn^+
$$
defined with respect to the finite-dimensional Lie subalgebras:
\begin{equation}
\label{eqn:loop filtration}
  L\fn^+ \supset L^{(s)}\fn^+ =
  \bigoplus_{\alpha \in \Delta^+} \bigoplus_{d = -s|\alpha|}^{s|\alpha|} \BQ \cdot e_\alpha^{(d)}
\end{equation}
where $|\alpha|$ denotes the \underline{height} of a root, i.e.
$$
  |\alpha| = \sum_{i\in I} k_i
$$
if $\alpha = \sum_{i \in I} k_i \alpha_i$.

\medskip

\noindent
As a Lie algebra, $L^{(s)}\fn^+$ is generated by $\{e_i^{(d)}| i\in I, -s\leq d\leq s\}$.
Therefore, we may apply Definition~\ref{def:standard} to yield a notion of standard (Lyndon) loop words
with respect to the finite-dimensional Lie algebras $L^{(s)}\fn^+$, where the corresponding words
will only be made up of the symbols $i^{(d)}$ with $i\in I, d \in \{-s,\dots,s\}$.

\medskip

\begin{proposition}
\label{prop:classification}
There exists a bijection:
\begin{equation}
\label{eqn:bijection lyndon}
  \ell \colon
  \Big\{(\alpha,d) \in \Delta^+ \times \BZ, |d| \leq s|\alpha| \Big\} \stackrel{\sim}\longrightarrow
  \Big\{\text{standard Lyndon loop words for } L^{(s)}\fn^+ \Big\}
\end{equation}
explicitly determined by $\ell(\alpha_i,d)=\left[i^{(d)}\right]$ and the following property:
\begin{equation}
\label{eqn:property lyndon}
  \ell(\alpha,d) \ =
  \mathop{\mathop{\max_{(\gamma_1,d_1)+(\gamma_2,d_2) = (\alpha,d)}}_
  {\gamma_k \in \Delta^+, \ |d_k| \leq s |\gamma_k|}}_{\ell(\gamma_1,d_1) < \ell(\gamma_2,d_2)}
  \Big \{ \text{concatenation } \ell(\gamma_1,d_1)\ell(\gamma_2,d_2) \Big\}
\end{equation}
\end{proposition}

\medskip

\noindent
In view of Proposition~\ref{prop:stand via Lyndonstand} (see Remark~\ref{rem:generalization}),
this also gives a parametrization of standard loop words for $L^{(s)}\fn^+$.
We note that both the property \eqref{eqn:property lyndon}, as well as the main idea of the
subsequent proof, are direct adaptations of the analogous results in \cite{L} (cf.~\eqref{eqn:inductively}).

\medskip

\begin{proof}[Proof of Proposition~\ref{prop:classification}]
Because the root spaces of $L^{(s)} \fn^+$ are one-dimensional, as in \eqref{eqn:loop filtration}, then
for any Lyndon loop word $\ell$ of degree $(\alpha,d)\in Q^+\times \BZ$ with $|d| \leq s|\alpha|$, we have:
\begin{equation}
\label{eqn:lo}
  e_\ell \in \BQ \cdot e_{\alpha}^{(d)}
\end{equation}
The right-hand side is $0$ if $\alpha \notin \Delta^+$. By Definition~\ref{def:standard}(b), a word $\ell$
is standard Lyndon if and only if it is the maximal Lyndon word of its given degree, with the property that
$e_\ell \neq 0$. Together with the fact~\cite[\S2.1]{LR} that $\{e_\ell|\ell-\text{standard Lyndon}\}$ is
a basis of $L^{(s)}\fn^+$, this establishes the existence of a bijection \eqref{eqn:bijection lyndon}.

\medskip

\noindent
Let us now prove that this bijection takes the form \eqref{eqn:property lyndon}. Consider any $\gamma_1, \gamma_2 \in \Delta^+$
such that $\gamma_1+\gamma_2 \in \Delta^+$, and any integers $d_1$, $d_2$ such that $|d_k| \leq s |\gamma_k|$ for all
$k \in \{1,2\}$. Let us write $\ell_k = \ell(\gamma_k,d_k)$ for all $k \in \{1,2\}$ and $\ell = \ell (\gamma_1+\gamma_2,d_1+d_2)$;
we may assume without loss of generality that $\ell_1 < \ell_2$. We have:
\begin{equation}
\label{eqn:fi}
  e_{\ell_k} = \sum_{v_k \geq \ell_k} c^{v_k}_{\ell_k} \cdot {_{v_k}e}
\end{equation}
$\forall\, k \in \{1,2\}$, due to property \eqref{eqn:upper} (which holds in $L^{(s)}\fn^+$ as it did in $\fn^+$). Thus:
\begin{equation}
\label{eqn:hi}
  e_{\ell_1} e_{\ell_2} \, = \sum_{v \geq \ell_1\ell_2} x_v \cdot {_ve}
\end{equation}
for various coefficients $x_v$.\footnote{Here we are using the fact that if $v_1 \geq \ell_1$ and $v_2 \geq \ell_2$,
then $v_1v_2 \geq \ell_1\ell_2$; this fact is not true for arbitrary words $v_1$ and $v_2$, because we could have
$v_1 = \ell_1 u$ for some word $u < \ell_2$. However, such counterexamples are not allowed because the words $v_k$
which appear in \eqref{eqn:fi} have the same number of letters as $\ell_k$, for degree reasons.}
As a consequence of Claim \ref{claim:lyndon}, we have an analogue of formula \eqref{eqn:hi} when the indices 1 and 2
are swapped in the left-hand side. Hence we obtain the following formula for the commutator:
\begin{equation}
\label{eqn:lynd comm 0}
  [e_{\ell_1}, e_{\ell_2}] \, = \sum_{v \geq \ell_1\ell_2} y_v \cdot {_ve}
\end{equation}
for various coefficients $y_v$. Furthermore, we may restrict the sum above to standard $v$'s since, by the very definition
of this notion, any $_ve$ can be inductively written as a linear combination of $_ue$'s for standard $u \geq v$ (this uses
the fact that there exist finitely many words of any given degree, as we use a finite alphabet
$\{i^{(d)}\}_{i\in I}^{-s\leq d\leq s}$). By this very same reason, we may restrict the right-hand side of \eqref{eqn:upper}
to standard $v$'s, and conclude that $\{e_w|w - \text{standard}\}$ yield a basis which is upper triangular in terms of
the basis $\{_we|w - \text{standard}\}$. With this in mind, \eqref{eqn:lynd comm 0} implies:
\begin{equation}
\label{eqn:lynd comm 1}
  [e_{\ell_1}, e_{\ell_2}] \ = \mathop{\sum_{v \geq \ell_1\ell_2}}_{v-\text{standard}} z_v \cdot e_v
\end{equation}
for various coefficients $z_v$.

\medskip

\noindent
However, $[e_{\gamma_1},e_{\gamma_2}]\in \BQ^* \cdot e_{\gamma_1+\gamma_2}$ implies
$[e_{\gamma_1}^{(d_1)},e_{\gamma_2}^{(d_2)}]\in \BQ^* \cdot e_{\gamma_1+\gamma_2}^{(d_1+d_2)}$, so that:
\begin{equation}
\label{eqn:lynd comm 2}
  [e_{\ell_1}, e_{\ell_2}] \in \BQ^* \cdot e_{\ell}
\end{equation}
As $\{e_v|v-\text{standard}\}$ is a basis of $U(L^{(s)}\fn^+)$ (\cite[\S2.2]{LR}),
comparing~(\ref{eqn:lynd comm 1}) and~(\ref{eqn:lynd comm 2}), we conclude that $\ell \geq \ell_1\ell_2$.
This proves the inequality $\geq$ in \eqref{eqn:property lyndon}. As for the opposite inequality $\leq$,
it follows from the fact that $\ell(\alpha,d)$ admits a factorization 
\eqref{eqn:costandard factorization} $\ell(\alpha,d)=\ell_1\ell_2$ (with $\ell_1<\ell(\alpha,d)<\ell_2$),
and Propositions \ref{prop:standard}, \ref{prop:factor standard} (see Remark~\ref{rem:generalization})
imply that $\ell_k = \ell(\gamma_k,d_k)$ for some decomposition $(\alpha,d) = (\gamma_1,d_1)+(\gamma_2,d_2)$.
\end{proof}

\medskip

\noindent
Since standard Lyndon loop words give rise to bases of the finite-dimensional Lie algebra $L^{(s)} \fn^+$,
then the analogue of property \eqref{eqn:pbw lie} gives us:
\begin{equation}
\label{eqn:pbw lie loop filtration}
  U(L^{(s)}\fn^+) \ =
  \mathop{\mathop{\bigoplus^{k\in \BN}}_{\ell_1 \geq \dots \geq \ell_k \text{ standard Lyndon loop}}}_
  {\text{words with all exponents in }\{-s,\dots,s\}} \BQ \cdot e_{\ell_1} \dots e_{\ell_k}
\end{equation}
By the triangularity property \eqref{eqn:upper}, we could also get a basis of $U(L^{(s)}\fn^+)$ by replacing
$e_w = e_{\ell_1} \dots e_{\ell_k}$ in \eqref{eqn:pbw lie loop filtration} by $_we$, for any standard loop word $w$
with all exponents in $\{-s,\dots,s\}$.

\medskip


\subsection{}
\label{sub:properties affLynd}

Property \eqref{eqn:property lyndon} will allow us to deduce some facts about the bijection~(\ref{eqn:bijection lyndon}).

\medskip

\begin{proposition}
\label{prop:l1}
For any positive root $\alpha\in \Delta^+$ and integer $d\in \BZ$, we have:
\begin{equation}
\label{eqn:inequality lyndon}
  \ell(\alpha,d) < \ell(\alpha,d-1)
\end{equation}
where $\ell$ is the function of \eqref{eqn:bijection lyndon}, which a priori depends on a natural number $s$
(so we implicitly need $d-1,d \in \{-s|\alpha|, \dots,s|\alpha|\}$ in order for \eqref{eqn:inequality lyndon} to make sense).
\end{proposition}

\medskip

\begin{proof}
Let us prove \eqref{eqn:inequality lyndon} by induction on $|\alpha|$, the base case $|\alpha|=1$ being trivial.
According to \eqref{eqn:property lyndon}, there exist decompositions $\alpha = \gamma_1+\gamma_2$, $d = d_1+d_2$
such that:
$$
  \ell(\alpha,d) = \ell(\gamma_1,d_1)\ell(\gamma_2,d_2)
$$
with $\ell(\gamma_1,d_1) < \ell(\gamma_2,d_2)$. Note that $\gamma_1\ne \gamma_2$ as $\gamma_1+\gamma_2$ is a root.
Because we assume $d> -s|\alpha|$, then at least one of the following two options holds:

\medskip

\begin{itemize}[leftmargin=*]

\item
$d_1 > -s|\gamma_1|$, in which case the induction hypothesis implies $\ell(\gamma_1,d_1-1) > \ell(\gamma_1,d_1)$.
Then we either have $\ell(\gamma_1,d_1-1) < \ell(\gamma_2,d_2)$, in which case:
$$
  \ell(\alpha,d-1) \geq \ell(\gamma_1,d_1-1)\ell(\gamma_2,d_2)   > \ell(\gamma_1,d_1)\ell(\gamma_2,d_2) = \ell(\alpha,d)
$$
or $\ell(\gamma_1,d_1-1) > \ell(\gamma_2,d_2)$, in which case:
$$
  \ell(\alpha,d-1) \geq \ell(\gamma_2,d_2)\ell(\gamma_1,d_1-1)  >   \ell(\gamma_2,d_2)\ell(\gamma_1,d_1) >
  \ell(\gamma_1,d_1)\ell(\gamma_2,d_2) = \ell(\alpha,d)
$$

\medskip

\item
$d_2 > -s|\gamma_2|$, in which case the induction hypothesis implies $\ell(\gamma_2,d_2-1) > \ell(\gamma_2,d_2)$,
and so $\ell(\gamma_2,d_2-1) > \ell(\gamma_1,d_1)$. Then we have:
$$
  \ell(\alpha,d-1) \geq \ell(\gamma_1,d_1) \ell(\gamma_2,d_2-1)   > \ell(\gamma_1,d_1)\ell(\gamma_2,d_2) = \ell(\alpha,d)
$$
\end{itemize}

\noindent
In all chains of two or three inequalities above, the first inequality is due to \eqref{eqn:property lyndon},
while the third inequality uses Claim \ref{claim:lyndon}.
\end{proof}

\medskip

\noindent
Next, we estimate the exponents of letters in the standard Lyndon loop words for~$L^{(s)}\fn^+$.

\medskip

\begin{proposition}
\label{prop:l2}
For all $\alpha \in \Delta^+$ and $d \in \{-sk,\dots,sk\}$ with $k=|\alpha|$, we have:
\begin{equation}
\label{eqn:explicit lyndon}
  \ell(\alpha,d) = \left [i_1^{(d_1)} \dots\, i_{k}^{(d_{k})} \right] \quad \text{for various} \quad
  d_1, \dots, d_k \in \left \{ \left \lfloor \frac d{k} \right \rfloor, \left \lceil \frac d{k} \right \rceil \right \}
\end{equation}
\end{proposition}

\medskip

\begin{proof}
We will prove \eqref{eqn:explicit lyndon} by induction on $k$, the base case $k=1$ being trivial.

\medskip

\noindent
If $\frac dk = t \in \BZ$, then we must show that all exponents of $\ell(\alpha,d)$ are equal to $t$.
Indeed, pick a decomposition $\alpha = \gamma_1+\gamma_2$ into positive roots, and assume without loss
of generality that $\ell(\gamma_1,t|\gamma_1|) < \ell(\gamma_2,t|\gamma_2|)$ (otherwise, swap their order).
Then:
$$
  \ell(\alpha,d) \geq  \ell(\gamma_1,t|\gamma_1|) \ell(\gamma_2,t|\gamma_2|)
$$
by \eqref{eqn:property lyndon}. By the induction hypothesis, the word on the right has all exponents equal to $t$,
which implies that the first letter of $\ell(\alpha,d)$ has exponent $\leq t$. But because the first letter of
a Lyndon loop word is its smallest one, this implies that all letters of $\ell(\alpha,d)$ have exponent $\leq t$.
Because $\vdeg \ell(\alpha,d) = d = tk$ is also the sum of the exponents of $\ell(\alpha,d)$, this implies that
all letters of $\ell(\alpha,d)$ must have exponent equal to $t$, as we needed to prove.

\medskip

\noindent
If $tk < d < (t+1)k$ for some $t \in \BZ$, then we must show that all exponents of $\ell(\alpha,d)$ are equal
to either $t$ or $t+1$. By a slight modification of the argument in the preceding paragraph, we conclude that
the first letter of $\ell(\alpha,d)$ has exponent $= t+1$, which implies that all letters of $\ell(\alpha,d)$
have exponent $\leq t+1$. Then assume for the purpose of contradiction that there is some letter of $\ell(\alpha,d)$
with exponent $\leq t-1$. Consider the factorization \eqref{eqn:costandard factorization}:
\begin{equation}
\label{eqn:again}
  \ell(\alpha,d) = \ell(\gamma_1,d_1)\ell(\gamma_2,d_2)
\end{equation}
for some decomposition $\alpha = \gamma_1+\gamma_2$, $d = d_1+d_2$ with $|d_k|\leq s|\gamma_k|$ for $k\in \{1,2\}$.
Since the first letter of $\ell(\gamma_1,d_1)$ has exponent $t+1$, the induction hypothesis does not allow $\ell(\gamma_1,d_1)$
to have any letters with exponents $\leq t-1$. Therefore, the letters with exponents $\leq t-1$ must lie in $\ell(\gamma_2,d_2)$,
and so the induction hypothesis yields:
$$
  d_1 > t|\gamma_1| \quad \text{and} \quad d_2 < t|\gamma_2|
$$
However, if $\ell(\gamma_1,d_1-1) < \ell(\gamma_2,d_2+1)$ then the word $\ell(\gamma_1,d_1-1)\ell(\gamma_2,d_2+1)$
would be greater than $\ell(\gamma_1,d_1)\ell(\gamma_2,d_2) = \ell(\alpha,d)$, by Proposition \ref{prop:l1}, thus
contradicting the maximality of $\ell(\alpha,d)$ provided by \eqref{eqn:property lyndon}. The only other possibility is that
$\ell(\gamma_1,d_1-1) > \ell(\gamma_2,d_2+1)$, at which point the same property \eqref{eqn:property lyndon} implies that:
$$
  \ell(\alpha,d) \geq \ell(\gamma_2,d_2+1) \ell(\gamma_1,d_1-1)
$$
However, by the induction hypothesis, all the letters of $\ell(\gamma_2,d_2+1)$ have exponents $\leq t$,
which contradicts the fact that the first letter of $\ell(\alpha,d)$ has exponent $t+1$.
\end{proof}

\medskip


\subsection{}
\label{sub:standard stability}

Property \eqref{eqn:explicit lyndon} has one great advantage: it is independent of $s$.

\medskip

\begin{proposition}
\label{prop:coherent}
Any loop word $w$ with exponents in $\{-s, \dots, s\}$ is standard (Lyndon) with respect to $L^{(s)}\fn^{+}$
iff it is standard (Lyndon) with respect to $L^{(s+1)}\fn^{+}$.
\end{proposition}

\medskip

\begin{proof}
Due to Proposition~\ref{prop:stand via Lyndonstand} (see Remark~\ref{rem:generalization}), it suffices to consider
the case of standard Lyndon loop words. In other words, we must show that if $\alpha$ is a positive root
and $d$ is an integer such that $|d| \leq s |\alpha|$, then the Lyndon loop words:
\begin{align*}
  & \ell = \ell(\alpha,d) \text{ of  \eqref{eqn:bijection lyndon} with respect to }L^{(s)}\fn^+ \\
  & \ell' = \ell(\alpha,d) \text{ of  \eqref{eqn:bijection lyndon} with respect to }L^{(s+1)}\fn^+
\end{align*}
are equal. We may do so by induction on $|\alpha|$, the base case $|\alpha|=1$ being trivial.
Due to property \eqref{eqn:property lyndon}, both $\ell$ and $\ell'$ are defined as the maximum
over various concatenations, but the set of concatenations defining $\ell'$ is a priori larger.
In other words, the only situation in which $\ell \neq \ell'$ would be if:
$$
  \ell' = \ell(\gamma_1,d_1)\ell(\gamma_2,d_2) > \ell
$$
with $\ell(\gamma_1,d_1)$ or $\ell(\gamma_2,d_2)$ having an exponent $\pm (s+1)$. However, this can not happen
due to~\eqref{eqn:explicit lyndon} applied to $\ell'$, since it would force $|d| > s|\alpha|$.
\end{proof}

\medskip

\noindent
Proposition \ref{prop:coherent} implies that the notion ``standard Lyndon loop word" does not depend on
the particular $L^{(s)}\fn^+$ with respect to which it is defined. We conclude that there exists a bijection:
\begin{equation}
\label{eqn:associated word loop}
  \ell \colon \Delta^+\times \BZ \ \stackrel{\sim}\longrightarrow \ \Big\{\text{standard Lyndon loop words}\Big\}
\end{equation}
satisfying properties \eqref{eqn:property lyndon} and \eqref{eqn:explicit lyndon} (with $s = \infty$).

\medskip


\subsection{}
\label{sub:degree reduction}

Because of the Lie algebra isomorphism:
$$
  L \fn^+ \ \iso\ L \fn^+ \qquad \mathrm{given\ by} \qquad e_{\alpha}^{(d)} \mapsto e_\alpha^{(d+|\alpha|)}
$$
the procedure:
\begin{equation}
\label{eqn:procedure}
  \left[ i_1^{(d_1)} \dots\, i_k^{(d_k)} \right] \leadsto \left[ i_1^{(d_1+1)} \dots\, i_k^{(d_k+1)} \right]
\end{equation}
preserves the property of a loop word being standard. It obviously also preserves the property of a loop word being Lyndon,
hence also of being standard Lyndon, due to Proposition~\ref{prop:standard} (see Remark~\ref{rem:generalization}).
This implies the following result.

\medskip

\begin{proposition}
\label{prop:peridocitiy}
For any $(\alpha,d) \in \Delta^+ \times \BZ$, $\ell(\alpha,d+|\alpha|)$ is obtained from $\ell(\alpha,d)$
by adding 1 to all the exponents of its letters, i.e.\ by the procedure \eqref{eqn:procedure}.
\end{proposition}

\medskip

\noindent
Therefore, to describe the bijection \eqref{eqn:associated word loop}, it suffices to specify a finite amount
of data, i.e.\ the standard Lyndon loop words corresponding to $(\alpha,d)$ for all $\alpha\in \Delta^+$ and
$d \in \{1, \dots, |\alpha|\}$. This will be done in the Appendix for all classical types, corresponding to
a specific order of the simple roots.

\medskip

\begin{proposition}
\label{prop:horizontal}
The restriction of \eqref{eqn:associated word loop} to $\Delta^+ \times \{0\}$ matches \eqref{eqn:associated word}.
\end{proposition}


\medskip

\noindent
The result above is simply the $s = 0$ case of Proposition \ref{prop:coherent}. Since $U(L\fn^+)$ is the
direct limit as $s\rightarrow \infty$ of the $U(L^{(s)}\fn^+)$, then \eqref{eqn:pbw lie loop filtration} implies:
\begin{equation}
\label{eqn:pbw lie loop}
  U(L\fn^+) \ =
  \bigoplus^{k\in \BN}_{\ell_1 \geq \dots \geq \ell_k \text{ standard Lyndon loop words}}
    \BQ \cdot e_{\ell_1} \dots e_{\ell_k}
\end{equation}
By Proposition \ref{prop:stand via Lyndonstand} (see Remark~\ref{rem:generalization}), we then have:
\begin{equation}
\label{eqn:pbw lie loop 2}
  U(L\fn^+) \ = \bigoplus_{w \text{ standard loop words}} \BQ \cdot e_w
\end{equation}
The following result will be used in Section \ref{sec:quantum}.

\medskip

\begin{corollary}
\label{cor:finitely many}
For any loop word $w$, there exist finitely many standard loop words $\leq w$
in any fixed degree $(\alpha,d) \in Q^+ \times \BZ$.
\end{corollary}

\medskip

\begin{proof}
Any standard loop word $v$ admits a canonical factorization $v=\ell_1 \dots \ell_k$ where
$\ell_1 \geq \dots \geq \ell_k$ are all standard Lyndon loop words. If $v \leq w$, then we note that
all the $\ell_r$'s are bounded from above by $w$, due to $\ell_r\leq \ell_1\leq v$. Combining this
with \eqref{eqn:explicit lyndon}, we see that the exponents which appear among the letters of
the $\ell_r$'s are bounded from below. Therefore, there are only finitely many choices of
$\ell_1, \dots, \ell_k$ with a fixed number of letters, whose exponents sum up to precisely $d$.
\end{proof}

\medskip
\noindent
We conclude this Section with a few fundamental properties of the total order~(\ref{eqn:induces affine})
on $\Delta^+ \times \BZ$ induced by~\eqref{eqn:associated word loop} from the lexicographic order. 
The loop version of the convexity result from Proposition~\ref{prop:finite convexity} is
established in Proposition~\ref{prop:convex loop}. A corollary of the latter implies Proposition~\ref{prop:lyndon is minimal}
which is key to the proof of Theorem~\ref{thm:PBW quantum loop}.


\subsection{}
\label{sub:convex affLyndon}

The bijection \eqref{eqn:associated word loop} gives rise to a total order~(\ref{eqn:induces affine}) on
$\Delta^+ \times \BZ$, by transporting the total lexicographic order on loop words. We will now show
that this order is convex, a notion which is the direct generalization of Definition~\ref{def:convex}.

\medskip

\begin{proposition}
\label{prop:convex loop}
For all $(\alpha,d) , (\beta,e), (\alpha+\beta, d+e) \in \Delta^+ \times \BZ$, we have:
\begin{equation}
\label{eqn:convex loop}
  \ell  (\alpha,d) < \ell (\alpha+\beta,d+e) < \ell (\beta,e)
\end{equation}
if $\ell(\alpha , d) < \ell(\beta,e)$.
\end{proposition}

\medskip

\begin{proof}
We will prove the required statement by induction on $|\alpha+\beta|$, the base case being vacuous.
By \eqref{eqn:property lyndon}, we have:
$$
  \ell(\alpha+\beta,d+e) \geq \ell(\alpha,d)\ell(\beta,e) > \ell(\alpha,d)
$$
Therefore, it remains to show that $\ell(\alpha+\beta,d+e) < \ell(\beta,e)$.
Let us assume for the purpose of contradiction that the opposite inequality holds:
\begin{equation}
\label{eqn:contradiction 1}
  \ell(\alpha+\beta,d+e) > \ell(\beta,e) > \ell(\alpha,d)
\end{equation}
By \eqref{eqn:property lyndon}, we have:
\begin{equation}
\label{eqn:contradiction 2}
  \ell(\alpha+\beta,d+e) = \ell(\alpha',d')\ell(\beta',e')
\end{equation}
where $\ell(\alpha',d') < \ell(\beta',e')$, for certain positive roots $\alpha',\beta'$ satisfying
$\alpha+\beta = \alpha'+\beta'$ and integers $d',e'$ satisfying $d+e = d'+e'$. Comparing the formulas above,
we have two options:
\begin{align*}
  & \underline{\text{Case 1}}: \quad  \ell(\alpha',d') > \ell(\beta,e) \\
  & \underline{\text{Case 2}}: \quad  \ell(\alpha',d') < \ell(\beta,e)
\end{align*}
(note that the equality $(\alpha',d') = (\beta,e)$ would imply $(\alpha,d) = (\beta',e')$, which would
contradict various inequalities above). In \underline{Case 1}, we would have:
\begin{equation}
\label{eqn:three inequalities}
  \ell(\beta',e') > \ell(\alpha',d') > \ell(\beta,e) > \ell(\alpha,d)
\end{equation}
We will use \eqref{eqn:three inequalities} to obtain a contradiction, but first we make an elementary claim:

\medskip

\begin{claim}
\label{claim:root}
Given positive roots $\alpha, \beta, \alpha',\beta'$ such that $\alpha+\beta = \alpha'+\beta'$, then:
$$
  \alpha' = \alpha + \gamma \quad \text{and} \quad \beta' = \beta - \gamma
$$
or:
$$
  \alpha' = \beta + \gamma \quad \text{and} \quad \beta' = \alpha - \gamma
$$
for some $\gamma \in \Delta \sqcup \{0\}$.
\end{claim}

\medskip

\noindent
The Claim is proved as follows. Suppose first that $(\alpha,\alpha') > 0$. Then, the reflection
$s_{\alpha}(\alpha') = \alpha' - k\alpha$ is also a root, for some positive integer $k>0$. This implies that
$\alpha' - \alpha$ is either a root or $0$, hence $\alpha' - \alpha = \gamma$ for some $\gamma \in \Delta\sqcup \{0\}$,
thus proving the claim. The analogous argument applies if $(\alpha,\beta')>0$, $(\beta,\alpha')>0$, or $(\beta,\beta')>0$.
However, one of the aforementioned 4 inequalities must hold, or else
$0 \geq (\alpha+\beta, \alpha'+\beta') = (\alpha+\beta,\alpha+\beta)$, a contradiction.

\medskip

\noindent
Using Claim \ref{claim:root}, we conclude that there exist $\gamma \in \Delta \sqcup \{0\}$ and $x \in \BZ$ such that:
\begin{equation}
\label{eqn:case 1}
  (\alpha',d') = (\alpha + \gamma, d + x) \quad \text{and} \quad (\beta',e') = (\beta - \gamma, e - x)
\end{equation}
or:
\begin{equation}
\label{eqn:case 2}
  (\alpha',d') = (\beta + \gamma, e + x) \quad \text{and} \quad (\beta',e') = (\alpha - \gamma, d - x)
\end{equation}
(one just needs to pick the integer $x$ such that the equalities above hold). First of all, we cannot have $\gamma = 0$,
as Proposition \ref{prop:l1} and the chain of inequalities \eqref{eqn:three inequalities} would simultaneously require
$x > 0$ and $x < 0$. If $\gamma \neq 0$, then the induction hypothesis of \eqref{eqn:convex loop} contradicts the chain of
inequalities in \eqref{eqn:three inequalities}, as per the following:

\medskip

\begin{itemize}[leftmargin=*]

\item
If \eqref{eqn:case 1} holds and $\gamma \in \Delta^+$, the contradiction arises from the fact that
$\ell(\gamma,x)$ would have to be simultaneously bigger than $\ell(\alpha',d')$ and smaller than $\ell(\beta,e)$.

\medskip

\item
If \eqref{eqn:case 1} holds and $\gamma \in \Delta^-$, the contradiction arises from the fact that
$\ell(-\gamma,-x)$ would have to be simultaneously bigger than $\ell(\beta',e')$ and smaller than $\ell(\alpha,d)$.

\medskip

\item
If \eqref{eqn:case 2} holds and $\gamma \in \Delta^+$, the contradiction arises from the fact that
$\ell(\gamma,x)$ would have to be simultaneously bigger than $\ell(\alpha',d')$ and smaller than $\ell(\alpha,d)$.

\medskip

\item
If \eqref{eqn:case 2} holds and $\gamma \in \Delta^-$, the contradiction arises from the fact that
$\ell(-\gamma,-x)$ would have to be simultaneously bigger than $\ell(\beta',e')$ and smaller than $\ell(\beta,e)$.

\end{itemize}

\medskip

\noindent
In \underline{Case 2}, the only situation when \eqref{eqn:contradiction 1} and \eqref{eqn:contradiction 2}
are compatible would be if:
\begin{equation}
\label{eqn:how}
  \ell(\beta,e) = \ell(\alpha',d') w
\end{equation}
for some loop word $w$, which would need to satisfy:
$$
  \ell(\beta',e') > w > \ell(\beta,e)
$$
(the first inequality is a consequence of \eqref{eqn:contradiction 1} and \eqref{eqn:contradiction 2},
while the second inequality is a consequence of the fact that $\ell(\beta,e)$ is Lyndon). However, being
a suffix of a standard loop word, $w$ is also standard and hence admits a canonical factorization:
$$
  w = \ell(\gamma_1,f_1) \dots \ell(\gamma_k,f_k)
$$
for various $(\gamma_r,f_r) \in \Delta^+ \times\BZ$ which satisfy
$\ell(\gamma_r,f_r) \leq \ell(\gamma_1,f_1) \leq w < \ell(\beta',e')$ for all $1\leq r\leq k$.
However, \eqref{eqn:how} implies:
$$
  (\beta,e) = (\alpha',d') + \sum_{r=1}^k (\gamma_r, f_r) \quad \Rightarrow \quad
  (\beta',e') = (\alpha,d) + \sum_{r=1}^k (\gamma_r, f_r)
$$
Because $\alpha,\gamma_1,\dots,\gamma_k,\beta'$ are all positive roots, we claim that there exist
positive roots $\epsilon_1, \dots,\epsilon_k$ and a permutation $\sigma \in S(k)$ such that:
\begin{equation}
\label{eqn:epsilon}
  \epsilon_r = \alpha + \gamma_{\sigma(1)} + \dots + \gamma_{\sigma(r)} \quad \forall\, r \in \{1, \dots,k\}
\end{equation}
Since $\ell(\alpha,d)$ and all the $\ell(\gamma_r,f_r)$ are $< \ell(\beta',e')$, then the induction hypothesis
of \eqref{eqn:convex loop} implies (inductively in $r$) that:
$$
  \ell(\epsilon_r, d + f_{\sigma(1)} + \dots + f_{\sigma(r)}) < \ell(\beta',e')
$$
However, $(\epsilon_k, d+f_1+\dots+f_k) = (\beta',e')$, which provides the required contradiction.

\medskip

\noindent
It remains to prove \eqref{eqn:epsilon}, which we will do by induction on $k$, the base case $k=1$ being trivial.
If $(\alpha,\gamma_r) < 0$ for some $r$, then the reflection $s_{\alpha}(\gamma_r) = \gamma_r + p \alpha$ is also a root,
for some positive integer $p>0$. This implies that $\alpha+\gamma_r$ is a root, hence we can apply the induction hypothesis
for the collection of positive roots $(\alpha+\gamma_r,\gamma_1,\dots,\gamma_{r-1},\gamma_{r+1},\dots,\gamma_k, \beta')$.
The analogous argument applies if $(\beta',\gamma_r) > 0$ for some $r$, in which case we can apply the induction hypothesis
for the collection of positive roots $(\alpha,\gamma_1,\dots,\gamma_{r-1},\gamma_{r+1},\dots,\gamma_k, \beta'-\gamma_r)$.
Hence the only situation when we could not prove the claim via the argument above would be if:
$$
  (\alpha, \gamma_r) \geq 0 \geq (\beta',\gamma_r) \quad \forall r \qquad \Rightarrow \qquad
  (\alpha,\beta'-\alpha) \geq 0 \geq (\beta',\beta'-\alpha)
$$
But this would imply $(\beta'-\alpha,\beta'-\alpha) \leq 0$, which is impossible since $\beta'-\alpha \neq 0$.
\end{proof}

\medskip

\begin{remark}
We note that such ``lexicographic order on Lyndon words are convex" results are well-known in representation
theory, see e.g.~\cite{A} for slightly different (but more systematic and general) setting from ours.
\end{remark}

\medskip

\begin{corollary}
\label{cor:convex several}
Consider any $k,k' \geq 1$ and any:
$$
  (\gamma_1,d_1), \dots, (\gamma_k,d_k) ,(\gamma'_1,d'_1), \dots, (\gamma'_{k'},d'_{k'})  \in \Delta^+ \times \BZ
$$
such that:
\begin{equation}
\label{eqn:sum check}
  (\gamma_1,d_1)+ \dots + (\gamma_k,d_k) = (\gamma_1',d_1') + \dots + (\gamma_{k'}',d_{k'}')
\end{equation}
Then we have:
\begin{equation}
\label{eqn:min max}
  \min \Big \{\ell(\gamma_1,d_1), \dots, \ell(\gamma_k,d_k) \Big\} \leq
  \max \Big \{\ell(\gamma'_1,d'_1), \dots, \ell(\gamma'_{k'},d'_{k'}) \Big\}
\end{equation}
\end{corollary}

\medskip

\begin{proof}
Proposition~\ref{prop:convex loop} is simply the $(k,k') \in \{ (1,2), (2,1) \}$ case of the Corollary. Let us prove
the Corollary by induction on $\min(k,k')$, and to break ties, by $k+k'$. This means that we must start with the case
$\min(k,k') = 1$, and we will show how to deal with the $k'=1$ case (as the $k = 1$ case is an analogous exercise that
we leave to the interested reader). The assumption implies that $\gamma_1+\dots+\gamma_k \in \Delta^+$, in which case
\eqref{eqn:epsilon} shows that we can relabel indices such that $\gamma_1+\gamma_2 \in \Delta^+$.
Then the induction hypothesis shows that:
$$
  \min \Big\{ \ell(\gamma_1+\gamma_2,d_1+d_2), \ell(\gamma_3,d_3), \dots, \ell(\gamma_k, d_k) \Big\} \leq
  \ell(\gamma_1+ \dots + \gamma_k, d_1 + \dots + d_k)
$$
Then Proposition~\ref{prop:convex loop} for $(\gamma_1,d_1)$ and $(\gamma_2,d_2)$ implies that the left-hand side is
$\geq$ the minimum of all the $\ell(\gamma_s,d_s)$'s, as we needed to prove.

\medskip

\noindent
Let us now assume that $k,k' > 1$. Since:
$$
  \gamma_1 + \dots + \gamma_k = \gamma_1' + \dots + \gamma_{k'}'
$$
there exist $s,s'$ such that $(\gamma_s,\gamma_{s'}') > 0$. Let us relabel indices such that $s = s' = 1$.
As we saw in the proof of Claim \ref{claim:root}, this implies that:
$$
  (\gamma'_1,d'_1) = (\gamma_1, d_1) + (\epsilon, x)
$$
for some $\epsilon \in \Delta \sqcup \{0\}$ and some $x \in \BZ$. Then \eqref{eqn:sum check} implies:
$$
  (\gamma_2,d_2)+ \dots + (\gamma_k,d_k) = (\gamma_2',d_2') + \dots + (\gamma'_{k'}, d'_{k'}) + (\epsilon,x)
$$
If $\epsilon \in \Delta^+$, then the induction hypothesis gives us:
\begin{align*}
  & \min \Big\{\ell(\gamma_1, d_1), \ell(\epsilon, x) \Big\} \leq \ell(\gamma'_1,d'_1) \\
  & \min \Big \{\ell(\gamma_2,d_2), \dots, \ell(\gamma_k,d_k) \Big\} \leq
    \max \Big \{\ell(\epsilon,x),\ell(\gamma'_2,d'_2), \dots, \ell(\gamma'_{k'},d'_{k'}) \Big\}
\end{align*}
which implies \eqref{eqn:min max}. If $\epsilon \in \Delta^-$, then the induction hypothesis gives us:
\begin{align*}
  & \ell(\gamma_1, d_1) \leq \max \Big\{\ell(-\epsilon,-x),\ell(\gamma'_1,d'_1)\Big\} \\
  & \min \Big \{\ell(-\epsilon,-x),\ell(\gamma_2,d_2), \dots, \ell(\gamma_k,d_k) \Big\} \leq
    \max \Big \{\ell(\gamma'_2,d'_2), \dots, \ell(\gamma'_{k'},d'_{k'}) \Big\}
\end{align*}
which also implies \eqref{eqn:min max}. Finally, if $\epsilon = 0$ and $x \leq 0$, then Proposition \ref{prop:l1}
implies that $\ell(\gamma_1,d_1) \leq \ell(\gamma_1',d'_1)$, which easily yields \eqref{eqn:min max}.
If $\epsilon = 0$ and $x > 0$, then:
\begin{multline*}
  \min \Big \{\ell(\gamma_2,d_2), \dots, \ell(\gamma_k,d_k) \Big\} \leq
  \min \Big \{\ell(\gamma_2,d_2-x), \ell(\gamma_3,d_3), \dots, \ell(\gamma_k,d_k) \Big\} \leq \\
  \leq \max \Big \{\ell(\gamma'_2,d'_2), \dots, \ell(\gamma'_{k'},d'_{k'}) \Big\}
\end{multline*}
where the first inequality is due to~\eqref{eqn:inequality lyndon} and the second inequality holds because
of the induction hypothesis. The chain of inequalities above implies \eqref{eqn:min max}.
\end{proof}

\medskip

\begin{proposition}
\label{prop:lyndon is minimal}
If $\ell_1 < \ell_2$ are standard Lyndon loop words such that $\ell_1\ell_2$ is also a standard Lyndon loop word,
then we cannot have:
$$
  \ell_1 < \ell_1' < \ell_2' < \ell_2
$$
for standard Lyndon loop words $\ell_1',\ell_2'$ such that $\deg \ell_1 + \deg \ell_2 = \deg \ell_1' + \deg \ell_2'$.
\end{proposition}

\medskip

\begin{proof}
Assume such $\ell_1'$, $\ell_2'$ existed. Then by \eqref{eqn:property lyndon}, we would have:
\begin{equation}
\label{eqn:ineq conc}
  \ell_1' \ell_2' \leq \ell_1 \ell_2
\end{equation}
The only way this is compatible with $\ell_1 < \ell_1'$ is if:
$$
  \ell_1' = \ell_1 w
$$
for some loop word $w$, which must be standard due to Proposition \ref{prop:factor standard}
(or more precisely, its straightforward loop generalization). However, \eqref{eqn:ineq conc} then implies:
\begin{equation}
\label{eqn:ineq conc 2}
  w \ell_2' \leq \ell_2
\end{equation}
If we consider the canonical factorization \eqref{eqn:canonical factorization} of $w = u_1 \dots u_k$ for
standard Lyndon loop words $u_1 \geq \dots \geq u_k$, then \eqref{eqn:ineq conc 2} implies that:
$$
  u_k \leq \dots \leq u_1 < \ell_2
$$
Together with the assumption that $\ell_2' < \ell_2$, this violates Corollary \ref{cor:convex several} since:
$$
  \deg u_1 + \dots + \deg u_k + \deg \ell_2' = \deg w + \deg \ell_2' =
  \deg \ell_1' - \deg \ell_1 + \deg \ell_2' = \deg \ell_2
$$
\end{proof}

\medskip


\section{Lyndon words and Weyl groups}
\label{sec:weyl lyndon}

In the present Section, we will show that the lexicographic order~\eqref{eqn:induces affine} on $\Delta^+ \times \BZ$ induced by
\eqref{eqn:associated word loop} is closely related to the construction of \cite{P1,P2} applied to a reduced decomposition of
a certain translation element in the extended affine Weyl group associated to $\fg$. The reader who is interested in quantum groups,
and prepared to accept the proof of Theorem \ref{thm:weyl to lyndon}, may skip ahead to Section \ref{sec:quantum}.

\medskip


\subsection{}
\label{sub:affine Weyl}

Let us consider the \underline{affine root system} of type $\fg$:
$$
  \wDelta = \wDelta^+ \sqcup \wDelta^- \subset \wQ
$$
The affine root system has one more simple root $\alpha_0$ besides the simple roots $\{\alpha_i\}_{i\in I}$ of
the finite root system. Therefore, we may use formulas \eqref{eqn:cartan matrix} for $I$ replaced by:
$$
  \wI = I \sqcup 0
$$
which lead to the affine Cartan matrix $(a_{ij})_{i,j\in \wI}$ and the affine symmetrized Cartan matrix $(d_{ij})_{i,j\in \wI}$.
There is a natural identification:
\begin{equation}
\label{eqn:finite vs affine roots}
  \wQ\ \iso\ Q \times \BZ \qquad \mathrm{with} \qquad
  \alpha_i \mapsto (\alpha_i,0) \ \ \forall\, i\in I, \qquad \alpha_0 \mapsto (-\theta,1)
\end{equation}
where $\theta \in \Delta^+$ is the highest root of the finite root system. Note that $(0,1) \in Q \times \BZ$
is the minimal \underline{imaginary root} of the affine root system. With this in mind, we have the following
explicit description of the affine root system in terms of finite roots:
\begin{align}
  & \widehat{\Delta}^+ = \Big\{ \Delta^+ \times \BZ_{\geq 0} \Big\} \sqcup \Big\{ 0 \times \BZ_{>0} \Big\}
    \sqcup \Big\{ \Delta^- \times \BZ_{>0} \Big\}
    \label{eqn:hat plus} \\
  & \widehat{\Delta}^- = \Big\{ \Delta^- \times \BZ_{\leq 0} \Big\} \sqcup
    \Big\{ 0 \times \BZ_{<0} \Big\} \sqcup \Big\{ \Delta^+ \times \BZ_{<0} \Big\}
    \label{eqn:hat minus}
\end{align}
where $\BZ_{\geq 0}$, $\BZ_{>0}$, $\BZ_{\leq 0}$, $\BZ_{<0}$ denote the obvious subsets of $\BZ$.

\medskip

\begin{definition}
\label{def:affine Lie}
Let $\widehat{\fg}$ be as in Definition \ref{def:finite lie}, but using $\wI$ instead of $I$.
\end{definition}

\medskip

\noindent
As opposed from the non-degenerate pairing on finite type root systems, the pairing on affine type root systems has
a $1$-dimensional kernel, which is spanned by the imaginary root. Explicitly, this implies the fact that:
$$
  (\alpha_0 + \theta, - ) = 0 \quad \Leftrightarrow \quad  d_{0j} + \sum_{i \in I} \theta_i d_{ij} = 0
$$
for all $j \in I$, where the positive integers $\{\theta_i\}_{i\in I}$ (called the ``labels" of the corresponding
extended Dynkin diagram) are defined via:
\begin{equation}
\label{eqn:labels}
  \theta = \sum_{i\in I} \theta_i \alpha_i
\end{equation}
Using formula \eqref{eqn:rel 2 finite lie}, this implies that the Cartan element:
\begin{equation}
\label{eqn:classical central}
  c = h_0 + \sum_{i \in I} \theta_i h_i
\end{equation}
is central in $\widehat{\fg}$. Furthermore, we have the following relation between $\widehat{\fg}$ and $L\fg$.

\medskip

\begin{lemma}
There exists a Lie algebra isomorphism:
$$
  \widehat{\fg}/(c)\ \iso\ L\fg
$$
determined by the formulas:
\begin{align*}
  & e_i \mapsto e_i \otimes t^0 & & e_0 \mapsto f_\theta \otimes t^1 \\
  & f_i \mapsto f_i \otimes t^0 & & f_0 \mapsto e_\theta \otimes t^{-1} \\
  & h_i \mapsto h_i \otimes t^0 & & h_0 \mapsto -\sum_{i \in I} \theta_i h_i \otimes t^0
\end{align*}
for all $i \in I$, where $e_\theta$ (resp.\ $f_\theta$) is a root vector of degree $\theta$ (resp.\ $-\theta$).
\end{lemma}

\medskip


\subsection{}

We have already mentioned that convex orders of $\Delta^+$ are in 1-to-1 correspondence with reduced decompositions
of the longest element of the finite Weyl group $W$ associated to $\fg$. To define the latter explicitly, consider
the \underline{coroot lattice}:
\begin{equation}
\label{eqn:coroot lattice}
  Q^\vee =\, \bigoplus_{i\in I} \BZ\cdot \bs_i^\vee
\end{equation}
where for any $\alpha \in \Delta^+$ the corresponding \underline{coroot} $\alpha^\vee$ is defined via:
\begin{equation}
\label{eqn:coroots}
  \alpha^\vee = \frac {2\alpha}{(\alpha, \alpha)}
\end{equation}
The finite Weyl group $W$, i.e.\ the abstract Coxeter group associated to the Cartan matrix $(a_{ij})_{i,j\in I}$,
acts faithfully on the coroot lattice $Q^\vee$ as well as on the root lattice $Q$:
\begin{equation}
\label{eqn:action coroots}
  W \curvearrowright Q^\vee \qquad \text{and} \qquad W \curvearrowright Q
\end{equation}
via the following assignments:
\begin{equation}
\label{eqn:action coroots explicit}
  s_i(\mu) = \mu - (\bs_i, \mu) \bs_i^\vee\ \qquad \text{and} \qquad
  s_i(\lambda) = \lambda - (\lambda,\bs^\vee_i) \bs_i
\end{equation}
$\forall\, i \in I,\, \mu\in Q^\vee,\, \lambda\in Q$.

\medskip


\subsection{}

We will also encounter the \underline{affine Weyl group}, which is by definition the semidirect product:
\begin{equation}
\label{eqn:affine Weyl group}
  \wW = W \ltimes Q^\vee
\end{equation}
defined with respect to the action \eqref{eqn:action coroots}. It is well-known that $\wW$ is also the Coxeter group
associated to the Cartan matrix $(a_{ij})_{i,j\in \wI}$. In other words, the affine Weyl group is generated by the symbols
$\{s_i\}_{i\in \wI}$ defined by:
\begin{align*}
  & s_i=(s_i,0), \quad \forall\, i \in I \\
  & s_0=(s_\theta,-\theta^\vee)
\end{align*}
The affine analogue of the action $W \curvearrowright Q$ from the previous Subsection is:
\begin{equation}
\label{eqn:action affine coroots}
  \wW \curvearrowright \wQ
\end{equation}
where the generators of the affine Weyl group act by the following formulas:
\begin{align}
  & s_i(\lambda,d) = \left(\lambda - (\lambda,\bs^\vee_i) \bs_i , d \right), \quad \forall\, i \in I
    \label{eqn:action roots 1} \\
  & s_0(\lambda,d) = \left(\lambda - (\lambda,\theta^\vee) \theta , d + (\lambda,\theta^\vee) \right)
    \label{eqn:s0-action}
\end{align}
for all $(\lambda,d) \in Q \times \BZ \simeq \wQ$, see~\eqref{eqn:finite vs affine roots}. An important feature of
the affine Weyl group is that it contains a large commutative subalgebra:
$$
  1 \ltimes Q^\vee \subset \wW
$$
which acts on the affine root lattice $\wQ\simeq Q\times \BZ$ by translations:
\begin{equation}
\label{eqn:coroot action}
  \wmu(\lambda,d) = \left(\lambda, d-(\lambda, \mu) \right)
\end{equation}
$\forall\, \mu\in Q^\vee,\, \lambda\in Q,\, d\in \BZ$.
Here and henceforth, we write $\wmu$ for the element $1 \ltimes \mu \in \wW$ and
call it a translation element.

\medskip


\subsection{}

We will also need to consider the \underline{extended affine Weyl group}, which is by definition the semidirect product:
\begin{equation}
\label{eqn:extended affine Weyl group}
  \weW = W \ltimes P^\vee
\end{equation}
Above, $P^\vee$ is the \underline{coweight lattice}:
\begin{equation}
\label{eqn:coweight lattice}
  P^\vee =\, \bigoplus_{i\in I} \BZ\cdot \omega_i^\vee
\end{equation}
where the fundamental coweights $\{\omega^\vee_i\}_{i\in I}$ are dual to the simple roots $\{\alpha_j\}_{j\in I}$:
\begin{equation}
\label{eqn:coweights}
  (\alpha_j,\omega^\vee_i)=\delta_i^j
\end{equation}
In particular, $Q^\vee$ is a finite index subgroup of $P^\vee$. It is well-known that:
\begin{equation}
\label{eqt:extended vs usual}
  \weW\simeq \CT\ltimes \wW
\end{equation}
where the finite subgroup $\CT$ of $\weW$ is naturally identified with a subgroup of automorphisms of the Dynkin diagram
of $\hg$. The semi-direct product~\eqref{eqt:extended vs usual} is such that:
$$
  \tau s_i=s_{\tau(i)}\tau, \qquad \forall\, \tau\in \CT, i\in \wI
$$
Finally, the action~\eqref{eqn:action affine coroots} extends to:
\begin{equation}
\label{eqn:extended action affine coroots}
  \weW \curvearrowright \wQ
\end{equation}
via:
$$
  \tau(\alpha_i)=\alpha_{\tau(i)}, \qquad \forall\, \tau\in \CT, i\in \wI
$$
We still have the following formula, akin to \eqref{eqn:coroot action}:
\begin{equation}
\label{eqn:coweight action}
  \wmu(\lambda,d) = \left(\lambda, d-(\lambda, \mu) \right)
\end{equation}
$\forall\, \mu\in P^\vee,\, \lambda\in Q,\, d\in \BZ$, where $\wmu$ denotes the translation element $1 \ltimes \mu \in \weW$.

\medskip


\subsection{}
\label{sub:convex affine order}

Recall that the \underline{length} of an element $x \in \wW$, denoted by $l(x)\in \BN$, is the smallest number $l \in \BN$
such that we can write:
\begin{equation}
\label{eqn:reduced decomposition}
  x = s_{i_{1-l}} \dots s_{i_0}
\end{equation}
for various $i_{1-l}, \dots, i_0 \in \wI$. Every factorization \eqref{eqn:reduced decomposition} with $l=l(x)$
is called a \underline{reduced decomposition} of $x$. Given such a reduced decomposition, the terminal
subset (a priori, a multiset) of the affine root system is:
\begin{equation}
\label{eqn:terminal set}
  E_x=\Big\{ s_{i_0}s_{i_{-1}}\dots s_{i_{k+1}}(\bs_{i_{k}}) \Big| 0\geq k > -l \Big\} \subset \widehat{\Delta}
\end{equation}
It is well-known that $E_x$ is independent of the reduced decomposition of $x$, and consists of the positive
affine roots (all with multiplicity one) that are mapped to negative ones under the action of $x$:
\begin{equation}
\label{eqn:terminal description}
  E_x = \left\{ \widetilde{\lambda}\in \widehat{\Delta}^+ \Big| x(\widetilde{\lambda})\in \widehat{\Delta}^- \right\}
\end{equation}
In particular, we get the following description of the length of $x$:
\begin{equation}
\label{eqn:length def2}
  l(x) = \# \left\{ \widetilde{\lambda}\in \widehat{\Delta}^+ \Big| x(\widetilde{\lambda})\in \widehat{\Delta}^- \right\}
\end{equation}

\medskip

\noindent
The aforementioned length function $l\colon \wW\to \BN$ naturally extends to $\weW$ via:
$$
  l(\tau w)=l(w),\qquad \forall\, \tau\in \CT, w\in \wW
$$
Thus, the length $l(x)$ of $x\in \weW$ is the smallest number $l$ such that we can write:
\begin{equation}
\label{eqn:extended reduced decomposition}
  x = \tau s_{i_{1-l}} \dots s_{i_0}
\end{equation}
for various $i_{1-l},\dots,i_0\in \wI$ and (uniquely determined) $\tau\in \CT$. Given a reduced decomposition of $x\in \weW$
as in~\eqref{eqn:extended reduced decomposition} with $l=l(x)$, define $E_x$ via~\eqref{eqn:terminal set}. We note that $E_x$
is still described via~\eqref{eqn:terminal description} since $\tau$ acts by permuting negative affine roots.
Therefore, $E_x$ is independent of the reduced decomposition of $x$ and we still have:
\begin{equation}
\label{eqn:length def3}
  l(x) = \# \left\{ \widetilde{\lambda}\in \widehat{\Delta}^+ \Big| x(\widetilde{\lambda})\in \widehat{\Delta}^- \right\}
\end{equation}

\medskip

\begin{remark}
\label{rem:longest element}
A restricted case of the discussion above is when $\wW, \widehat{\Delta}$ are replaced by $W,\Delta$. In this case,
applying~(\ref{eqn:terminal description}) to the longest element $w_0\in W$ yields $E_{w_0}=\Delta^+$. Furthermore,
choosing a reduced decomposition $w_0=s_{i_{1-l}} \dots s_{i_0}$ amounts to placing a total order on $E_{w_0}=\Delta^+$ via:
\begin{equation}
\label{eqn:order-via-decomposition}
  \alpha_{i_0}<s_{i_0}(\alpha_{i_{-1}})<\dots<s_{i_0} s_{i_{-1}} \dots s_{i_{2-l}}(\alpha_{i_{1-l}})
\end{equation}
According to~\cite{P}, this total order of $\Delta^+$ is convex (see Definition~\ref{def:convex}), and conversely, any convex order of $\Delta^+$ arises
in this way for a certain (unique) reduced decomposition of $w_0$. We will study the affine version of this picture in
Subsection~\ref{sub:Lyndon via affWeyl}.
\end{remark}

\medskip

\noindent
Let us recall the element $\rho\in \frac{1}{2}Q$ defined by:
$$
  \rho = \frac{1}{2}\sum_{\alpha \in \Delta^+} \alpha
$$
The following result is standard (\cite[Exercise 6.10]{K}).

\medskip

\begin{proposition}
\label{prop:length for lattice}
For any $\mu \in P^\vee$ such that $(\alpha_i,\mu) \in \BN$ for all $i \in I$:
$$
  l(\wmu) = (2\rho, \mu)
$$
\end{proposition}

\medskip

\begin{proof}
Applying formula~(\ref{eqn:coweight action}) for the action of $\wmu \in \weW$ on $\wQ\simeq Q\times \BZ$, we see that
the only positive affine roots $\widetilde{\lambda}\in \widehat{\Delta}^+$ that are mapped to negative ones are:
\begin{equation}
\label{eqn:some roots}
  \Big\{ (\alpha,d) \Big| \alpha\in \Delta^+, 0\leq d<(\alpha,\mu) \Big\}
\end{equation}
Combining this with formula~(\ref{eqn:length def3}), we find
\begin{equation*}
  l(\wmu) \, = \sum_{\alpha\in \Delta^+} (\alpha,\mu) = (2\rho,\mu)
\end{equation*}
\end{proof}

\medskip


\subsection{}
\label{sub:Lyndon via affWeyl}

Let us pick any $\mu \in P^\vee$ such that $(\alpha_i, \mu) \in \BN$ for all $i \in I$. Let $l = (2\rho,\mu)$ be
the length of $\wmu \in \weW$ (Proposition~\ref{prop:length for lattice}) and consider any reduced decomposition:
\begin{equation}
\label{eqn:reduced-rho}
  \wmu = \tau s_{i_{1-l}}s_{i_{2-l}}\dots s_{i_0}
\end{equation}
Extend $i_{1-l},\dots,i_0$ to a ($\tau$-quasiperiodic) bi-infinite sequence $\{i_k\}_{k \in \BZ}$ via:
\begin{equation}
\label{eqn:tau-twisted sequence}
  i_{k+l}=\tau(i_k), \qquad \forall\, k\in \BZ
\end{equation}
To such a bi-infinite sequence~(\ref{eqn:tau-twisted sequence}), one assigns the following bi-infinite sequence of affine roots:
\begin{equation}
\label{eqn:beta-roots}
  \beta_k =
  \begin{cases}
    s_{i_1} s_{i_2} \dots s_{i_{k-1}}(-\alpha_{i_k}) &\text{if } k > 0 \\
    s_{i_0} s_{i_{-1}} \dots s_{i_{k+1}}(\alpha_{i_k}) &\text{if } k \leq 0
  \end{cases}
\end{equation}
According to~\cite{P1,P2}, the sequences:
\begin{align}
  & \beta_1 > \beta_2 > \beta_3 > \dots \label{eqn:beta positive} \\
  & \beta_0 < \beta_{-1} < \beta_{-2} < \dots \label{eqn:beta negative}
\end{align}
give convex orders of the sets $\Delta^+ \times \BZ_{<0}$ and $\Delta^+ \times \BZ_{\geq 0}$, respectively.

\medskip

\begin{remark}
\label{rem:damiani vs beck}
The above exposition follows that of~\cite{D3} as we consider $\mu\in P^\vee$. To reduce it to the setup of~\cite{B,P1,P2},
where only elements of $Q^\vee$ are treated, we note that if $r\in \BN$ is the order of $\tau$, then $r\mu\in Q^\vee$,
$s_{i_{1-rl}}s_{i_{2-rl}}\dots s_{i_{-1}}s_{i_0}$ is a reduced decomposition of $\widehat{r\mu}$, and the sequence
$\{i_k\}_{k\in \BZ}$ is periodic with period $l(\widehat{r\mu})=rl$.
\end{remark}

\medskip

\begin{remark}
For any $k \in \BZ$, if $\beta_k = (\alpha,d)$ and $\beta_{k+l} = (\alpha',d')$, then:
\begin{equation}
\label{eqn:beta periodicity}
  \beta_{k+l} = \wmu(\beta_k) \qquad \Rightarrow \qquad \alpha = \alpha' \text{ and } d = d'+(\alpha,\mu)
\end{equation}
due to~(\ref{eqn:coweight action}). This reveals a periodicity of the entire set $\Delta^+ \times \BZ$, not just of
its two halves $\Delta^+ \times \BZ_{<0}$ and $\Delta^+ \times \BZ_{\geq 0}$ (it is also the reason for the minus
sign in \eqref{eqn:beta-roots}).
\end{remark}

\medskip


\subsection{}

Recall the element $\rho^\vee \in P^\vee\cap \frac{1}{2}Q^\vee$ defined by:
$$
  \rho^\vee =\, \sum_{i\in I} \omega^\vee_i = \frac{1}{2}\sum_{\alpha \in \Delta^+} \alpha^\vee
$$

\medskip

\noindent
The following is the main result of this Section.

\medskip

\begin{theorem}
\label{thm:weyl to lyndon}
There exists a reduced decomposition of $\widehat{\rho^\vee}\in \weW$ such that:

\medskip

\begin{itemize}[leftmargin=*]

\item
the order~\eqref{eqn:beta positive} of the roots $\{(\alpha,d)|\alpha\in \Delta^+, d<0\}$ matches
the lexicographic order of the standard Lyndon loop words $\ell(\alpha,-d)$ via~(\ref{eqn:induces affine}),

\medskip

\item
the order~\eqref{eqn:beta negative} of the roots $\{(\alpha,d)|\alpha\in \Delta^+,d\geq 0\}$ matches
the lexicographic order of the standard Lyndon loop words $\ell(\alpha,-d)$ via~(\ref{eqn:induces affine}).

\end{itemize}
\end{theorem}

\medskip

\noindent
The second bullet implies that $i_0$ equals the smallest letter in $I$. On the other hand, combining
$s_{i_1}\rho^\vee=s_{\tau(i_{1-l})}\tau s_{i_{1-l}}s_{i_{2-l}}\dots s_{i_0}=\tau s_{i_{2-l}}\dots s_{i_0}$ with
the fact that $l(s_j\rho^\vee)>l(\rho^\vee) \ \forall j \in I$ (a consequence of~\eqref{eqn:length def3}),
implies that $i_1=0$.

\medskip

\begin{proof}[Proof of Theorem~\ref{thm:weyl to lyndon}]
Consider the finite subset:
$$
  L=\Big\{(\alpha,d) \Big| \alpha\in \Delta^+, 0\leq d<|\alpha|\Big\}
$$
of $\widehat{\Delta}^+$, ordered via:
\begin{equation}
\label{eqn:order-L}
  (\alpha,d)<(\beta,e) \quad \Leftrightarrow \quad \ell(\alpha,-d)<\ell(\beta,-e)
\end{equation}
If $(\alpha,d),(\beta,e)\in L$ with $(\alpha,d)<(\beta,e)$ and $(\alpha+\beta,d+e)\in \widehat{\Delta}$, then clearly
$(\alpha+\beta,d+e)\in L$, as well as $(\alpha,d) < (\alpha+\beta,d+e) < (\beta,e)$, due to Proposition~\ref{prop:convex loop}.

\medskip

\noindent
Furthermore, we claim that if $\widetilde{\lambda},\widetilde{\mu}\in \widehat{\Delta}^+$ with
$\widetilde{\lambda} + \widetilde{\mu} \in L$, then at least one of $\widetilde{\lambda}$ or $\widetilde{\mu}$
belongs to $L$ and is $<\widetilde{\lambda} + \widetilde{\mu}$. This is clear when
$\widetilde{\lambda}=(\alpha,d)$, $\widetilde{\mu}=(\beta,e)$ with $\alpha,\beta\in \Delta^+$ and $d,e\geq 0$.
In the remaining case, we may assume $\widetilde{\lambda}=(\alpha+\beta,d), \widetilde{\mu}=(-\beta,e)$, so that
$\alpha,\beta,\alpha+\beta \in \Delta^+$ and $d\geq 0, e>0$. Then $d<d+e<|\alpha|<|\alpha+\beta|$,
so that $\widetilde{\lambda}\in L$. It remains to verify $\widetilde{\lambda}<\widetilde{\lambda} + \widetilde{\mu}$,
that is, $\ell(\alpha+\beta,-d)<\ell(\alpha,-d-e)$. Since $(\alpha+\beta,-d)=(\beta,e)+(\alpha,-d-e)$, it suffices
to prove $\ell(\beta,e)<\ell(\alpha,-d-e)$, due to Proposition~\ref{prop:convex loop}. But applying Proposition~\ref{prop:l2},
we see that the exponent of the first letter in $\ell(\beta,e)$ is $>0$, while the exponent of the first letter in
$\ell(\alpha,-d-e)$ is $\leq 0$, hence, indeed $\ell(\beta,e)<\ell(\alpha,-d-e)$.

\medskip

\noindent
Invoking~\cite{P} (which also applies to finite subsets in affine root systems), we get:

\begin{enumerate}
    \item[(I)] there is a unique element $w\in \wW$ such that $L=E_w$
\medskip
    \item[(II)] the order of $L$ arises via a certain reduced decomposition of $w$, cf.~(\ref{eqn:order-via-decomposition}).
\end{enumerate}

\medskip

\noindent
However, as noticed in our proof of Proposition~\ref{prop:length for lattice}, we have
$$
  L=E_{\widehat{\rho^\vee}} = \Big\{\beta_0,\beta_{-1},\dots,\beta_{1-l}\Big\}
$$
There is a unique $\tau\in \CT$ such that $\tau^{-1}\widehat{\rho^\vee}\in \wW$
(note that $\tau^2 = 1$ since $2\rho^\vee \in Q^\vee$). Then:
$$
  L=E_{\widehat{\rho^\vee}} = E_{\tau^{-1}\widehat{\rho^\vee}}
$$
Therefore, in view of the uniqueness statement of~(I), the result of~(II) implies that there exists a reduced
decomposition~(\ref{eqn:reduced-rho}) of $\widehat{\rho^\vee}$ such that the ordered finite sequence
$\beta_0<\beta_{-1}<\dots<\beta_{1-l}$ exactly coincides with $L$ ordered via~(\ref{eqn:order-L}).

\medskip

\noindent
The proof of Theorem~\ref{thm:weyl to lyndon} now follows by a simple combination of~(\ref{eqn:beta periodicity}) and
Propositions~\ref{prop:l2},~\ref{prop:peridocitiy}. Indeed, let us split $\Delta^+ \times \BZ$ into the blocks:
$$
  L_N=\Big\{(\alpha,d) \Big| \alpha\in \Delta^+, N|\alpha| \leq d < (N+1)|\alpha|\Big\}
$$
so that:
\begin{align*}
  & \bigsqcup_{N\geq 0} L_N = \Delta^+\times \BZ_{\geq 0} = \{\beta_k\}_{k\leq 0} \\
  & \bigsqcup_{N<0} L_N = \Delta^+\times \BZ_{<0} = \{\beta_k\}_{k>0}
\end{align*}
According to~(\ref{eqn:beta periodicity}) and $L_0=L=\{\beta_0,\dots,\beta_{1-l}\}$, we have:
$$
  L_N =\Big\{\beta_{-Nl},\beta_{-Nl-1},\dots,\beta_{1-(N+1)l} \Big\}, \qquad \forall\, N\in \BZ
$$
For any $(\alpha,d)\in L_N$, the exponent of the first letter in $\ell(\alpha,-d)$ is $-N$,
due to Proposition~\ref{prop:l2} (and its proof). Therefore, for any $(\alpha,d)\in L_M, (\beta,e)\in L_N$
with $M>N$, we have $\ell(\alpha,-d)>\ell(\beta,-e)$. As for the affine roots from the same block,
consider $\beta_{r-Nl},\beta_{s-Nl}\in L_N$ with $1-l\leq s<r\leq 0$. If $\beta_r=(\alpha,d)$ and $\beta_s=(\beta,e)$,
then $\beta_{r-Nl}=(\alpha, d+N|\alpha|)$ and $\beta_{s-Nl}=(\beta,e+N|\beta|)$, due to~(\ref{eqn:beta periodicity}).
On the other hand, the words $\ell(\alpha,-d-N|\alpha|)$ and $\ell(\beta,-e-N|\beta|)$ are obtained from
$\ell(\alpha,-d)$ and $\ell(\beta,-e)$, respectively, by decreasing each exponent by $N$, due to
Proposition~\ref{prop:peridocitiy}. Since the latter operation obviously preserves the lexicographic order,
and $\ell(\alpha,-d)<\ell(\beta,-e)$ as a consequence of $r > s$, we obtain the required inequality
$\ell(\alpha,-d-N|\alpha|)<\ell(\beta,-e-N|\beta|)$.
\end{proof}

\medskip

\noindent
We actually have the stronger result that the order of $\Delta^+ \times \BZ$ given by:
\begin{equation}
\label{eqn:full chain}
  \dots < \beta_3 < \beta_{2} < \beta_1 < \beta_0 < \beta_{-1} < \beta_{-2} < \cdots
\end{equation}
matches the lexicographic order of the standard Lyndon loop words $\ell(\alpha,-d)$
(since $\ell(\alpha,-d) < \ell(\beta,-e)$ if $d < 0 \leq e$, itself a consequence of Proposition \ref{prop:l2}).

\medskip

\begin{remark}
We expect that a similar treatment can be done for any $\mu \in P^\vee$ such that $(\alpha_i,\mu)>0$ for all $i\in I$.
On the side of Lyndon loop words, this would require an analogue of Proposition \ref{prop:peridocitiy} stating that
$\ell(\alpha,d+(\alpha,\mu))$ is obtained from $\ell(\alpha,d)$ by adding $(\alpha_i,\mu)$ to all the exponents of
letters $i \in I$. For this operation to preserve the property of words being Lyndon, one can replace the order
\eqref{eqn:lex affine} on loop letters $\{i^{(d)}\}_{i\in I}^{d\in \BZ}$ by:
$$
  i^{(d)} < j^{(e)} \quad \text{if} \quad
  \begin{cases}
    \frac d{(\alpha_i,\mu)} > \frac e{(\alpha_j,\mu)} \\
     \text{ or } \\
     \frac d{(\alpha_i,\mu)} = \frac e{(\alpha_j,\mu)} \text{ and } i < j
  \end{cases}
$$
We expect the contents of Sections~\ref{sec:lie} and~\ref{sec:weyl lyndon} to carry through in this more general setup,
but we make no claims in this regard.
\end{remark}

\medskip


\section{Quantum groups and shuffle algebras}
\label{sec:quantum}

We will review the connection between Drinfeld-Jimbo quantum groups and shuffle algebras, following~\cite{G, R1, S}.
We will also recall the point of view of~\cite{L} (see also~\cite{R2}), which connects shuffle algebras with the notion
of standard Lyndon words. Then we develop a loop version of this treatment, and prove Theorem~\ref{thm:main 1}
(modulo the proof of Theorem~\ref{thm:PBW quantum loop}, which will be dealt with in the next Section).

\medskip
\noindent
We start with the exposition of the relevant results for finite quantum groups in Subsections~\ref{sub:Drinfeld-Jimbo}--\ref{sub:good words},
following the aforementioned references~\cite{G, R1, S} and~\cite{L}.


\subsection{}
\label{sub:Drinfeld-Jimbo}

Let us recall the notation of Subsection~\ref{sub:root system} which, as we have seen, corresponds to a finite-dimensional
simple Lie algebra $\fg$. Consider the $q$-numbers, $q$-factorials and $q$-binomial coefficients:
$$
  [k]_i = \frac {q_i^k - q_i^{-k}}{q_i - q_i^{-1}}, \qquad
  [k]!_i = [1]_i \dots [k]_i, \qquad
  {n \choose k}_i = \frac {[n]!_i}{[k]!_i [n-k]!_i}
$$
for any $i \in I$, where $q_i = q^{\frac {d_{ii}}2}$.

\medskip

\begin{definition}
\label{def:finite quantum group}
The Drinfeld-Jimbo quantum group associated to $\fg$ is:
$$
  \uu = \BQ(q) \Big \langle e_i, f_i, \ph_i^{\pm 1} \Big \rangle_{i \in I} \Big/
  \text{relations \eqref{eqn:rel 1 finite}--\eqref{eqn:rel 3 finite}}
$$
where we impose the following relations for all $i,j\in I$:
\begin{equation}
\label{eqn:rel 1 finite}
  \sum_{k=0}^{1-a_{ij}} (-1)^k {1-a_{ij} \choose k}_i e_i^k e_j e_i^{1-a_{ij}-k} = 0, \quad
  \qquad \mathrm{if }\ i\ne j
\end{equation}
\begin{equation}
\label{eqn:rel 2 finite}
  \ph_j e_i = q^{d_{ji}} e_i \ph_j, \qquad \quad \ph_i\ph_j = \ph_j \ph_i
\end{equation}
as well as the opposite relations with $e$'s replaced by $f$'s, and finally the relation:
\begin{equation}
\label{eqn:rel 3 finite}
  [e_i, f_j] = \delta_i^j \cdot \frac {\ph_i - \ph_i^{-1}}{q_i - q_i^{-1}}
\end{equation}
\end{definition}

\medskip

\noindent
If we let $\ph_i = q_i^{h_i}$ and take the limit $q \rightarrow 1$, then $\uu$ degenerates to $U(\fg)$.

\medskip


\subsection{}
\label{sub:bialgebra DrJim}

Recall that $\uu$ is a bialgebra with respect to the coproduct (\cite[\S4.11]{J}):
\begin{align*}
  & \Delta(\ph_i) = \ph_i \otimes \ph_i \\
  & \Delta(e_i) = \ph_i \otimes e_i + e_i \otimes 1 \\
  & \Delta(f_i) = 1 \otimes f_i + f_i \otimes \ph_i^{-1}
\end{align*}
This bialgebra structure preserves the $Q$-grading induced by setting (\cite[\S4.13]{J}):
$$
  \deg e_i = \alpha_i,\quad \deg \ph_i = 0, \quad \deg f_i = -\alpha_i
$$
Recall the triangular decomposition (\cite[\S4.21]{J}):
\begin{equation}
\label{eqn:triangular finite}
  \uu = \uup \otimes \uuo \otimes \uum
\end{equation}
where $\uup, \uuo, \uum$ are the subalgebras of $\uu$ generated by the $e_i$'s, $\ph^{\pm 1}_i$'s, $f_i$'s, respectively.
We will also consider the following sub-bialgebras of $\uu$:
\begin{align*}
  & \uug = \uup \otimes \uuo \\
  & \uul = \uuo \otimes \uum
\end{align*}

\medskip

\begin{remark}
\label{rmk:positive as ass.algebra}
As an associative algebra, $\uup$ (resp.~$\uug$) is generated by $e_i$'s (resp.\ $e_i,\ph_i^{\pm 1}$'s) with the defining
relations~(\ref{eqn:rel 1 finite}) (resp.~(\ref{eqn:rel 1 finite},~\ref{eqn:rel 2 finite})), see e.g.~\cite[\S4.21]{J}.
\end{remark}

\medskip


\subsection{}

It is well-known (\cite[\S6.12]{J}) that there is a non-degenerate bialgebra pairing:\footnote{Henceforth, given two algebras
$A,B$ over a ring $K$, a $K$-valued bilinear pairing $A\times B \to K$ shall be rather denoted $A\otimes B\to K$ (with $\otimes$
standing for $\otimes_K$) to indicate its $K$-bilinear nature.}
\begin{equation}
\label{eqn:bialg pair finite}
  \langle \cdot, \cdot \rangle \colon \uug \otimes \uul \longrightarrow \BQ(q)
\end{equation}
where the word ``bialgebra" means that it satisfies the following properties:
\begin{align}
  & \langle a, bc \rangle = \langle \Delta(a), b \otimes c \rangle
    \label{eqn:bialg pair 1} \\
  & \langle ab, c \rangle = \langle b \otimes a, \Delta(c) \rangle
    \label{eqn:bialg pair 2}
\end{align}
for all applicable $a,b,c$. Then \eqref{eqn:bialg pair finite} is determined by the assignments:
$$
  \langle e_i, f_j \rangle =\frac {\delta_i^j}{q_i^{-1} - q_i}, \qquad
  \langle \ph_i, \ph_j \rangle = q^{-d_{ij}}
$$
and the fact that
$$
  \langle a,b\rangle = 0 \quad \mathrm{unless} \quad \deg a + \deg b = 0
$$

\medskip

\noindent
The quantum group $\uu$ is the Drinfeld double of $\left(\uug,\uul,\langle\cdot,\cdot\rangle\right)$,
which means that the multiplication map induces an isomorphism:
$$
  \uug \otimes \uul \Big/ ( \ph_i \otimes \ph_i^{-1} - 1 \otimes 1 ) \ \iso \ \uu
$$
and that the commutation rule of the two factors is governed by the
relation:\footnote{According to~\cite[Remark 2.4]{N2}, formula~\eqref{eqn:double} is equivalent to a more standard
commutation rule appearing in the literature. We prefer our formula as it does not require us to define the antipode,
which exists but will not be necessary in the present paper.}
\begin{equation}
\label{eqn:double}
  a_1 b_1 \langle a_2, b_2 \rangle = \langle a_1, b_1 \rangle b_2 a_2
\end{equation}
for all $a \in \uug$ and $b \in \uul$. Here we use Sweedler notation $\Delta(a) = a_1 \otimes a_2$ for the coproduct
of Subsection \ref{sub:bialgebra DrJim} (a summation sign is implied in front of $a_1 \otimes a_2$).

\medskip


\subsection{}
\label{sub:root vectors finite}

Since the quantum group of Definition \ref{def:finite quantum group} is a $q$-deformation of the universal enveloping of the
Lie algebra of Definition \ref{def:finite lie}, it is natural that many features of the latter admit $q$-deformations as well.
For example, let us recall the notion of standard Lyndon words from Subsections~\ref{sub:finite words}--\ref{sub:standard words},
and consider the following $q$-version of the construction of Definition \ref{def:bracketing}.

\medskip

\begin{definition}
\label{def:quantum bracketing}
(\cite{L})
For any word $w$, define $e_w \in U_q(\fn^+)$ by:
$$
  e_{[i]} = e_i
$$
for all $i \in I$, and then recursively by:
\begin{equation}
\label{eqn:quantum bracketing lyndon}
  e_{\ell} = \left[ e_{\ell_1}, e_{\ell_2} \right]_q =
  e_{\ell_1} e_{\ell_2} - q^{(\deg \ell_1, \deg \ell_2)} e_{\ell_2} e_{\ell_1}
\end{equation}
if $\ell$ is a Lyndon word with factorization~\eqref{eqn:costandard factorization}, and:
\begin{equation}
\label{eqn:quantum bracketing arbitrary}
  e_w = e_{\ell_1} \dots e_{\ell_k}
\end{equation}
if $w$ is an arbitrary word with the canonical factorization $\ell_1 \dots \ell_k$, as in~\eqref{eqn:canonical factorization}.
\end{definition}

\medskip

\noindent
We also define $f_w \in U_q(\fn^-)$ by replacing $e$'s by $f$'s in the Definition above.
Then we have the following $q$-deformation of the PBW statement \eqref{eqn:pbw lie}.

\medskip

\begin{theorem}
\label{thm:PBW quantum finite}
We have:
\begin{multline}
\label{eqn:fin PBW lyndon}
  U_q(\fn^+) \ =
  \bigoplus^{k \in \BN}_{\ell_1 \geq \dots \geq \ell_k \text{ standard Lyndon words}}
  \BQ(q) \cdot e_{\ell_1} \dots e_{\ell_k} \ = \\ \bigoplus_{w \text{ standard words}} \BQ(q) \cdot e_w
\end{multline}
The analogous result also holds with $+ \leftrightarrow -$ and $e \leftrightarrow f$.
\end{theorem}

\medskip

\noindent
This result is a consequence of the usual PBW theorem for $U_q(\fn^\pm)$, since $e_\ell$'s are simply renormalizations
of the standard root vectors constructed in \cite{Lu}, according to~\cite[Theorem 28]{L}. We shall provide more details in
Subsection~\ref{sub:proof of finite PBW for Lyndon} to motivate our treatment of the loop counterpart.


\medskip


\subsection{}
\label{sub:finite shuffle}

One of the main tools of~\cite{L} is the $q$-shuffle algebra interpretation of the quantum group $\uup$,
due to~\cite{G, R1, S}, which we recall now.

\medskip

\begin{definition}
\label{def:shuf finite}
Consider the $\BQ(q)$-vector space $\CF$ with a basis given by words:
\begin{equation}
\label{eqn:finite words}
  [i_1 \dots i_k]
\end{equation}
for arbitrary $k \in \BN$, $i_1, \dots, i_k \in I$, and endow it with the following
\underline{shuffle product}:\footnote{We note that formula \eqref{eqn:shuf finite} is worded differently
from \cite[formula (9)]{L}, but it is an immediate consequence of \cite[formula (8)]{L}.}
\begin{equation}
\label{eqn:shuf finite}
  [i_1 \dots i_k] * [j_1 \dots j_l] \ =
  \mathop{\sum_{\{1, \dots ,k+l\} = A \sqcup B}}_{|A| = k, |B| = l}
    q^{\lambda_{A,B}} \cdot [s_1 \dots s_{k+l}]
\end{equation}
where in the right-hand side, if $A = \{a_1< \dots <a_k\}$ and $B = \{b_1< \dots <b_l\}$, we write:
\begin{equation}
\label{eqn:cases}
  s_c = \begin{cases}
    i_\bullet &\text{if } c = a_\bullet \\
    j_\bullet &\text{if } c = b_\bullet
  \end{cases}
\end{equation}
and:
\begin{equation}
\label{lambda powers}
  \lambda_{A,B} \ = \sum_{A \ni a > b \in B} d_{s_a s_b}
\end{equation}
\end{definition}

\medskip

\noindent
It is straightforward to see that $(\CF, *)$ is an associative algebra. If we set $q = 1$, then $\CF$ coincides
with the classical shuffle algebra on the alphabet $I$. The classical shuffle algebra is actually a bialgebra,
with coproduct defined by splitting words:
$$
  \Delta \left( [i_1 \dots i_k] \right) = \sum_{a=0}^k \, [i_1 \dots i_a] \otimes [i_{a+1} \dots i_k]
$$
But for generic $q$, the coproduct above is no longer multiplicative with respect to the shuffle
product~\eqref{eqn:shuf finite}. To remedy this, we consider the \underline{extended shuffle algebra}:
$$
  \CF^{\text{ext}} = \CF \otimes \BQ(q) \left[ \ph_i^{\pm 1} \right]_{i \in I}
$$
with pairwise commuting $\ph_i$'s, where the multiplication is governed by the rule:
\begin{equation}
\label{eqn:finite shuffle comm phi}
  \ph_j \cdot [i_1 \dots i_k] = q^{\sum_{a=1}^k d_{j i_a}} [i_1 \dots i_k] \cdot \ph_j
\end{equation}
It is straightforward to check that the assignment $\Delta(\ph_i) = \ph_i \otimes \ph_i$ and:
\begin{equation}
\label{eqn:cop finite shuffle}
  \Delta \left( [i_1 \dots i_k] \right) =
  \sum_{a=0}^k \, [i_1 \dots i_a] \ph_{i_{a+1}} \dots \ph_{i_k} \otimes [i_{a+1} \dots i_k]
\end{equation}
is both coassociative and gives rise to a bialgebra structure on $\CF^{\text{ext}}$.

\medskip

\begin{remark}
Our construction differs slightly from~\cite{G,R1}, where $\CF$ itself is endowed with a bialgebra structure by modifying
the product on $\CF\otimes \CF$ in the spirit of \cite[p.~3]{Lu}. However, the two approaches are easily seen to be equivalent.
\end{remark}

\medskip


\subsection{}
\label{sub:finite group vs shuffle}

It is straightforward to check that there is a unique algebra homomorphism:
\begin{equation}
\label{eqn:shuffle map finite}
  \uup \stackrel{\Phi}\longrightarrow \CF
\end{equation}
sending $e_i$ to $[i]$ (as one just needs to check that relations~\eqref{eqn:rel 1 finite} hold in $\CF$, due to
Remark~\ref{rmk:positive as ass.algebra}). Moreover, it is easy to prove by induction on $|\deg x|$ (using the bialgebra
pairing properties~(\ref{eqn:bialg pair 1}) and~(\ref{eqn:bialg pair 2})) that the map $\Phi$ is explicitly given by:
\begin{equation}
\label{eqn:shuffle map finite explicit}
  \Phi(x) \ =
  \sum^{k\in \BN}_{i_1, \dots ,i_k \in I} \left[ \prod_{a=1}^{k} (q_{i_a}^{-1}-q_{i_a}) \right]
    \Big\langle x, f_{i_1} \dots f_{i_k} \Big\rangle \cdot [i_1 \dots i_k]
\end{equation}
Because the bialgebra pairing~(\ref{eqn:bialg pair finite}) is non-degenerate and
$\langle x,y\ph^- \rangle = \langle x,y \rangle$ for any $x\in \uup, y\in \uum$ and $\ph^-$ a product of $\ph^-_i$'s
(which is a simple consequence of the bialgebra pairing properties~(\ref{eqn:bialg pair 1}) and~(\ref{eqn:bialg pair 2})),
\eqref{eqn:shuffle map finite explicit} implies the injectivity of~$\Phi$.
The image of the map $\Phi$ is described in~\cite[Theorem 5]{L}, which states that:
\begin{equation}
\label{eqn:image}
  \text{Im}\,  \Phi = \left\{ \sum^{r\in \BN}_{i_1, \dots ,i_r \in I} \gamma(i_1 \dots i_r) \cdot [i_1 \dots i_r] \right\}
\end{equation}
where the constants $\gamma(i_1 \dots i_r) \in \BQ(q)$ vanish for all but finitely many values of $r$ and
satisfy the following property:
\begin{equation}
\label{eqn:leclerc image}
  \sum_{k=0}^{1-a_{ij}} (-1)^k {1-a_{ij} \choose k}_i \
  \gamma \left( w \quad \underbrace{i \ \dots \ i}_{k \text{ symbols}} \quad
                j\underbrace{i \ \dots \ i}_{1-a_{ij}-k \text{ symbols}} w' \right) = 0
\end{equation}
for any distinct $i,j \in I$ and any words $w, w'$.

\medskip

\noindent
Comparing~\eqref{eqn:rel 2 finite} with~\eqref{eqn:finite shuffle comm phi}, it is easy to see that
the algebra homomorphism~(\ref{eqn:shuffle map finite}) extends to a bialgebra homomorphism:
$$
  \uug \stackrel{\Phi}\longrightarrow \CF^{\text{ext}}
$$
by sending $\ph_i\mapsto \ph_i$.

\medskip


\subsection{}
\label{sub:good words}

As in Subsection~\ref{sub:finite words}, we fix a total order on the set $I$, and consider the induced
lexicographic order on the set of all words~\eqref{eqn:finite word}.

\medskip

\begin{definition}
\label{def:good}
(\cite{L})
A word $w$ is called \underline{good} if there exists an element:
\begin{equation}
\label{eqn:good}
  w + \sum_{v < w} c_v \cdot v
\end{equation}
in $\emph{Im }\Phi$, for certain constants $c_v \in \BQ(q)$.
\end{definition}

\medskip

\noindent
If a word is good, then so are all its prefixes and suffixes and hence all its subwords
(\cite[Lemma 13]{L}, see also Proposition~\ref{prop:good} for a version of this statement in the loop case).

\medskip

\begin{proposition}
\label{prop:good is standard}
(\cite[Lemma 21]{L})
A word is good if and only if it is standard.
\end{proposition}

\medskip

\noindent
Above, we invoke the notion of standard words from Definition~\ref{def:standard}(a).
Likewise, the standard Lyndon words from Definition~\ref{def:standard}(b) as well as
the bijection~\eqref{eqn:associated word} can also be characterized in terms of the map $\Phi$, as follows.

\medskip

\begin{lemma}
\label{lem:minimal}
(\cite[Corollary 27, Theorem 36]{L})
For any $\alpha \in \Delta^+$, the leading word of $\Phi(e_{\ell(\alpha)})$ is $\ell(\alpha)$.
Moreover, the word $\ell(\alpha)$ is the smallest good word of degree $\alpha$.
\end{lemma}

\medskip
\noindent
In the rest of this Section, we develop the loop version of the above results with the aim of proving Theorem~\ref{thm:main 1}.
To this end, we construct a PBW basis of $\UUp$ parametrized by standard loop words in Theorem~\ref{thm:PBW quantum loop}, introduce the
loop version $\hCF$ of the shuffle algebra $\CF$ and relate it to $\UUp$ in Subsections~\ref{sub:affine shuffle}--\ref{sub:fix},
establish a loop version of Proposition~\ref{prop:good is standard} in Proposition~\ref{prop:standard is good}, and conjecture a
loop version of Lemma~\ref{lem:minimal} in Conjecture~\ref{conj:minimal loop}.
Finally, with the aim of proving Theorem~\ref{thm:main 2} in Section~\ref{sec:fo}, we filter $\UUp$ by the subspaces $\UUp_{\leq w}$
of~\eqref{eqn:filtration} for any loop word $w$, whose graded dimension~\eqref{eqn:geq} is expressed in terms of good words $\leq w$,
and discuss their pairing with $\UUm^{\leq w}$ of~\eqref{eqn:subspace 1} in Proposition~\ref{prop:non-degenerate}.


\subsection{}
\label{sub:qaffine}

We will now develop a loop version of the above notions, with the goal of proving Theorem \ref{thm:main 1}.
In what follows, we will use the generating series:
$$
  e_i(z) = \sum_{k \in \BZ} \frac {e_{i,k}}{z^k}, \qquad
  f_i(z) = \sum_{k \in \BZ} \frac {f_{i,k}}{z^k}, \qquad
  \ph^\pm_i(z) = \sum_{l = 0}^\infty \frac {\ph^\pm_{i,l}}{z^{\pm l}}
$$
and consider the formal delta function $\delta(z) = \sum_{k \in \BZ} z^k$.
For any $i,j \in I$, set:
\begin{equation}
\label{eqn:zeta}
  \zeta_{ij} \left(\frac zw \right) = \frac {z - wq^{-d_{ij}}}{z - w}
\end{equation}
We now recall the definition of the quantum loop group (new Drinfeld realization).

\medskip

\begin{definition}
\label{def:quantum loop}
The quantum loop group associated to $\fg$ is:
$$
  \UU = \BQ(q) \Big \langle e_{i,k}, f_{i,k}, \ph_{i,l}^\pm \Big \rangle_{i \in I, k \in \BZ, l \in \BN} \Big/
  \text{relations \eqref{eqn:rel 0 affine}--\eqref{eqn:rel 3 affine}}
$$
where we impose the following relations for all $i,j \in I$:
\begin{equation}
\label{eqn:rel 0 affine}
  e_i(z) e_j(w) \zeta_{ji} \left(\frac wz \right) =\, e_j(w) e_i(z) \zeta_{ij} \left(\frac zw \right)
\end{equation}
\begin{multline}
\label{eqn:rel 1 affine}
  \sum_{\sigma \in S(1-a_{ij})} \sum_{k=0}^{1-a_{ij}} (-1)^k {1-a_{ij} \choose k}_i \cdot  \\
  \qquad \qquad \qquad
  e_i(z_{\sigma(1)})\dots e_i(z_{\sigma(k)}) e_j(w) e_i(z_{\sigma(k+1)}) \dots e_i(z_{\sigma(1-a_{ij})}) = 0,
  \quad \mathrm{if}\ i\ne j
\end{multline}
\begin{equation}
\label{eqn:rel 2 affine}
  \ph_j^\pm(w) e_i(z) \zeta_{ij} \left(\frac zw \right) = e_i(z) \ph_j^\pm(w) \zeta_{ji} \left(\frac wz \right)
\end{equation}
\begin{equation}
\label{eqn:rel 2 affine bis}
  \ph_{i}^{\pm}(z) \ph_{j}^{\pm'}(w) = \ph_{j}^{\pm'}(w) \ph_{i}^{\pm}(z), \qquad
  \ph_{i,0}^+ \ph_{i,0}^- = 1
\end{equation}
as well as the opposite relations with $e$'s replaced by $f$'s, and finally the relation:
\begin{equation}
\label{eqn:rel 3 affine}
  \left[ e_i(z), f_j(w) \right] =
  \frac {\delta_i^j \delta \left(\frac zw \right)}{q_i - q_i^{-1}} \cdot  \Big( \ph_i^+(z) - \ph_i^-(w) \Big)
\end{equation}
\end{definition}

\medskip

\noindent
Note that there is a unique algebra homomorphism:
$$
  \uu \hooklongrightarrow \UU
$$
sending $e_i \mapsto e_{i,0},\, f_i \mapsto f_{i,0},\, \ph^{\pm 1}_i \mapsto \ph_{i,0}^{\pm}$.

\medskip


\subsection{}
\label{sub:qaffine coproduct}

Recall that $\UU$ is a topological bialgebra with respect to the following coproduct (\cite[formulas (5)--(7)]{Dr}):
\begin{equation}
\label{eqn:cop phi series}
  \Delta \left( \ph_i^\pm(z) \right) = \ph_i^\pm(z) \otimes \ph_i^\pm(z)
\end{equation}
\begin{equation}
\label{eqn:cop e}
  \Delta \left( e_i(z) \right) = \ph_i^+(z) \otimes e_i(z) + e_i(z) \otimes 1
\end{equation}
\begin{equation}
\label{eqn:cop f}
  \Delta \left( f_i(z) \right) = 1 \otimes f_i(z) + f_i(z) \otimes \ph_i^-(z)
\end{equation}
This bialgebra structure preserves the $Q \times \BZ$-grading induced by setting:
$$
  \deg e_{i,k} = (\alpha_i,k),\quad \deg \ph_{i,l}^\pm = (0,\pm l), \quad \deg f_{i,k} = (-\alpha_i,k)
$$
for all applicable indices. Recall the triangular decomposition (\cite[\S3.3]{He}):
\begin{equation}
\label{eqn:triangular loop}
  \UU = \UUp \otimes \UUo \otimes \UUm
\end{equation}
where $\UUp, \UUo, \UUm$ are the subalgebras of $\UU$ generated by the $e_{i,k}$'s, $\ph_{i,l}^\pm$'s, $f_{i,k}$'s,
respectively. We note that the following subalgebras of $\UU$:
\begin{align*}
  & \UUg = \UUp \otimes \BQ(q) \left[ \ph_{i,0}^{\pm}, \ph^+_{i,1}, \ph_{i,2}^+, \dots \right]_{i \in I} \\
  & \UUl = \BQ(q) \left[ \ph_{i,0}^{\mp}, \ph^-_{i,1}, \ph_{i,2}^-, \dots \right]_{i \in I} \otimes \UUm
\end{align*}
are preserved by the coproduct $\Delta$, and hence are sub-bialgebras of $\UU$.

\medskip


\subsection{}
\label{sub:qaffine pairing}

It is well-known (\cite[Lemma 9.1]{Groj}, see also~\cite[\S4]{E1},~\cite[\S1.3--1.4]{Gr} for more details)
that there exists a bialgebra pairing:
\begin{equation}
\label{eqn:bialg pair affine}
  \langle \cdot, \cdot \rangle \colon \ \UUg \otimes \UUl \longrightarrow \BQ(q)
\end{equation}
that satisfies~(\ref{eqn:bialg pair 1},~\ref{eqn:bialg pair 2}) and is determined by the properties:
\begin{align}
  & \Big\langle e_i(z), f_j(w) \Big\rangle =
    \frac {\delta_i^j \delta \left(\frac zw \right)}{q_i^{-1}-q_i} \label{eqn:aff pair 1} \\
  & \Big\langle \ph_{i}^+(z), \ph_{j}^-(w) \Big\rangle =
    \frac {\zeta_{ij} \left(\frac zw \right)}{\zeta_{ji} \left(\frac wz \right)} \label{eqn:aff pair 2}
\end{align}
(the right-hand side of \eqref{eqn:aff pair 2} is expanded in $|z| \gg |w|$) and the fact that:
$$
  \langle a, b \rangle = 0 \quad \mathrm{unless} \quad \deg a + \deg b = (0,0) \in Q\times \BZ
$$
This pairing is known to be non-degenerate (cf.~\cite[Section 9.3]{Groj},~\cite[Proposition~9]{Gr},~\cite[Theorem 1.4]{E2}),
although we will provide an alternative argument below.

\medskip

\begin{proposition}
\label{prop:non-degenerate qaff pairing}
The pairing $\langle \cdot,\cdot \rangle$ of~(\ref{eqn:bialg pair affine}) is non-degenerate in each argument.
\end{proposition}

\medskip

\noindent
We will give a proof of this result in Subsection \ref{sub:injectivity}.

\medskip


\subsection{}
\label{sub:root vectors loop}

Let us now provide a loop version of the constructions of Subsection \ref{sub:root vectors finite}.

\medskip

\begin{definition}
\label{def:loop quantum bracketing}
For any loop word $w$, define $e_w \in \UUp$, $f_w \in \UUm$ by:
$$
  e_{[i^{(d)}]} = e_{i,d} \quad \text{and} \quad f_{[i^{(d)}]} = f_{i,-d}
$$
for all $i \in I$, $d \in \BZ$, and then recursively by:
\begin{align}
\label{eqn:quantum bracketing lyndon affine}
  & e_{\ell} = \left[ e_{\ell_1}, e_{\ell_2} \right]_q =
    e_{\ell_1} e_{\ell_2} - q^{(\mathrm{hdeg}\, \ell_1, \mathrm{hdeg}\, \ell_2)} e_{\ell_2} e_{\ell_1} \\
  & f_{\ell} = \left[ f_{\ell_1}, f_{\ell_2} \right]_q =
    f_{\ell_1} f_{\ell_2} - q^{(\mathrm{hdeg}\, \ell_1, \mathrm{hdeg}\, \ell_2)} f_{\ell_2} f_{\ell_1}
\end{align}
if $\ell$ is a Lyndon loop word with factorization~\eqref{eqn:costandard factorization}, and:
\begin{equation}
\label{eqn:quantum bracketing arbitrary affine}
  e_w = e_{\ell_1} \dots e_{\ell_k} \quad \text{and} \quad f_w = f_{\ell_1} \dots f_{\ell_k}
\end{equation}
if $w$ is an arbitrary loop word with the canonical factorization $\ell_1 \dots \ell_k$, as in~\eqref{eqn:canonical factorization}.
\end{definition}

\medskip

\noindent
Note that $\deg e_w = -\deg f_w = \deg w$ for all loop words $w$. We have the following result, which is simultaneously
an analogue of both~(\ref{eqn:pbw lie loop},~\ref{eqn:pbw lie loop 2}) and Theorem \ref{thm:PBW quantum finite}.

\medskip

\begin{theorem}
\label{thm:PBW quantum loop}
We have:
\begin{multline*}
  \UUp \ =
  \bigoplus^{k \in \BN}_{\ell_1 \geq \dots \geq \ell_k \text{ standard Lyndon loop words}}
    \BQ(q) \cdot e_{\ell_1} \dots e_{\ell_k} \ = \\
  \bigoplus_{w \text{ standard loop words}} \BQ(q) \cdot e_w
\end{multline*}
The analogous result also holds with $+ \leftrightarrow -$ and $e \leftrightarrow f$.
\end{theorem}

\medskip

\noindent
The proof of the Theorem above will occupy most of Section \ref{sec:loop affine}, where we will derive it from
the PBW Theorem for the affine quantum group (in the Drinfeld-Jimbo presentation) constructed by~\cite{B,D}.

\medskip


\subsection{}
\label{sub:affine shuffle}

We will now define a ``loop" version of the shuffle algebra, which is to $\UU$ as the shuffle algebra of
Definition~\ref{def:shuf finite} is to $\uu$. The careful reader will observe a slight error in
Definition~\ref{def:shuf affine} as the right-hand side in the shuffle product~\eqref{eqn:shuf affine}
contains infinitely many summands. This will be remedied in Subsection~\ref{sub:fix} by introducing an
appropriate completion, but we prefer this slightly imprecise approach in order to keep the exposition clear.

\medskip

\begin{definition}
\label{def:shuf affine}
Take the $\BQ(q)$-vector space $\hCF$ with a basis given by loop words:
\begin{equation}
\label{eqn:loop word def}
  \left[i^{(d_1)}_1 \dots\, i^{(d_k)}_k \right]
\end{equation}
for arbitrary $k \in \BN$, $i_1, \dots, i_k \in I$, $d_1, \dots, d_k \in \BZ$,
and endow it with the following \underline{shuffle product}:
\begin{equation}
\label{eqn:shuf affine}
\begin{split}
  & \left[i^{(d_1)}_1 \dots\, i^{(d_k)}_k \right] * \left[j^{(e_1)}_1 \dots\, j^{(e_l)}_l \right] =\\
  & \quad \mathop{\sum_{\{1, \dots ,k+l\} = A \sqcup B}}_{|A| = k, |B| = l}
     \left( \mathop{\sum_{\pi_1 + \dots +\pi_{k+l} = 0}}_{\pi_1, \dots, \pi_{k+l} \in \BZ}
     \gamma_{A,B,\pi_1, \dots ,\pi_{k+l}} \cdot
     \left[s^{(t_1+\pi_1)}_1 \dots\, s^{(t_{k+l} + \pi_{k+l})}_{k+l} \right] \right)
\end{split}
\end{equation}
where in the right-hand side, if $A = \{a_1< \dots <a_k\}$ and $B = \{b_1 < \dots < b_l\}$, we write:
\begin{equation}
\label{eqn:cases affine}
  s_c = \begin{cases} i_\bullet &\text{if } c = a_\bullet \\ j_\bullet &\text{if } c = b_\bullet \end{cases}, \qquad
  t_c = \begin{cases} d_\bullet &\text{if } c = a_\bullet \\ e_\bullet &\text{if } c = b_\bullet \end{cases}
\end{equation}
and $\gamma_{A,B,\pi_1, \dots,\pi_{k+l}}$ are defined as the coefficients of the Taylor expansion:
\begin{equation}
\label{eqn:power series}
  \prod_{A \ni a > b \in B} \frac {\zeta_{s_a s_b} \left( \frac {z_{a}}{z_{b}} \right)}
                                  {\zeta_{s_b s_a} \left( \frac {z_{b}}{z_{a}} \right)} \ =
  \mathop{\sum_{\pi_1+\dots+\pi_{k+l} = 0}}_{\pi_1,\dots,\pi_{k+l} \in \BZ}
  \gamma_{A,B,\pi_1,\dots,\pi_{k+l}} \cdot z_{1}^{\pi_1} \dots z_{k+l}^{\pi_{k+l}}
\end{equation}
in the limit when $|z_{a}| \gg |z_{b}|$ for all $a \in A,\ b \in B$.
\end{definition}

\medskip

\begin{remark}
\noindent
(a) We note that in the inner sum of~\eqref{eqn:shuf affine} the only terms which appear with non-zero coefficient
are those with $\pi_c \leq 0$ if $c \in A$ and $\pi_c \geq 0$ if $c \in B$.

\medskip

\noindent
(b) We also have $\gamma_{A,B,0,\dots,0} = q^{\lambda_{A,B}}$ with $\lambda_{A,B}$ defined in~(\ref{lambda powers}).
\end{remark}

\medskip

\noindent
It is straightforward to see that $(\hCF,*)$ is an associative algebra, $Q^+\times \BZ$-graded
by~\eqref{eqn:degree of loop word}, and we leave this check as an exercise to the interested reader.

\medskip

\begin{proposition}
\label{prop:loop to shuffle}
There is a unique algebra homomorphism:
\begin{equation}
\label{eqn:loop to shuffle}
  \UUp \stackrel{\wPhi}\longrightarrow \hCF
\end{equation}
sending $e_{i,d}\mapsto \left[i^{(d)}\right]$. The homomorphism $\wPhi$ is injective and is explicitly given by
\begin{equation}
\label{eqn:def wPhi}
  \wPhi(x) \ =
  \mathop{\mathop{\sum_{i_1, \dots ,i_k \in I}}_{d_1, \dots ,d_k \in \BZ}}^{k\in \BN}
  \left[ \prod_{a=1}^{k}(q_{i_a}^{-1} - q_{i_a}) \right]
  \Big \langle x, f_{i_1,-d_1} \dots f_{i_k,-d_k} \Big \rangle \cdot  \left[i_1^{(d_1)} \dots\, i_k^{(d_k)} \right]
\end{equation}
for all $x \in \UUp$, where the pairing is that of \eqref{eqn:bialg pair affine}.
\end{proposition}

\medskip

\noindent
The Proposition above is straightforward, so we leave it as an exercise to the interested reader
(alternatively, it follows from Proposition~\ref{prop:shuf hom} below). The injectivity follows immediately from
the non-degeneracy of~\eqref{eqn:bialg pair affine}, due to Proposition~\ref{prop:non-degenerate qaff pairing}.

\medskip

\begin{remark}
\label{rem:comparison 1 to SV}
We note that our definition of $\hCF$ is actually equivalent to the main construction of \cite[\S2]{Gr}
(in fact our presentation is to \loccit\ as Green's presentation~\cite{G} is to Rosso's presentation~\cite{R1}
of shuffle algebras in the finite type case). Moreover, a version of the above construction of $\hCF$
and the homomorphism~(\ref{eqn:loop to shuffle}) (which correspond in our notation to $|I| = 1$,
but a more complicated $\zeta$-factor) featured in~\cite[\S1.9]{SV}.
\end{remark}

\medskip


\subsection{}
\label{sub:fix}

We note a certain imprecision in Definition~\ref{def:shuf affine}, which we will remedy now: the right-hand side
of~\eqref{eqn:shuf affine} is an infinite sum. However, because of the power series nature of this infinite sum,
the imprecision can be easily fixed as follows. Amend Definition \ref{def:shuf affine} by considering instead:
\begin{equation}
\label{eqn:completion fix}
  \hCF \ = \bigoplus_{\bk \in Q^+, d \in \BZ} \hCF_{\bk, d}
\end{equation}
where we consider the following completions:
\begin{equation}
\label{eqn:infinite sums}
  \hCF_{\bk,d} =
  \left\{ \mathop{\sum_{d_1+ \dots +d_a \text{ bounded from}}}_{\text{below, for all }a \in \{1,\dots,k\}}
  c_{i_1,\dots,i_k;d_1,\dots,d_k} \cdot
  \underbrace{\left[i_1^{(d_1)} \dots\, i_k^{(d_k)}\right]}_{\text{has degree } (\bk,d)} \right\}
\end{equation}
with arbitrary coefficients $c_{i_1,\dots,i_k;d_1,\dots,d_k} \in \BQ(q)$.

\medskip

\begin{proposition}
\label{prop:well-defined completion}
The shuffle product~\eqref{eqn:shuf affine} is well-defined on $\hCF$ of~(\ref{eqn:completion fix},~\ref{eqn:infinite sums}).
\end{proposition}

\medskip

\begin{proof}
We begin by showing that the operation $w * w'$ of~\eqref{eqn:shuf affine}
extends to a well-defined operation on infinite linear combinations of the form:
\begin{equation}
\label{eqn:two infinite sums}
  \left( \sum_{\deg w = (\bk,d)} c_w \cdot w \right) * \left( \sum_{\deg w' = (\bk',d')} c'_{w'} \cdot w' \right)
\end{equation}
where we have $c_w \neq 0$ (resp.\ $c_{w'}' \neq 0$) only if every prefix of $w$ (resp.\ $w'$) has vertical degree
bounded from below by some fixed $m \in \BZ$. Take an arbitrary word $v$ and consider the set:
\begin{equation*}
  S = \Big \{(w,w')
      \text{ such that } c_w \neq 0, c_{w'}' \neq 0 \text{ and } v \text{ appears as a summand in } w * w' \Big\}
\end{equation*}
We need to show that $S$ is finite, which would imply that the coefficient of $v$ in the shuffle
product~\eqref{eqn:two infinite sums} is well-defined. Let us assume for the purpose of contradiction that
$S$ is infinite. Since the vertical degrees of arbitrary prefixes of $w$ and $w'$ are bounded from below,
this implies that one of these prefixes has arbitrarily large vertical degree. Without loss of generality,
let us assume that we are talking about the length $a$ prefix of $w$. Thus, for any $N \in \BN$,
there exists $(w,w') \in S$ such that the vertical degree of $w_{a|}$ is at least $N$. However,
since all the prefixes of $w'$ have vertical degree at least equal to the fixed constant $m$, then
all terms in the shuffle product $w * w'$ will have some prefix with vertical degree at least $N+m$.
If $N$ is large enough, this contradicts the fact that $v$ appears as a summand in $w * w'$.

\medskip

\noindent
We now need to prove that the expression \eqref{eqn:two infinite sums} is of the
form~(\ref{eqn:completion fix},~\ref{eqn:infinite sums}). The loop words $v$ that appear in the expression
\eqref{eqn:two infinite sums} also do appear in the shuffle products $w * w'$, where $w$ and $w'$ are loop words of
fixed degrees, such that every prefix of $w$ and $w'$ has vertical degree bounded from below by some fixed $m \in \BZ$.
Thus, any loop word appearing in the shuffle product $w * w'$ has degree $\deg w+ \deg w'$, while any of its prefixes has
vertical degree bounded from below by $2m$ (an immediate consequence of~(\ref{eqn:shuf affine}) and~(\ref{eqn:power series})),
which is precisely what we needed to prove.
\end{proof}

\medskip


\subsection{}
\label{sub:coproduct on affshuffle}

Just like in Subsection~\ref{sub:finite shuffle}, there is no bialgebra structure on $\hCF$.
However, there is a bialgebra structure on the \underline{extended shuffle algebra}:
$$
  \hCFext =
  \hCF \otimes \BQ(q) \left[ (\ph^+_{i,0})^{\pm 1}, \ph^+_{i,1},  \ph^+_{i,2}, \dots \right]_{i \in I}
$$
with pairwise commuting $\ph$'s, where the multiplication is governed by the rule:
\begin{equation}
\label{eqn:affine shuffle comm phi}
\begin{split}
   & \ph_{j,e}^+ * \left[i_1^{(d_1)} \dots\, i_k^{(d_k)} \right] = \\
   & \qquad \qquad \qquad
     \sum_{\pi_1,\dots,\pi_k \geq 0} \mu_{\pi_1,\dots,\pi_k} \cdot
     \left[i_1^{(d_1+\pi_1)} \dots\, i_k^{(d_k+\pi_k)} \right] * \ph^+_{j,e-\pi_1-\dots-\pi_k}
\end{split}
\end{equation}
where $\ph^+_{j,<0}=0$ and $\mu_{\pi_1,\dots,\pi_k}$ are defined as the coefficients of the Taylor expansion:
$$
  \prod_{r=1}^k \frac {\zeta_{ji_r} \left( w / z_{i} \right)}{\zeta_{i_r j} \left( z_{i}/w \right)} \ =
  \sum_{\pi_1,\dots,\pi_k \geq 0} \mu_{\pi_1,\dots,\pi_k} \cdot
  \frac {z_1^{\pi_1} \dots z_k^{\pi_k}}{w^{\pi_1+ \dots +\pi_k}}
$$
It is straightforward to check that the right-hand side of~(\ref{eqn:affine shuffle comm phi}) indeed lies in $\hCF$
of~(\ref{eqn:completion fix},~\ref{eqn:infinite sums}) tensored with
$\BQ(q) \left[ (\ph^+_{i,0})^{\pm 1}, \ph^+_{i,1},  \ph^+_{i,2}, \dots \right]_{i \in I}$,
and that~(\ref{eqn:affine shuffle comm phi}) extends to the entire $\hCF$.
It is also easy to check that the assignment
$$
  \Delta(\ph^+_i(z))=\ph^+_i(z)\otimes \ph^+_i(z)
$$
and
\begin{equation}
\label{eqn:cop affine shuffle}
  \Delta \left(\left[i_1^{(d_1)} \dots\, i_k^{(d_k)} \right]  \right) =
\end{equation}
$$
  \sum_{a=0}^k \sum_{\pi_{a+1}, \dots ,\pi_k \geq 0} \left[i_1^{(d_1)} \dots\, i_a^{(d_a)}\right]
  \ph_{i_{a+1},\pi_{a+1}}^+ \dots \ph_{i_k,\pi_k}^+ \otimes
  \left[i_{a+1}^{(d_{a+1}-\pi_{a+1})} \dots\, i_k^{(d_k-\pi_k)}\right]
$$
is both coassociative and gives rise to a bialgebra structure on $\hCFext$. We note that the
coproduct~\eqref{eqn:cop affine shuffle} is topological, in the same sense as the coproduct~\eqref{eqn:cop e}.

\medskip

\noindent
Finally, comparing~\eqref{eqn:rel 2 affine} with~\eqref{eqn:affine shuffle comm phi} as well
as~(\ref{eqn:cop e}) with~(\ref{eqn:cop affine shuffle}), we see that the algebra
homomorphism~\eqref{eqn:loop to shuffle} extends to a bialgebra homomorphism:
$$
  \UUg \stackrel{\wPhi}\longrightarrow \hCFext
$$
by sending $\ph^+_{i,r}\mapsto \ph^+_{i,r}$.

\medskip


\subsection{}
\label{sub:loop words}

Define good loop words just like in Definition~\ref{def:good} (by replacing $\Phi$ with $\wPhi$).

\medskip

\begin{proposition}
\label{prop:good}
Any subword of a good loop word is good.
\end{proposition}

\medskip

\begin{proof}
It is enough to prove that any prefix and suffix of a good loop word is good. To this end, assume that $w$
is a good loop word of length $k$, which implies that there exists $x \in \UUp$ such that:
$$
  \wPhi(x) = w + \sum_{v < w} c_v \cdot v
$$
for various $c_v \in \BQ(q)$. We may assume that $x$ is homogeneous of degree $\deg w = (\bk,d)$, which implies that
$c_v \neq 0$ only if $\deg v = (\bk, d)$. Formula \eqref{eqn:cop affine shuffle} implies:
\begin{equation}
\label{eqn:delta iota}
  \Delta (\wPhi(x)) = \sum_{b=0}^k  w_{b|} \cdot \ph \otimes w_{|k-b} + \dots
\end{equation}
where the ellipsis denotes tensors $\alpha \ph'\otimes \beta$ with $\ph'$ being products of $\ph^+_{i,r}$'s
and $\alpha,\beta$ being loop words, such that if the loop word $\alpha$ has length $b$,
then either ($\alpha < w_{b|}$) or ($\alpha = w_{b|}$ and $\beta < w_{|k-b}$) or ($\alpha = w_{b|}$ and $\vdeg \beta < \vdeg w_{|k-b}$) \footnote{Here the vertical degree vdeg of a word \eqref{eqn:loop word def} is naturally defined to be $d_1+\dots+d_k$, cf.~\eqref{eq:hor and vert}.}; the latter option accounts for the situation when $\ph'$ is a product of Cartan elements $\ph^+_{i,r}$ with at least one such element having $r > 0$. Fix $a \in \{0,\dots,k\}$. We will write:
$$
  \Delta(x) = \sum_c  y_c \cdot \ph \otimes z_c + \underline{\qquad}
$$
for some $y_c, z_c \in \UUp$ of degrees  $(\bk_a, d_a)$, $(\bk - \bk_a, d - d_a)$, respectively, where (above and henceforth)
$(\bk_a,d_a) = \deg w_{a|}$, $\ph$ is a product of $\ph_{i,0}^+$'s and their inverses that depends only on
$(\bk - \bk_a, d - d_a)$, and the blank denotes tensors of degrees other than $(\bk_a, d_a) \otimes (\bk - \bk_a, d - d_a)$.
Therefore, we have:
\begin{equation}
\label{eqn:iota delta}
  \Big(\wPhi \otimes \wPhi\Big)\left(\Delta(x)\right) =
  \sum_c \wPhi(y_c) \cdot \ph \otimes \wPhi(z_c) +\underline{\qquad}
\end{equation}
Using the fact that $\wPhi$ intertwines the coproducts, we conclude that the left-hand sides of \eqref{eqn:delta iota}
and~\eqref{eqn:iota delta} are equal, hence so are their right-hand sides. If we just look at the tensors
of degrees $(\bk_a, d_a) \otimes (\bk - \bk_a, d - d_a)$, then we obtain the following identity:
\begin{equation}
\label{eqn:blue}
 \sum_c \wPhi(y_c) \otimes \wPhi(z_c) = w_{a|} \otimes w_{|k-a} + \dots
\end{equation}
where the ellipsis denotes tensors $\alpha \ph'\otimes \beta$ with $\ph'$ being products of $\ph^+_{i,r}$'s
and $\alpha,\beta$ being loop words, such that if the loop word $\alpha$ has length $a$,
then either ($\alpha < w_{a|}$) or ($\alpha = w_{a|}$ and $\beta < w_{|k-a}$).
Among all the tensors $y_c \otimes z_c$ that appear in \eqref{eqn:blue}, let us consider the one for which:
$$
  \wPhi(y_c)
$$
has the maximal leading order term. If there are several such tensors with the same maximal leading order term,
then by taking appropriate linear combinations, we can ensure that there is a single one. Formula \eqref{eqn:blue}
then requires:
\begin{equation}
\label{eqn:prefix is good}
  \wPhi(y_c)  = s \cdot w_{a|} \, + \sum_{v < w_{a|}} r_{v,a} \cdot v
\end{equation}
for $s \in \BQ(q)^*$ and various $r_{v,a} \in \BQ(q)$. Since only the tensor $y_c \otimes z_c$ can produce terms
of the form $w_{a|} \otimes \underline{\qquad}$ in \eqref{eqn:blue}, then:
\begin{equation}
\label{eqn:suffix is good}
  \wPhi(z_{c}) = t \cdot w_{|k-a} \ + \sum_{v < w_{|k-a}} r'_{v,a} \cdot v
\end{equation}
for $t \in \BQ(q)^*$ and various $r'_{v,a} \in \BQ(q)$. Formulas~(\ref{eqn:prefix is good},~\ref{eqn:suffix is good})
imply that both $w_{a|}$ and $w_{|k-a}$ are good loop words, as we needed to show.
\end{proof}

\medskip

\begin{proposition}
\label{prop:good lyndon word}
A loop word is good if and only if it can be written as:
$$
  \ell_1 \dots \ell_k
$$
where $\ell_1 \geq \dots \geq \ell_k$ are good Lyndon loop words.
\end{proposition}

\medskip

\begin{proof}
The ``only if" statement is an immediate consequence of Proposition~\ref{prop:canonical factorization} and
Proposition~\ref{prop:good}. As for the ``if" statement, suppose that we have good Lyndon loop words
$\ell_1 \geq \dots \geq \ell_k$. By definition, there exist elements:
\begin{equation}
\label{eqn:tail}
  \wPhi(x_r) = \ell_r + \sum_{v < \ell_r} \text{coefficient} \cdot v
\end{equation}
for various $x_r \in \UUp$. We may assume that each $x_r$ is homogeneous, and that so are the $v$'s in~(\ref{eqn:tail}),
hence all of them have the same number of letters as $\ell_r$. But then the leading order term of $\wPhi(x_1 \dots x_k)$
is the leading word in the shuffle product $\ell_1 * \dots *\ell_k$. By the obvious analogue of~\cite[Lemma 15]{L}, this
shuffle product has the leading order term equal to the concatenation $\ell_1 \dots \ell_k$. This exactly means that
the latter concatenation is a good loop word, as we needed to show.
\end{proof}

\medskip


\subsection{}
\label{sub:pairing on filtrations}

Invoking Definition~\ref{def:loop quantum bracketing}, for any loop word $w$ consider:
\begin{equation}
\label{eqn:subspace 1}
  \UUm^{\leq w} \ = \bigoplus_{v \leq w \text{ standard loop word }} \BQ(q) \cdot f_v
\end{equation}
which is finite-dimensional in any degree $\in Q^- \times \BZ$ according to Corollary \ref{cor:finitely many}.
For any loop word $w$, we also define:
\begin{equation}
\label{eqn:filtration}
  \UUp_{\leq w} \subset \UUp
\end{equation}
to consist of those elements $x$ such that the leading order term of $\wPhi(x)$ is $\leq w$.
Invoking~(\ref{eqn:def wPhi}), we note that $\UUp_{\leq w}$ consists of those $x\in \UUp$ such that:
\begin{equation}
\label{eqn:explicit lower filtration}
  \langle x,\, {_uf} \rangle = 0,
  \quad \forall \, u > w
\end{equation}
where for any loop word $u = \left [ i_1^{(d_1)} \dots\, i_k^{(d_k)} \right]$ we set:
\begin{equation}
\label{eqn:vf elements}
  _uf := f_{i_1,-d_1} \dots f_{i_k,-d_k}
\end{equation}

\medskip

\begin{proposition}
\label{prop:non-degenerate}
The restriction of the pairing \eqref{eqn:bialg pair affine} to the subspaces:
$$
  \UUp_{\leq w} \otimes \UUm^{\leq w} \longrightarrow \BQ(q)
$$
is still non-degenerate in the first factor, i.e.\ $\langle x, - \rangle = 0$ implies $x = 0$.
\end{proposition}

\medskip

\begin{proof}
Assume $x \in \UUp_{\leq w}$ has the property that:
\begin{equation}
\label{eqn:non-deg 1}
  \langle x, f_v \rangle = 0
\end{equation}
for any standard loop word $v \leq w$, and our goal is to show that $x = 0$. To this end, note that for
any loop word $v$ we have (by analogy with~\cite[Proposition 20]{L}):
\begin{equation}
\label{eqn:fv vs vf}
  f_v \in \sum_{u \geq v} \BQ(q) \cdot {_uf}
\end{equation}
Since $\langle x, {_uf} \rangle = 0$ for all $u > w$ by~(\ref{eqn:explicit lower filtration}), we conclude:
\begin{equation}
\label{eqn:non-deg 2}
  \langle x, f_v \rangle = 0
\end{equation}
for any loop word $v > w$. By Theorem~\ref{thm:PBW quantum loop}, the set $\{f_v|v \text{ standard loop word}\}$
is a basis of $\UUm$, so relations \eqref{eqn:non-deg 1} and \eqref{eqn:non-deg 2} imply that:
$$
  \left \langle x, \UUm \right \rangle = 0
$$
Thus $x = 0$ due to the non-degeneracy statement of Proposition~\ref{prop:non-degenerate qaff pairing}.
\end{proof}

\medskip


\subsection{}
\label{sub:dimension for affine pieces}

As a consequence of Proposition \ref{prop:non-degenerate}, we conclude that:
\begin{equation}
\label{eqn:leq}
  \dim \UUp_{\leq w} \,\leq\, \# \Big\{\text{standard loop words } \leq w\Big\}
\end{equation}
Note a slight imprecision in the inequality above: what we actually mean is that the dimension of the left-hand side in any
fixed degree $(\alpha,d)\in Q^+\times \BZ$ is less than or equal to the number of standard loop words $\leq w$ of degree $(\alpha,d)$
(the latter number is finite by Corollary \ref{cor:finitely many}). On the other hand, by the very definition of a good loop word,
we have:
\begin{equation}
\label{eqn:geq}
  \dim \UUp_{\leq w} \,=\, \# \Big\{\text{good loop words } \leq w\Big\}
\end{equation}

\medskip

\noindent
The following Proposition establishes the fact that we have equality in \eqref{eqn:leq}.

\medskip

\begin{proposition}
\label{prop:standard is good}
A loop word is standard if and only if it is good.
\end{proposition}

\medskip

\begin{proof}
Assume for the purpose of contradiction that there exists a good loop word $w$ which is not standard, and choose it
such that its degree $(\alpha,d) \in Q^+ \times \BZ$ has minimal $|\alpha|$. This minimality, combined with
Propositions~\ref{prop:stand via Lyndonstand} (see Remark~\ref{rem:generalization}) and~\ref{prop:good lyndon word},
implies that $w$ must be Lyndon. Therefore, we may write it as~\eqref{eqn:costandard factorization}:
$$
  w = \ell_1 \ell_2
$$
where $\ell_1 < w < \ell_2$ are Lyndon loop words. By Proposition \ref{prop:good}, $\ell_1$ and $\ell_2$ are good Lyndon
loop words, hence by the minimality of $|\alpha|$, standard Lyndon loop words. However, because of \eqref{eqn:leq} and
\eqref{eqn:geq}, there must exist a standard loop word $v < w$ with $\deg v = \deg w$. Then let us consider the canonical
factorization~\eqref{eqn:canonical factorization} $v=\ell'_1\dots\ell'_k$ where $\ell'_1 \geq \dots \geq \ell'_k$
are standard Lyndon loop words. Because:
$$
  \deg \ell_1 + \deg \ell_2 = \deg w = \deg v = \deg \ell_1'+ \dots + \deg \ell_k'
$$
Corollary~\ref{cor:convex several} implies that $\ell_1' \geq \ell_1$. However, the only way this is compatible with:
$$
  \ell_1\ell_2 = w > v = \ell_1' \dots \ell_k'
$$
is if $\ell_1' = \ell_1u$ for some loop word $u$ that satisfies:
\begin{equation}
\label{eqn:must contradict}
  \ell_2 > u\ell_2'\dots \ell_k' \qquad \text{and} \qquad
  \deg \ell_2 = \deg u + \deg \ell_2' + \dots +\deg \ell_k'
\end{equation}
Because $\ell_1'$ is standard, Proposition \ref{prop:factor standard} (see Remark~\ref{rem:generalization}) implies
that so is $u$. Therefore we may write $u = \ell_1''\dots \ell_{m}''$ for various standard Lyndon loop words
$\ell_1'' \geq \dots \geq \ell_m''$. Formula \eqref{eqn:must contradict} implies that $\ell_2 > u$,
so $\ell_2 > \ell_1'' \geq \dots \geq \ell_m''$. However, we also have $\ell_2 > w >v> \ell_2' \geq \dots \geq \ell_k'$,
and so \eqref{eqn:must contradict} contradicts Corollary \ref{cor:convex several}. Thus, any good loop word is standard.

\medskip

\noindent
For the converse, let us prove by induction on $|\alpha|$ that for any standard loop word $w$ of degree $(\alpha,d)$,
there exists a linear combination:
\begin{equation}
\label{eqn:written}
  \sum_{v \geq w} \text{coefficient} \cdot \wPhi(e_v) \, \in \, \BQ(q)^* \cdot w + \text{smaller words}
\end{equation}
for various coefficients in $\BQ(q)$ with $v$ being standard loop words, where we may further assume that all summands
have the same $Q^+\times \BZ$-degree $(\alpha,d)$.

\medskip

\begin{claim}
\label{claim:written}
If \eqref{eqn:written} holds for two loop words $w = \ell_1$ and $w' = \ell_2 \dots \ell_{k}$, where
$\ell_1 \geq \ell_2 \geq \dots \geq \ell_{k}$ are all standard Lyndon loop words, then \eqref{eqn:written}
also holds for the concatenation $ww'$.
\end{claim}

\medskip

\noindent
Let us first show how the Claim allows us to complete the proof of the Proposition. Since any standard loop word
can be written as $w = \ell_1 \dots \ell_k$ where $\ell_1 \geq \dots \geq \ell_k$ are standard Lyndon loop words,
then the Claim says that it suffices to prove \eqref{eqn:written} when $w = \ell$ is a standard Lyndon loop word.
To this end, let us write:
$$
  \wPhi(e_\ell) = c\cdot u + \sum_{v < u} \text{coefficient} \cdot v
$$
for some $c\in \BQ(q)^*$ and a loop word $u$. Since $u$ is the leading word, it must be good, hence standard.
Corollary~\ref{cor:convex several} implies that $u \geq \ell$. If $u=\ell$, then we have proved~\eqref{eqn:written}.
If $u > \ell$, then $u$ is a concatenation of standard Lyndon loop words of length less than that of $\ell$, to which
we may apply the induction hypothesis. According to the Claim, we may thus use \eqref{eqn:written} for $u$ to write:
$$
  \wPhi(e_\ell) - \text{coefficient} \cdot \wPhi(e_u) \, =\, \sum_{v < u} \text{coefficient} \cdot v
$$
By repeating this argument (finitely many times, due to Corollary \ref{cor:finitely many}) we either
establish~\eqref{eqn:written} for $w=\ell$ as wanted, or arrive at the following equality:
\begin{equation}
\label{eqn:impossible}
  \wPhi(e_\ell) - \sum_{v > \ell} \text{coefficient} \cdot \wPhi(e_v) \, =\, \sum_{v < \ell} \text{coefficient} \cdot v
\end{equation}
Since $\wPhi$ is injective and $\{e_v|v\ \mathrm{standard\ loop\ word}\}$ is a basis of $\UUp$ due to
Theorem \ref{thm:PBW quantum loop}, the left-hand side of~\eqref{eqn:impossible} is non-zero, hence so is the right-hand side.
This implies that there are good, hence standard, loop words of degree $\deg\ell$ which are $<\ell$. The latter contradicts
Corollary~\ref{cor:convex several}, and so~\eqref{eqn:impossible} is impossible.

\medskip

\noindent
Claim \ref{claim:written} follows immediately from the two facts below (assume $w,w',\ell_1,\dots,\ell_k$ are as
in the statement of the Claim):

\begin{enumerate}

\item
the largest word which appears in the shuffle product $w * w'$ is $ww'$

\item
$e_v e_{v'}$ is a linear combination of $e_t$'s with $t \geq ww'$, for all $v \geq w$ and $v' \geq w'$
satisfying $\deg v =\deg w$ and $\deg v'=\deg w'$

\end{enumerate}

\noindent
The first fact is proved as in \cite[Lemma 15]{L} (cf.\ our proof of Proposition~\ref{prop:good lyndon word}).
To prove the second fact, note that formula \eqref{eqn:q comm general affine twisted} (as we will see,
for any $(\alpha,d) \in \Delta^+ \times \BZ$, our $e_{\ell(\alpha,d)}$ will be a scalar multiple of the element
denoted by $\varpi(\sfe_{-(\alpha,-d)})$ later on) implies that for all standard Lyndon loop words $\ell < \ell'$,
we can write:
\begin{equation}
\label{eqn:linear combination}
  e_\ell e_{\ell'} = \text{a linear combination of } e_{\ell' \ell} \text{ and various } e_{m''_1 \dots m''_{t''}}
\end{equation}
with $\ell' > m_1'' \geq \dots \geq m_{t''}'' > \ell$ standard Lyndon loop words.
Consider the canonical factorizations \eqref{eqn:canonical factorization}:
$$
  v = m_1 \dots m_t \quad \text{and} \quad v' = m_1' \dots m'_{t'}
$$
where $m_1 \geq \dots \geq m_t$ and $m_1' \geq \dots \geq m'_{t'}$ are standard Lyndon loop words.
It is elementary to prove that $v \geq w$, $\deg v=\deg w$, and $w$ being Lyndon imply that either $m_1 > w$,
or that $v = w$. In the former case ($m_1 > w$), \eqref{eqn:linear combination} implies that (cf.\ the argument
in the proofs of Lemmas~\ref{lem:segmental pbw},~\ref{lem:segmental pbw loop}):
$$
  e_ve_{v'} = e_{m_1} \dots e_{m_t} e_{m_1'} \dots e_{m'_{t'}} = \text{a linear combination of }e_t\text{'s}
$$
for standard $t$ with the canonical factorization $m_1'' \dots m''_{t''}$ satisfying $m_1'' \geq m_1 > w$.

\noindent
A result of
Melan\c{c}on (\cite{M}), which states that two words with the canonical factorization \eqref{eqn:canonical factorization} are
in the relative order $>$ if the largest Lyndon words in their canonical factorizations are in the relative order $>$,
implies that $t > ww'$, as we needed to show. In the latter case ($v = w = \ell_1$), we have two more possible situations:

\begin{itemize}[leftmargin=*]

\item
if $\ell_1 \geq m_1'$, then $e_v e_{v'} = e_{vv'}$ and we are done since $vv' \geq ww'$
(as $v\geq w, v'\geq w'$ and the loop words $v,w$ are of the same length)

\item
if $m_i' > \ell_1 \geq m_{i+1}'$ for some $i \in \{1,\dots,t'\}$ (where $\ell_1 \geq m_{t'+1}'$ is vacuous), then \eqref{eqn:linear combination} implies that:
$$
  e_v e_{v'} = e_{\ell_1} e_{m_1'} \dots e_{m_{t'}'} = \text{a linear combination of }e_t\text{'s}
$$
where $t = m_1'' \dots m''_{i''} m'_{i+1} \dots m'_{t'}$ satisfies
$m_1'' \geq \dots \geq m_{i''}'' \geq m'_{i+1} \geq \dots \geq m'_{t'}$ and $m_1'' > \ell_1$.
Thus, the aforementioned result of Melan\c{c}on implies that $t > ww'$.
\end{itemize}
\end{proof}


\subsection{}

The results of the present Section amount to the proof of Theorem \ref{thm:main 1}.

\medskip

\begin{proof}[Proof of Theorem~\ref{thm:main 1}]
The statement about the homomorphism $\wPhi$ is proved in Subsection \ref{sub:affine shuffle}. The classification of
standard Lyndon loop words is accomplished in \eqref{eqn:associated word loop}. The construction of the root vectors
\eqref{eqn:root vectors intro 2} is done in Definition~\ref{def:loop quantum bracketing}. Finally, the PBW statement
\eqref{eqn:pbw intro loop} is the subject of Theorem~\ref{thm:PBW quantum loop}, whose proof will be completed in
the next Section.
\end{proof}

\medskip

\noindent
Computer experiments (in all types, but for a particular order of the simple roots) suggest that the generalization
of Lemma \ref{lem:minimal} to the loop case holds.

\medskip

\begin{conjecture}
\label{conj:minimal loop}
For any $(\alpha,d) \in \Delta^+ \times \BZ$, the leading word of $\wPhi(e_{\ell(\alpha,d)})$ is $\ell(\alpha,d)$.
Moreover, the word $\ell(\alpha,d)$ is the smallest good loop word of degree $(\alpha,d)$.
\end{conjecture}

\medskip


\section{Quantum affine and quantum loop: two presentations}
\label{sec:loop affine}

In the present Section, we will recall the general framework (due to Lusztig in the finite case,
and Beck and Damiani in the affine case) of PBW bases for quantum groups, from which we will deduce
Theorems \ref{thm:PBW quantum finite} and \ref{thm:PBW quantum loop}. The former of these will be immediate,
while the latter will require some work to connect quantum loop and quantum affine groups, and will require
the use of the results of Section \ref{sec:weyl lyndon}.

\medskip


\subsection{}
\label{sub:finite group}

Consider any convex order $\leq$ of the set of positive roots $\Delta^+$, as in Definition~\ref{def:convex}.
According to~\cite{P} (see also Remark~\ref{rem:longest element}) there is a unique reduced decomposition
$w_0=s_{i_{1-l}} \dots s_{i_0}$ of the longest element $w_0$ of the Weyl group $W$ such that the ordered set
$\alpha_{i_0} < s_{i_0}(\alpha_{i_{-1}}) < \dots < s_{i_0}\dots s_{i_{2-l}}(\alpha_{i_{1-l}})$
precisely recovers $(\Delta^+,\leq)$. To this choice, one may associate (\cite{Lu}) a collection of ``root vectors":
\begin{equation}
\label{eqn:root vectors}
  E_{\pm \beta} \in \uupm
\end{equation}
for all $\beta \in \Delta^+$, via the following formula for all $0\geq k>-l$:
\begin{equation}
\label{eqn:root via braid}
  \text{if } \beta = s_{i_0}\dots s_{i_{k+1}}(\alpha_{i_k}) \quad \text{then} \quad
  \begin{cases}
    E_\beta := T_{i_0}^{-1} \dots T_{i_{k+1}}^{-1}(e_{i_k}) \\
    E_{-\beta} := T_{i_0}^{-1} \dots T_{i_{k+1}}^{-1}(f_{i_k})
\end{cases}
\end{equation}
where $\{T_i\}_{i \in I}$ determine Lusztig's braid group action~\cite{Lu} on $\uu$ (cf.~\cite[\S8]{J}).
Then we have by~\cite{Lu} (cf.~\cite[\S8.24]{J}):
\begin{equation}
\label{eqn:pbw finite}
  \uupm \ =
  \bigoplus^{k\in \BN}_{\gamma_1 \leq \dots \leq \gamma_k \in \Delta^+}
    \BQ(q) \cdot E_{\pm \gamma_1} \dots E_{\pm \gamma_k}
\end{equation}
By analogy with~\eqref{eqn:pbw intro 1}, the formula above is called a PBW theorem for $\uupm$.
\medskip

\begin{remark}
(a) Due to~\cite[formula (9) of \S8.14]{J}, the PBW decompositions~(\ref{eqn:pbw finite}) for $\uup$ and $\uum$
are intertwined by an algebra automorphism $\omega$ of $\uu$:
\begin{equation}
\label{eqn:finite involution}
  \omega\colon e_i\mapsto f_i,\ f_i\mapsto e_i,\ \ph^{\pm 1}_i\mapsto \ph^{\mp 1}_i, \quad \forall\, i\in I
\end{equation}

\medskip

\noindent
(b) We note that formulas~(\ref{eqn:root via braid},~\ref{eqn:pbw finite}) differ slightly from~\cite[\S8.21, \S8.24]{J}
(and some other standard literature) in that the latter uses $T_i$ instead of $T_i^{-1}$ to define the root vectors, as well
the opposite order of $\Delta^+$ in the PBW theorem for $\uup$. To relate the exposition of~\cite{J} to ours,
recall the algebra anti-involution $\tau$ of~$\uu$:
\begin{equation}
\label{eqn: finite antiinvolution}
  \tau\colon e_i\mapsto e_i,\ f_i\mapsto f_i,\ \ph^{\pm 1}_i\mapsto \ph^{\mp 1}_i, \quad \forall\, i\in I
\end{equation}
It can be easily verified (\cite[formula (10) of \S8.14]{J}) that:
$$
  \tau\circ T_i\circ \tau = T_i^{-1}
$$
for any $i\in I$. Therefore,~\eqref{eqn:pbw finite} is obtained from~\cite[formula (3) of \S8.24]{J} applied
to the opposite reduced decomposition $w_0=w_0^{-1}=s_{i_0}s_{i_{-1}}\dots s_{i_{1-l}}$, followed by $\tau$.
Likewise,~\cite[formula (2) of \S8.24]{J} implies the opposite PBW decomposition:
\begin{equation}
\label{eqn:pbw finite opposite}
  \uupm \ = \bigoplus^{k\in \BN}_{\gamma_1 \geq \dots \geq \gamma_k \in \Delta^+}
  \BQ(q) \cdot E_{\pm \gamma_1} \dots E_{\pm \gamma_k}
\end{equation}

\medskip

\noindent
(c) Henceforth, we will use the following non-tautological equalities (\cite[\S8.20]{J}):
\begin{equation}
\label{eqn:simple root generators agree}
  E_{\alpha_i}=e_i,\quad  E_{-\alpha_i}=f_i, \quad \forall\, i\in I
\end{equation}
\end{remark}

\medskip


\subsection{}
\label{sub:root vectors recursive}

There is a way to construct the root vectors \eqref{eqn:root vectors}, up to scalar multiples,
without explicitly invoking the braid group action. Formula \eqref{eqn:pbw finite} entails the fact that any product
of $E_{\pm \alpha}$'s can be written as a sum of the ordered products. However, there is a restriction on the products
that may appear, as in~\cite[Proposition~7]{B} (which takes its origins in the formulas of~\cite{LS}):
\begin{equation}
\label{eqn:q comm general}
  E_{\pm \beta} E_{\pm \alpha} - q^{(\alpha, \beta)} E_{\pm \alpha} E_{\pm \beta} \ \in
  \mathop{\bigoplus^{k\in \BN}_{\alpha < \gamma_1 \leq \dots \leq \gamma_k < \beta}}_
    {\gamma_1 + \dots + \gamma_k = \alpha + \beta}
  \BQ(q) \cdot E_{\pm \gamma_1} \dots E_{\pm \gamma_k}
\end{equation}
for any positive roots $\alpha < \beta$. If we assume that $\alpha+\beta$ is also a positive root, and that its
decomposition as the sum of $\alpha$ and $\beta$ is \underline{minimal} in the sense that:
\begin{equation}
\label{eqn:minimal}
  \not \exists \ \alpha',\beta' \in \Delta^+ \quad \text{s.t.} \quad
  \alpha < \alpha' < \beta' < \beta \quad \text{and} \quad \alpha+\beta = \alpha'+\beta'
\end{equation}
then the sum in the right-hand side of \eqref{eqn:q comm general} consists of a single term:\footnote{Indeed,
assume that there existed a decomposition $\alpha+\beta = \gamma_1+ \dots + \gamma_k$ for positive roots
$\alpha < \gamma_1 \leq \dots \leq \gamma_k < \beta$ with $k > 1$. Then, as shown in the proof of
Proposition~\ref{prop:convex loop}, we can modify the decomposition by clumping some of the $\gamma_r$'s together
so as to ensure $k=2$ (the resulting two roots are still bounded by $\alpha$ and $\beta$, due to the convexity).
This would contradict \eqref{eqn:minimal}.}
\begin{equation}
\label{eqn:q comm}
  [E_{\pm \beta}, E_{\pm \alpha}]_q = E_{\pm \beta} E_{\pm \alpha} - q^{(\alpha,\beta)} E_{\pm \alpha} E_{\pm \beta}
  \in \BQ(q)^* \cdot E_{\pm (\alpha+\beta)}
\end{equation}
(the coefficient of $E_{\pm (\alpha+\beta)}$ in the right-hand side actually lies in $\BZ[q,q^{-1}]^*$, as shown
in~\cite[Theorem 6.7(a)]{Lu2}). Therefore, one can recover the root vectors $\{E_\beta\}_{\beta\in \Delta^+}$
(resp.\ $\{E_{-\beta}\}_{\beta\in \Delta^+}$) as iterated $q$-commutators of the $e_i$'s (resp.\ $f_i$'s) times scalars,
based solely on the chosen convex order of the set of positive roots $\Delta^+$.

\medskip

\noindent
We conclude this Subsection with another important corollary of formula~\eqref{eqn:q comm general}.

\medskip

\begin{lemma}
\label{lem:segmental pbw}
For any $\alpha\leq \beta \in \Delta^+$, let $U^\pm_q([\alpha,\beta])$ be the subalgebra of $\uupm$ generated by
$\{E_{\pm \gamma}|\alpha\leq \gamma\leq \beta\}$.
Then:
\begin{equation}
\label{eqn:pbw finite segment}
  U^\pm_q([\alpha,\beta]) \ =
  \bigoplus^{k\in \BN}_{\beta \geq \gamma_1 \geq \dots \geq \gamma_k\geq \alpha \in \Delta^+}
    \BQ(q) \cdot E_{\pm \gamma_1} \dots E_{\pm \gamma_k}
\end{equation}
as well as:
\begin{equation}
\label{eqn:pbw finite segment opposite}
  U^\pm_q([\alpha,\beta]) \ =
  \bigoplus^{k\in \BN}_{\alpha\leq  \gamma_1 \leq \dots \leq \gamma_k\leq \beta \in \Delta^+}
    \BQ(q) \cdot E_{\pm \gamma_1} \dots E_{\pm \gamma_k}
\end{equation}
\end{lemma}

\medskip

\begin{proof}
Let $\alpha=\gamma_1<\gamma_2<\dots<\gamma_s=\beta$ be a complete list of positive roots $\gamma\in \Delta^+$ satisfying
$\alpha\leq \gamma\leq \beta$. First, let us note that the ordered monomials featuring in the right-hand sides
of~(\ref{eqn:pbw finite segment}) and~(\ref{eqn:pbw finite segment opposite}) are linearly independent since they already
appeared as part of the basis of $\uupm$ in~(\ref{eqn:pbw finite opposite}) and~(\ref{eqn:pbw finite}), respectively.
Therefore, to prove~\eqref{eqn:pbw finite segment} (resp.~\eqref{eqn:pbw finite segment opposite}), it suffices to show
that any product $E_{\pm \gamma_{i_1}} \dots E_{\pm \gamma_{i_k}}$ can be reordered as a linear combination of such products
with $i_1 \geq \dots \geq i_k$ (resp.\ $i_1 \leq \dots \leq i_k$). We prove this by induction primarily on $M - m$
(where $M = \max i_a$ and $m = \min i_a$) and then secondarily on the total number of times $M$ and $m$ appear in the sequence
$i_1, \dots, i_k$. Indeed, formula~\eqref{eqn:q comm general} allows to move all the $E_{\pm \gamma_M}$'s to the left and all
the $E_{\pm \gamma_m}$'s to the right (resp.\ all the $E_{\pm \gamma_M}$'s to the right and all the $E_{\pm \gamma_m}$'s
to the left), at the cost of gaining extra products $E_{\pm \gamma_{j_1}} \dots E_{\pm \gamma_{j_l}}$ to which
the induction hypothesis applies.
\end{proof}

\medskip


\subsection{}
\label{sub:proof of finite PBW for Lyndon}

We shall now see that Theorem \ref{thm:PBW quantum finite} is in fact equivalent to the PBW decomposition~(\ref{eqn:pbw finite})
applied to the convex order~(\ref{eqn:induces}) of $\Delta^+$, see Proposition~\ref{prop:finite convexity}.

\medskip

\begin{proof}[Proof of Theorem \ref{thm:PBW quantum finite}]
Consider the anti-involution $\varpi$ of $\uu$ defined via:
$$
  \varpi\colon e_{i}\mapsto f_{i},\ f_{i}\mapsto e_{i},\ \ph^\pm_{i}\mapsto \ph^{\pm}_{i}
$$
for $i\in I$; thus $\varpi$ is a composition of~(\ref{eqn:finite involution},~\ref{eqn: finite antiinvolution}).
Applying $\varpi$ to~\eqref{eqn:pbw finite}, we obtain:
\begin{equation}
\label{eqn:pbw finite double-opposite}
  \uupm \ =
  \bigoplus^{k\in \BN}_{\gamma_1 \geq \dots \geq \gamma_k \in \Delta^+}
    \BQ(q) \cdot\varpi (E_{\mp \gamma_1}) \dots \varpi(E_{\mp \gamma_k})
\end{equation}
We claim that Theorem~\ref{thm:PBW quantum finite} follows from~(\ref{eqn:pbw finite double-opposite}).
To this end, it suffices to show:
\begin{equation}
\label{eqn:coincidence}
  e_{\ell(\alpha)} \in \BQ(q)^* \cdot \varpi(E_{-\alpha})
    \qquad \mathrm{and} \qquad
  f_{\ell(\alpha)} \in \BQ(q)^* \cdot \varpi(E_{\alpha})
\end{equation}
for any $\alpha\in \Delta^+$, where $\ell$ is the bijection \eqref{eqn:associated word}. We prove~(\ref{eqn:coincidence})
by induction on the height of $\alpha$. The base case $\alpha=\alpha_i$ (with $i\in I$) is immediate, due
to~\eqref{eqn:simple root generators agree}:
$$
  e_{[i]} = e_i = \varpi(f_i)=\varpi(E_{-\alpha_i}) \quad \mathrm{and} \quad
  f_{[i]} = f_i = \varpi(e_i)=\varpi(E_{\alpha_i})
$$
For the induction step, consider the factorization \eqref{eqn:costandard factorization} of $\ell = \ell(\alpha)$:
$$
  \ell = \ell_1\ell_2
$$
Since factors of standard words are standard, we have $\ell_1 = \ell(\gamma_1)$ and $\ell_2 = \ell(\gamma_2)$ for some
$\gamma_1,\gamma_2 \in \Delta^+$ such that $\alpha = \gamma_1 + \gamma_2$. By the induction hypothesis, we have:
$$
  e_{\ell_k} \in \BQ(q)^* \cdot \varpi(E_{-\gamma_k})
    \qquad \mathrm{and} \qquad
  f_{\ell_k} \in \BQ(q)^* \cdot \varpi(E_{\gamma_k})
$$
for $k \in \{1,2\}$. However, by (the finite counterpart of) Proposition \ref{prop:lyndon is minimal} and the
definition~(\ref{eqn:induces}), we note that $\gamma_1 < \alpha < \gamma_2$ is a minimal decomposition in the
sense of \eqref{eqn:minimal}. Therefore, comparing \eqref{eqn:quantum bracketing lyndon} (and its $f$-analogue)
with \eqref{eqn:q comm}, we obtain:
\begin{equation*}
\begin{split}
  & e_\ell = [e_{\ell_1}, e_{\ell_2}]_q \in
    \BQ(q)^* \cdot \varpi([E_{-\gamma_2}, E_{-\gamma_1}]_q) = \BQ(q)^* \cdot \varpi(E_{-\alpha})\\
  & f_\ell = [f_{\ell_1}, f_{\ell_2}]_q \in
    \BQ(q)^* \cdot \varpi([E_{\gamma_2}, E_{\gamma_1}]_q) = \BQ(q)^* \cdot \varpi(E_{\alpha})
\end{split}
\end{equation*}
as we needed to prove.
\end{proof}

\medskip


\subsection{}

We would now like to consider the loop version of the construction above, with the goal of proving
Theorem \ref{thm:PBW quantum loop}. However, here we run into a technical snag, in that one only has a version
of Lusztig's braid group action available in the Drinfeld-Jimbo  affine quantum group. Therefore, we will recall
the construction of PBW bases of affine quantum groups of \cite{B,D} and bridge the gap between them and
the sought after bases of quantum loop groups (in the new Drinfeld presentation).

\medskip

\noindent
Henceforth, we will use the notations of Subsection~\ref{sub:affine Weyl}.

\medskip

\begin{definition}
\label{def:quantum affine}
Let $\VV$ be as in Definition \ref{def:finite quantum group}, but using $\wI$ instead of $I$.
\end{definition}

\medskip

\noindent
Letting $\VVp, \VVo,\VVm$ be the subalgebras generated by the $e_i$'s, $\ph_i^{\pm 1}$'s, $f_i$'s, respectively
(with $i \in \wI$), we obtain a triangular decomposition analogous to \eqref{eqn:triangular finite}:
\begin{equation}
\label{eqn:triangular affine}
  \VV = \VVp \otimes \VVo \otimes \VVm
\end{equation}
We will also consider the following sub-bialgebras of $\VV$:
\begin{align*}
  & \VVg = \VVp \otimes \VVo \\
  & \VVl = \VVo \otimes \VVm
\end{align*}
The algebra $\VV$ is $\wQ\simeq Q\times \BZ$-graded via:
\begin{align*}
  & \deg e_0 = \alpha_0 = (-\theta, 1) & & \deg e_i = \alpha_i = (\alpha_i,0) \\
  & \deg f_0 = -\alpha_0 = (\theta, -1) & & \deg f_i = -\alpha_i = (-\alpha_i,0) \\
  & \deg \ph_0 = 0 =(0,0) & & \deg \ph_i = 0 = (0,0)
\end{align*}
for $i\in I$, where $\theta$ is the highest root of $\Delta^+$, and $\wQ$ is identified with $Q\times \BZ$
via~\eqref{eqn:finite vs affine roots}. Sending $e_i \mapsto e_i$, $f_i \mapsto f_i$, $\ph^{\pm 1}_i \mapsto \ph^{\pm 1}_i$
for $i \in I$ yields an algebra homomorphism:
$$
  \uu \hooklongrightarrow \VV
$$
of $Q \times \BZ$-graded algebras, where the $\BZ$-grading on $\uu$ is set to be trivial.
Finally, we observe the fact that the element:
\begin{equation}
\label{eqn:central c}
  C = \ph_0 \prod_{i \in I} \ph_i^{\theta_i}
\end{equation}
is central in $\VV$, where the positive integers $\{\theta_i\}_{i\in I}$ were introduced in~(\ref{eqn:labels}).
Note that $C$ of~\eqref{eqn:central c} is to $c$ of~(\ref{eqn:classical central}) as $\{\ph_i\}_{i\in \wI}$ are
to $\{h_i\}_{i\in \wI}$.

\medskip


\subsection{}
\label{sub:pbw}

Let us now recall, following~\cite{B}, the affine version of the construction of the root vectors from
Subsection \ref{sub:finite group}. Following Subsection~\ref{sub:Lyndon via affWeyl}, pick any
$\mu \in P^\vee$ such that $(\alpha_i,\mu) > 0$ for all $i \in I$ and consider $\wmu=1\ltimes \mu \in \weW$.
Let $l = l(\wmu) = (2\rho,\mu)$ be the length of $\wmu$ (Proposition~\ref{prop:length for lattice})
and consider any reduced decomposition:
$$
  \widehat{\mu} = \tau s_{i_{1-l}}s_{i_{2-l}} \dots s_{i_0}
$$
as in~\eqref{eqn:reduced-rho} with (a uniquely determined) $\tau\in \CT$. Following~\eqref{eqn:tau-twisted sequence},
we extend $\{i_k|-l<k\leq 0\}$ to a ($\tau$-quasiperiodic) bi-infinite sequence $\{i_k\}_{k\in \BZ}$ via:
$$
  i_{k+l} = \tau(i_k), \qquad \forall\, k \in \BZ
$$
By analogy with~\eqref{eqn:beta-roots}, we may construct the following set of positive affine roots:
\begin{equation}
\label{eqn:affine real roots}
  \tbeta_k =
  \begin{cases}
    s_{i_1} s_{i_2} \dots s_{i_{k-1}}(\alpha_{i_k}) &\text{if } k > 0 \\
    s_{i_0} s_{i_{-1}} \dots s_{i_{k+1}}(\alpha_{i_k}) &\text{if } k \leq 0
  \end{cases}
\end{equation}
which are related to the roots $\beta_k$ of~\eqref{eqn:beta-roots} via:
\begin{equation}
\label{eqn:beta vs tbeta}
  \tbeta_k=
  \begin{cases}
    -\beta_k  &\text{if } k > 0 \\
    \beta_k &\text{if } k \leq 0
  \end{cases}
\end{equation}
Following~\cite{B}, we shall order those roots as follows:
\begin{equation}
\label{eqn:Becks order}
  \tbeta_0<\tbeta_{-1}<\tbeta_{-2}<\tbeta_{-3}<\dots<\tbeta_4<\tbeta_3<\tbeta_2<\tbeta_1
\end{equation}

\medskip

\begin{remark}
\label{rem:tbeta-convexity}
Formula \eqref{eqn:affine real roots} provides all real positive roots of the affine root system:
\begin{equation}
\label{eqn:real positive}
  \widehat{\Delta}^{\mathrm{re},+} =
  \Big\{ \Delta^+ \times \BZ_{\geq 0} \Big\} \sqcup \Big\{ \Delta^- \times \BZ_{>0} \Big\} \subset \widehat{\Delta}^+
\end{equation}
Furthermore,~\eqref{eqn:Becks order} induces convex orders on the corresponding halves:
\begin{equation*}
  \Delta^+ \times \BZ_{\geq 0}=\Big\{ \tbeta_0 < \tbeta_{-1} <\tbeta_{-2}<\dots \Big\}
    \quad \mathrm{and} \quad
  \Delta^- \times \BZ_{> 0}=\Big\{ \dots< \tbeta_3 < \tbeta_2 < \tbeta_1 \Big\}
\end{equation*}
To have a complete theory, one also needs to deal with the imaginary roots:
$$
  \widehat{\Delta}^{\mathrm{im},+} = \Big\{ 0 \times \BZ_{>0} \Big\} \subset \widehat{\Delta}^+
$$
but they will not feature in the present paper.
\end{remark}

\medskip


\subsection{}
\label{sub:affine root vectors}

We may define the ``root vectors'':
\begin{equation}
\label{eqn:affine root vectors}
  E_{\pm \tbeta} \in \VVpm
\end{equation}
for all $\tbeta \in \wDelta^{\mathrm{re},+}$ of~(\ref{eqn:real positive}) via the following analogue of \eqref{eqn:root via braid}:
\begin{equation}
\label{eqn:affine root generators}
  E_{\tbeta_k} =
  \begin{cases}
    T_{i_1} \dots T_{i_{k-1}} (e_{i_k}) &\text{if } k > 0 \\
    T_{i_0}^{-1} \dots T_{i_{k+1}}^{-1} (e_{i_k}) &\text{if } k \leq 0
  \end{cases}
\end{equation}
and
\begin{equation}
\label{eqn:affine negative root generators}
  E_{-\tbeta_k} =
  \begin{cases}
    T_{i_1} \dots T_{i_{k-1}} (f_{i_k}) &\text{if } k > 0 \\
    T_{i_0}^{-1} \dots T_{i_{k+1}}^{-1} (f_{i_k}) &\text{if } k \leq 0
  \end{cases}
\end{equation}
where $\{T_i\}_{i \in \wI}$ determine Lusztig's affine braid group action~\cite{Lu} on $\VV$.

\begin{remark}
\label{rem:comparison to Beck}
(a)  The above construction applied to $\wmu\in \weW$ is equivalent to that of~\cite{B} applied to
$\widehat{x}\in \wW$ for a multiple $x=r\mu\in Q^\vee$ with $r\in \BN$, see Remark~\ref{rem:damiani vs beck}.

\medskip

\noindent
(b) We also note that~\cite{B} defines root vectors $E_{-\tbeta}\in \VVm$ for $\tbeta \in \wDelta^{\mathrm{re},+}$ via:
\begin{equation}
\label{eqn:negative via positive}
  E_{-\tbeta}:=\Omega(E_{\tbeta})
\end{equation}
where the $\BQ$-algebra anti-involution $\Omega$ of $\VV$ is defined via:
\begin{equation}
\label{eqn:anti-involution}
  \Omega\colon e_i\mapsto f_i,\ f_i\mapsto e_i,\ \ph^{\pm 1}_i\mapsto \ph_i^{\mp 1},\ q\mapsto q^{-1},
  \quad \forall\, i\in \wI
\end{equation}
Formulas~\eqref{eqn:affine negative root generators} and~\eqref{eqn:negative via positive} agree,
as $\Omega$ commutes with the affine braid group action:
\begin{equation}
\label{eqn:omega vs T}
  \Omega\circ T_i = T_i\circ \Omega, \quad \forall\, i\in \wI
\end{equation}
\end{remark}

\medskip

\noindent
Due to~\cite[Proposition 7]{B}, the root vectors satisfy the natural analogue of~\eqref{eqn:q comm general}:\footnote{See
Remark~\ref{rem:opposite} for the unexpected ordering of the $\tgamma$'s when the sign $\pm$ is $-$.}
\begin{equation}
\label{eqn:q comm general affine}
  E_{\pm \tbeta} E_{\pm \talpha} - q^{(\talpha, \tbeta)} E_{\pm \talpha} E_{\pm \tbeta} \ \in
  \mathop{\bigoplus^{k\in \BN}_{\talpha < \tgamma_1 \leq \dots \leq \tgamma_k < \tbeta}}_
    {\tgamma_1 + \dots + \tgamma_k=\talpha+\tbeta}
  \BQ(q) \cdot E_{\pm \tgamma_1} \dots E_{\pm \tgamma_k}
\end{equation}
for any real positive affine roots $\talpha < \tbeta$ which both belong to either $\Delta^+ \times \BZ_{\geq 0}$ or
$\Delta^- \times \BZ_{>0}$. Therefore, due to the convexity of the corresponding orders of $\Delta^+ \times \BZ_{\geq 0}$
or $\Delta^- \times \BZ_{>0}$ (Remark~\ref{rem:tbeta-convexity}), we also have the following analogue of~\eqref{eqn:q comm}:
\begin{equation}
\label{eqn:q comm affine}
  [E_{\pm \tbeta}, E_{\pm \talpha}]_q = E_{\pm \tbeta} E_{\pm \talpha} - q^{(\talpha, \tbeta)} E_{\pm \talpha} E_{\pm \tbeta}
  \in \BQ(q)^* \cdot E_{\pm (\talpha + \tbeta)}
\end{equation}
for any real positive affine roots $\talpha<\tbeta$ as above, which have the additional property that $\talpha + \tbeta$ is
a positive affine root whose decomposition as the sum of $\talpha$ and $\tbeta$ is \underline{minimal} in the sense that:
\begin{equation}
\label{eqn:minimal loop}
  \not \exists \ \talpha',\tbeta' \in \wDelta^{\mathrm{re},+} \quad \text{s.t.} \quad
  \talpha < \talpha' < \tbeta' < \tbeta \quad \text{and} \quad \talpha + \tbeta = \talpha' + \tbeta'
\end{equation}

\medskip

\begin{remark}
\label{rem:gavarini result}
According to~\cite[Theorem 4.8]{Ga}, there is an explicit subring $R$ of $\BQ(q)$ which admits a $q=1$ specialization,
such that the Lusztig $R$-form $\mathfrak{U}_R(\widehat{\fg})$ of $\VV$ admits a natural PBW basis. Moreover, the $q=1$
specialization gives rise to:
\begin{equation}
\label{eqn:classical limit}
  \mathfrak{U}_R(\widehat{\fg})/(q-1) \ \iso \ U(\widehat{\fg}) \qquad \mathrm{with} \qquad
  E_{\pm \tbeta}\mapsto E_{\pm \tbeta}
\end{equation}
the latter denoting Chevalley generators of $\widehat{\fg}$, see~\cite[formulas (5.5, 5.7)]{Ga}.
This implies that $[E_{\pm \tbeta}, E_{\pm \talpha}]_q \in R^* \cdot E_{\pm (\talpha + \tbeta)}$ under the same assumptions
as in~(\ref{eqn:q comm affine}). A more detailed analysis of~\cite{Ga} shows that for real root vectors one can replace $R$
with $\BZ[q,q^{-1}]$, hence the following refinement of~\eqref{eqn:q comm affine}, under the same assumptions:
\begin{equation}
\label{eqn:q comm affine refined}
  [E_{\pm \tbeta}, E_{\pm \talpha}]_q \in \BZ[q,q^{-1}]^* \cdot E_{\pm(\talpha + \tbeta)}
\end{equation}
\end{remark}

\medskip

\noindent
As a consequence of \eqref{eqn:q comm affine}, we obtain the following.

\medskip

\begin{corollary}
\label{cor:generating set for DJ root generators}
(a) The root vectors $\{E_{\pm \tbeta}\}_{\tbeta\in \Delta^+\times \BZ_{\geq 0}}$ can be obtained (up to non-zero scalars)
as iterated $q$-commutators of the root vectors:
$$
  \{E_{\pm(\alpha_i,d)}\}_{i\in I}^{d\geq 0}
$$
(b) The root vectors $\{E_{\pm \tbeta}\}_{\tbeta\in \Delta^-\times \BZ_{> 0}}$ can be obtained (up to non-zero scalars)
as iterated $q$-commutators of the root vectors:
$$
  \{E_{\pm (-\alpha_i,d)}\}_{i\in I}^{d>0} \qquad \text{and} \qquad \{E_{\pm(-\alpha,1)}\}_{\alpha\in \Delta^+}
$$
(c) The root vectors $\{E_{\pm (-\alpha,1)}\}_{\alpha\in \Delta^+}$ can be obtained (up to non-zero scalars)
as iterated $q$-commutators of the root vectors:
$$
  \{E_{\pm (\alpha_i,0)}\}_{i\in I} \qquad \text{and} \qquad E_{\pm (-\theta,1)}
$$
where $\theta \in \Delta^+$ denotes the highest root as before.
\end{corollary}

\medskip

\noindent
Parts (a) and (b) follow readily from the combinatorics of $\Delta^\pm$.
Part (c) follows from the convexity~\cite[Corollary 4]{B} of the entire PBW basis of $\VVpm$.

\medskip


\subsection{}

We have the following analogue of Lemma~\ref{lem:segmental pbw}.

\medskip

\begin{lemma}
\label{lem:segmental pbw loop}
For all $0<s\leq r$ or $s\leq r\leq 0$, let $U^\pm_q([s,r])$ be the subalgebra of $\VVpm$ generated by
$\{E_{\pm \tbeta_k}|s\leq k \leq r\}$. Then:
\begin{equation}
\label{eqn:pbw finite segment loop}
  U^\pm_q([s,r]) \ =
  \bigoplus_{n_s,n_{s+1},\dots,n_{r-1},n_r\in \BZ_{\geq 0}}
    \BQ(q) \cdot E_{\pm \tbeta_s}^{n_s} E_{\pm \tbeta_{s+1}}^{n_{s+1}} \dots E_{\pm \tbeta_{r-1}}^{n_{r-1}} E_{\pm \tbeta_r}^{n_r}
\end{equation}
as well as:
\begin{equation}
\label{eqn:pbw finite segment opposite loop}
  U^\pm_q([s,r]) \ =
  \bigoplus_{n_s,n_{s+1},\dots,n_{r-1},n_r\in \BZ_{\geq 0}}
    \BQ(q) \cdot E_{\pm \tbeta_r}^{n_r} E_{\pm \tbeta_{r-1}}^{n_{r-1}} \dots E_{\pm \tbeta_{s+1}}^{n_{s+1}} E_{\pm \tbeta_s}^{n_s}
\end{equation}
\end{lemma}

\medskip

\begin{remark}
\label{rem:opposite}
We should note right away that Lemma~\ref{lem:segmental pbw loop} has been implicitly used in~\eqref{eqn:q comm general affine},
since applying $\Omega$ of~\eqref{eqn:anti-involution} to~\cite[Proposition 7]{B} one actually obtains:
\begin{equation}
\label{eqn:q comm general affine alternative}
  E_{-\tbeta} E_{- \talpha} - q^{(\talpha, \tbeta)} E_{- \talpha} E_{- \tbeta} \ \in
  \mathop{\bigoplus_{\tbeta > \tgamma_1 \geq \dots \geq \tgamma_k > \talpha}}_
    {\tgamma_1 + \dots + \tgamma_k=\talpha+\tbeta}
  \BQ(q) \cdot E_{- \tgamma_1} \dots E_{- \tgamma_k}
\end{equation}
under the same assumptions as in~(\ref{eqn:q comm general affine}). Therefore, we need the equivalence of
\eqref{eqn:pbw finite segment loop} and \eqref{eqn:pbw finite segment opposite loop} to
convert~\eqref{eqn:q comm general affine alternative} into~\eqref{eqn:q comm general affine}.
\end{remark}

\medskip

\begin{proof}[Proof of Lemma~\ref{lem:segmental pbw loop}]
The fact that the ordered monomials featuring in the right-hand sides of~(\ref{eqn:pbw finite segment loop})
or~(\ref{eqn:pbw finite segment opposite loop}) span $U^\pm_q([s,r])$ is proved exactly as in our proof of
Lemma~\ref{lem:segmental pbw}. Meanwhile, their linear independence follows from the usual PBW theorem for $U(\hg)$
in view of~\eqref{eqn:classical limit}.
\end{proof}

\medskip

\noindent
We shall also need the limit cases of~(\ref{eqn:pbw finite segment loop},~\ref{eqn:pbw finite segment opposite loop})
when $(s,r)=(-\infty,0) \ \mathrm{or}\ (1,+\infty)$. To this end, let $U^\pm_q(+\infty)$ and $U^\pm_q(-\infty)$ denote
the subalgebras of $\VVpm$ generated by $\{E_{\pm \tbeta_k}|k\geq 1\}$ and $\{E_{\pm \tbeta_k}|k\leq 0\}$, respectively.
In accordance with~(\ref{eqn:pbw finite segment loop},~\ref{eqn:pbw finite segment opposite loop}), each of these
subalgebras admits a pair of opposite PBW decompositions:
\begin{equation}
\label{eqn:quarter 01}
\begin{split}
  & U^\pm_q(+\infty) \ = \\
  & \qquad \
    \mathop{\bigoplus_{n_1,n_2,\dots \in \BZ_{\geq 0}}}_{n_1+n_2+\dots < \infty}
    \BQ(q) \cdot E_{\pm \tbeta_1}^{n_1} E_{\pm \tbeta_2}^{n_2} \dots \ =
  \mathop{\bigoplus_{n_1,n_2,\dots \in \BZ_{\geq 0}}}_{n_1+n_2+\dots < \infty}
    \BQ(q) \cdot \dots E_{\pm \tbeta_2}^{n_2} E_{\pm \tbeta_1}^{n_1}
\end{split}
\end{equation}
\begin{equation}
\label{eqn:quarter 02}
\begin{split}
  & U^\pm_q(-\infty) \ = \\
  & \mathop{\bigoplus_{n_{0},n_{-1},\dots \in \BZ_{\geq 0}}}_{n_0+n_{-1}+\dots < \infty}
    \BQ(q) \cdot E_{\pm \tbeta_0}^{n_0} E_{\pm \tbeta_{-1}}^{n_{-1}} \dots \ =
  \mathop{\bigoplus_{n_{0},n_{-1},\dots \in \BZ_{\geq 0}}}_{n_0+n_{-1}+\dots < \infty}
    \BQ(q) \cdot \dots E_{\pm \tbeta_{-1}}^{n_{-1}} E_{\pm \tbeta_{0}}^{n_{0}}
\end{split}
\end{equation}

\medskip


\subsection{}
\label{sub:cop}

Following Damiani~\cite{D2}, we shall now recall the behaviour of the real root vectors~\eqref{eqn:affine root vectors}
with respect to the Drinfeld-Jimbo coproduct $\Delta$ of the affine quantum group (i.e.\ the construction of Subsection
\ref{sub:bialgebra DrJim} applied to $\wI$ instead of $I$). Following the notations of Lemma~\ref{lem:segmental pbw loop},
\emph{loc.\ cit.}\ defines the subalgebra $U^\pm_q(r)$ of $\VVpm$ via:
$$
  U^\pm_q(r)=
  \begin{cases}
    U^\pm_q([r,0]) &\text{for } r\leq 0 \\
    U^\pm_q([1,r]) &\text{for } r\geq 1
  \end{cases}
$$
Henceforth, given a homogeneous element $z$ of degree $\sum_{i\in \wI} r_i\alpha_i\in \wQ$, we set:
\begin{equation}
\label{eqn:Cartan phi}
  \ph_{\deg(z)} :=  \ph_{\sum_{i\in \wI} r_i\alpha_i} = \prod_{i\in \wI} \ph_i^{r_i} \, \in \VVo
\end{equation}
According to~\cite[Proposition 7.1.2]{D2} we have:
\begin{equation}
\label{eqn:Damiani positive}
  \Delta\left(E_{\tbeta_r}\right) =
  \ph_{\tbeta_r}\otimes E_{\tbeta_r} + E_{\tbeta_r}\otimes 1 +  \sum \ph_{\deg(y)}\, x\otimes y
\end{equation}
where the last sum is vacuous for $r=0,-1$, while otherwise it is restricted by:
\begin{equation}
\label{eqn:Damiani positive restrictions}
\begin{split}
  & x \in U^+_q(r-1) \ \text{ and } \ y \in \VVp \quad \mathrm{for}\ r>1 \\
  & x \in \VVp  \ \text{ and } \ y \in U^+_q(r+1) \quad \mathrm{for}\ r<0
\end{split}
\end{equation}
Note that the anti-involution $\Omega$ of~\eqref{eqn:anti-involution} intertwines $\Delta$ and its opposite $\Delta^{\text{op}}$:
$$
  \Delta\circ \Omega = (\Omega\otimes \Omega)\circ \Delta^{\text{op}}
$$
Therefore, applying $\Omega$ to~(\ref{eqn:Damiani positive},~\ref{eqn:Damiani positive restrictions}), we obtain:
\begin{equation}
\label{eqn:Damiani negative}
  \Delta\left(E_{-\tbeta_r}\right) =
  1\otimes E_{-\tbeta_r} + E_{-\tbeta_r}\otimes \ph_{-\tbeta_r} + \sum y\otimes \ph_{\deg(y)}\, x
\end{equation}
where the last sum is vacuous for $r=0,-1$, while otherwise it is restricted by:
\begin{equation}
\label{eqn:Damiani negative restrictions}
\begin{split}
  & x \in U^-_q(r-1) \ \text{ and } \ y \in \VVm \quad \mathrm{for}\ r>1 \\
  & x \in \VVm  \ \text{ and } \ y \in U^-_q(r+1) \quad \mathrm{for}\ r<0
\end{split}
\end{equation}

\medskip


\subsection{}

We will henceforth specialize the discussion of Subsections \ref{sub:pbw} - \ref{sub:cop} to:
$$
  \mu=\rho^\vee
$$
Our interest in the Drinfeld-Jimbo affine quantum groups of Definition \ref{def:quantum affine} is motivated by
the following connection with Drinfeld's new presentation of quantum loop groups of Definition~\ref{def:quantum loop}
(due to~\cite{B,B2,D3}):\footnote{The inverse of~\eqref{eqn:two presentations} was provided (without a proof)
earlier in~\cite[Theorem 3]{Dr2}.}

\medskip

\begin{theorem}
\label{thm:two presentations}
There exists an algebra isomorphism:
\begin{equation}
\label{eqn:two presentations}
  \UU \ \iso \ \VV/(C-1)
\end{equation}
determined by the following assignment for all $i \in I$ and $d \in \BZ$:
\begin{equation}
\label{eqn:Beck isomorphism}
\begin{split}
  & e_{i,d} \mapsto
  \begin{cases}
    o(i)^d E_{(\alpha_i,d)} &\text{if } d\geq 0 \\
    -o(i)^d E_{(\alpha_i,d)}\ph_i^{-1} &\text{if } d<0
  \end{cases} \\
  &  f_{i,d} \mapsto
  \begin{cases}
    -o(i)^d \ph_i E_{(-\alpha_i,d)} &\text{if } d>0 \\
    o(i)^d E_{(-\alpha_i,d)} &\text{if } d\leq 0
  \end{cases}
\end{split}
\end{equation}
where $o\colon I\to \{\pm 1\}$ is a map satisfying $o(i)o(j)=-1$ whenever $a_{ij}<0$.
\end{theorem}

\medskip

\begin{proof}
The isomorphism~\eqref{eqn:two presentations} was proved in~\cite[Theorem 4.7]{B2} with respect to
the following seemingly different formula:
\begin{equation}
\label{eqn:Beck original formulas}
  e_{i,d} \mapsto o(i)^d T_{\widehat{\omega^\vee_i}}^{-d}(e_i), \qquad
  f_{i,d} \mapsto o(i)^d T_{\widehat{\omega^\vee_i}}^d (f_i)
\end{equation}
Here, the aforementioned action of the affine braid group on $\VV$ has been extended to
the extended affine braid group by adding automorphisms $\{T_{\tau}\}_{\tau\in \CT}$:
$$
  T_\tau\colon e_i\mapsto e_{\tau(i)}, \ f_i\mapsto f_{\tau(i)}, \ \ph^{\pm 1}_i\mapsto \ph^{\pm 1}_{\tau(i)},
  \qquad \forall\, \tau\in \CT,i\in \wI
$$
which satisfy the following relations:
$$
  T_\tau T_i=T_{\tau(i)}T_\tau, \qquad \forall\, \tau\in \CT,i\in \wI
$$
Therefore, it remains for us to show that~\eqref{eqn:Beck isomorphism} is equivalent to~\eqref{eqn:Beck original formulas}
by proving:
\begin{equation}
\label{eqn:our vs Beck e}
  T_{\widehat{\omega^\vee_i}}^{-d}(e_i) =
  \begin{cases}
     E_{(\alpha_i,d)} &\text{if } d\geq 0 \\
     -E_{(\alpha_i,d)}\ph_i^{-1} &\text{if } d<0
  \end{cases}
\end{equation}
\begin{equation}
\label{eqn:our vs Beck f}
  T_{\widehat{\omega^\vee_i}}^d (f_i)=
  \begin{cases}
    -\ph_i E_{(-\alpha_i,d)} &\text{if } d>0 \\
    E_{(-\alpha_i,d)} &\text{if } d\leq 0
  \end{cases}
\end{equation}
Fix $i\in I$. According to our proof of Theorem~\ref{thm:weyl to lyndon}, there is $-l<k\leq 0$ such that
$\tbeta_k=(\alpha_i,0)$. Then, due to Remark~\ref{rem:comparison to Beck}(b) and the affine version
of~\eqref{eqn:simple root generators agree}, we get:
\begin{equation}
\label{eqn:d=0 compatibility}
  E_{(\alpha_i,0)}=E_{\tbeta_k}=e_i,\qquad  E_{(-\alpha_i,0)}=E_{-\tbeta_k}=\Omega(E_{\tbeta_k})=\Omega(e_i)=f_i
\end{equation}
thus verifying~(\ref{eqn:our vs Beck e},~\ref{eqn:our vs Beck f}) for $d=0$. On the other hand, we note that:
\begin{equation*}
  \tbeta_{k-dl}=(\alpha_i,d), \qquad \forall\, d>0
\end{equation*}
as follows from the sequence of equalities:
\begin{multline*}
  \tbeta_{k-dl} =
  \widehat{\rho^\vee}^{-1}\tau s_{i_{-l}}\dots s_{i_{k-dl+1}}(\alpha_{i_{k-dl}}) =
  \widehat{\rho^\vee}^{-1} s_{\tau(i_{-l})}\dots s_{\tau(i_{k-dl+1})}(\alpha_{\tau(i_{k-dl})}) = \\
  \widehat{\rho^\vee}^{-1} s_{i_{0}}\dots s_{i_{k-(d-1)l+1}}(\alpha_{i_{k-(d-1)l}}) = \dots =
  \widehat{\rho^\vee}^{-d} s_{i_{0}}\dots s_{i_{k+1}}(\alpha_{i_{k}}) =
  \widehat{\rho^\vee}^{-d} (\tbeta_k) = (\alpha_i,d)
\end{multline*}
with the last equality due to~\eqref{eqn:coweight action}. Then, the same argument verifies:
\begin{multline*}
  E_{\tbeta_{k-dl}} =
  T^{-1}_{\widehat{\rho^\vee}} T_{\tau} T^{-1}_{i_{-l}}\dots T^{-1}_{i_{k-dl+1}}(e_{i_{k-dl}}) =
  T^{-1}_{\widehat{\rho^\vee}} T^{-1}_{\tau(i_{-l})}\dots T^{-1}_{\tau(i_{k-dl+1})}(e_{\tau(i_{k-dl})}) = \\
  T^{-1}_{\widehat{\rho^\vee}} T^{-1}_{i_{0}}\dots T^{-1}_{i_{k-(d-1)l+1}}(e_{i_{k-(d-1)l}}) = \dots =
  T^{-d}_{\widehat{\rho^\vee}} T^{-1}_{i_{0}}\dots T^{-1}_{i_{k+1}}(e_{i_{k}}) =
  T^{-d}_{\widehat{\rho^\vee}} (E_{\tbeta_k})
\end{multline*}
To simplify the latter, we note that $\rho^\vee=\sum_{j\in I} \omega^\vee_j$ and
$l(\, \widehat{\rho^\vee}\, )=\sum_{j\in I} l(\, \widehat{\omega_j^\vee}\, )$,  due to Proposition~\ref{prop:length for lattice}.
Therefore, $T_{\widehat{\rho^\vee}}$ may be evaluated by using a reduced decomposition of $\widehat{\rho^\vee}$ obtained
as a concatenation of the reduced decompositions of $\widehat{\omega_j^\vee}$'s, hence:
$$
  T^{-d}_{\widehat{\rho^\vee}} =
  T^{-d}_{\widehat{\omega_i^\vee}}\cdot \prod_{j\in I}^{j\ne i} T^{-d}_{\widehat{\omega_j^\vee}}
$$
As $T^{\pm 1}_{\widehat{\omega_j^\vee}}(e_i)=e_i$ for $j\ne i$ by~\cite[Corollary 3.2]{B2},
and $E_{\tbeta_k}=e_i$ by~\eqref{eqn:d=0 compatibility}, we get:
\begin{equation}
\label{eqn:chain}
  E_{(\alpha_i,d)}=E_{\tbeta_{k-dl}} = T^{-d}_{\widehat{\rho^\vee}} (E_{\tbeta_k}) =
  T^{-d}_{\widehat{\rho^\vee}}(e_i) = T^{-d}_{\widehat{\omega_i^\vee}}(e_i)
\end{equation}
which proves~\eqref{eqn:our vs Beck e} for $d>0$.
Furthermore, as $\Omega$ commutes with the extended affine braid group action
(due to~\eqref{eqn:omega vs T} and $\Omega\circ T_\tau=T_\tau\circ \Omega$ for $\tau\in \CT$), we obtain:
$$
  E_{(-\alpha_i,-d)} = E_{-\tbeta_{k-dl}} = \Omega(E_{\tbeta_{k-dl}}) = \Omega (T^{-d}_{\widehat{\omega_i^\vee}} (e_i)) =
  T^{-d}_{\widehat{\omega_i^\vee}} (\Omega(e_i)) = T^{-d}_{\widehat{\omega_i^\vee}} (f_i)
$$
which proves~\eqref{eqn:our vs Beck f} for $d<0$.
For the remaining cases, let us note first that
\begin{equation}
\label{eqn:raz}
  \tbeta_{k+dl}=(-\alpha_i,d), \qquad \forall\, d>0
\end{equation}
Indeed, the $d=1$ case of~\eqref{eqn:raz} follows from:
\begin{multline}
\label{eqn:d=1 computation}
  (\alpha_i,-1) = \widehat{\rho^\vee} (\alpha_i,0) = \widehat{\rho^\vee} (\tbeta_k) =
  s_{i_1} \dots s_{i_l} \tau s_{i_0}\dots s_{i_{k+1}} (\alpha_{i_k}) = \\
  s_{i_1} \dots s_{i_l} s_{i_l}\dots s_{i_{l+k+1}} (\alpha_{i_{k+l}}) =
  s_{i_1} \dots s_{i_{k+l-1}} s_{i_{k+l}} (\alpha_{i_{k+l}}) =
  -\tbeta_{k+l}
\end{multline}
while the $d>1$ case of~\eqref{eqn:raz} follows from:
\begin{multline*}
  \tbeta_{k+dl} =
  \widehat{\rho^\vee} \tau^{-1} s_{i_{l+1}}\dots s_{i_{k+dl-1}} (\alpha_{i_{k+dl}}) =
  \widehat{\rho^\vee} s_{\tau^{-1}(i_{l+1})}\dots s_{\tau^{-1}(i_{k+dl-1})} (\alpha_{\tau^{-1}(i_{k+dl})}) \\
  = \widehat{\rho^\vee} s_{i_1}\dots s_{i_{k+(d-1)l-1}} (\alpha_{i_{k+(d-1)l}}) = \dots =
  \widehat{\rho^\vee}^{d-1} \tbeta_{k+l} = \widehat{\rho^\vee}^{d-1} (-\alpha_i,1) = (-\alpha_i,d)
\end{multline*}
with the last equality due to~(\ref{eqn:coweight action}). Using the same arguments as before, we obtain:
\begin{multline}
\label{eqn:d=1 computation quantum}
  T_{\widehat{\rho^\vee}}(e_i) = T_{\widehat{\rho^\vee}} (E_{\tbeta_k}) =
  T_{i_1} \dots T_{i_l} T_{\tau} T_{i_0}^{-1} \dots T_{i_{k+1}}^{-1} (e_{i_k}) =  \\
  T_{i_1} \dots T_{i_{k+l-1}} T_{i_{k+l}}(e_{i_{k+l}})
  = T_{i_1} \dots T_{i_{k+l-1}} (-f_{i_{k+l}}\ph_{i_{k+l}}) =\\
  -E_{-\tbeta_{k+l}}\ph_i^{-1} = -E_{(\alpha_i,-1)} \ph_i^{-1}
\end{multline}
where we used $T_j(e_j)=-f_j\ph_j$ (for any $j\in \wI$) as well as (using notation~\eqref{eqn:Cartan phi}):
$$
  T_{i_1} \dots T_{i_{k+l-1}} \ph_{i_{k+l}} = \ph_{s_{i_1}\dots s_{i_{k+l-1}}(\alpha_{i_{k+l}})} =
  \ph_{\tbeta_{k+l}} = \ph_{(-\alpha_i,1)} = \ph_i^{-1}
$$
with the last equality due to $C=1$. Likewise, for any $d>1$ we obtain:
\begin{multline}
\label{eqn:d>1 computation quantum}
  E_{(\alpha_i,-d) } = E_{-\tbeta_{k+dl}} = T_{i_1}\dots T_{i_{k+dl-1}}(f_{i_{k+dl}}) = \\
  T_{\widehat{\rho^\vee}} T^{-1}_\tau T_{i_{l+1}}\dots T_{i_{k+dl-1}}(f_{i_{k+dl}}) =
  T_{\widehat{\rho^\vee}} T_{i_{1}}\dots T_{i_{k+(d-1)l-1}}(f_{i_{k+(d-1)l}}) = \\
  \dots = T_{\widehat{\rho^\vee}}^{d-1} T_{i_{1}}\dots T_{i_{k+l-1}}(f_{i_{k+l}}) =
  - T_{\widehat{\rho^\vee}}^{d} (e_i) \ph_i
\end{multline}
with the last equality due to~\eqref{eqn:d=1 computation quantum} and $C=1$.
Combining~(\ref{eqn:d=1 computation quantum},~\ref{eqn:d>1 computation quantum}) with the equality
$T_{\widehat{\rho^\vee}}^{d} (e_i)=T_{\widehat{\omega^\vee_i}}^{d} (e_i)$, already established as part
of \eqref{eqn:chain}, verifies~\eqref{eqn:our vs Beck e} for $d<0$. Then, we also get:
$$
  T_{\widehat{\omega^\vee_i}}^{d} (f_i) = T_{\widehat{\omega^\vee_i}}^{d} (\Omega(e_i)) =
  \Omega(-E_{(\alpha_i,-d)}\ph^{-1}_i)= -\ph_i \Omega(E_{(\alpha_i,-d)}) = -\ph_i E_{(-\alpha_i,d)}
$$
which verifies~\eqref{eqn:our vs Beck f} for $d>0$. This completes our proof of Theorem~\ref{thm:two presentations}.
\end{proof}

\medskip

\begin{remark}
\label{rmk:compatibility of antiinvolutions}
The two formulas of~\eqref{eqn:Beck isomorphism} are compatible with each other, in that~\eqref{eqn:two presentations}
intertwines~\eqref{eqn:anti-involution} with the following $\BQ$-algebra anti-involution $\Omega^L$ of $\UU$:
\begin{equation}
\label{eqn:loop antiinvolution}
  \Omega^L\colon e_{i,k}\mapsto f_{i,-k},\ f_{i,k}\mapsto e_{i,-k},\ \ph^\pm_{i,l}\mapsto \ph^\mp_{i,l},\ q\mapsto q^{-1}
\end{equation}
for any $i\in I,\, k\in \BZ,\, l\in \BZ_{\geq 0}$.
\end{remark}

\medskip

\begin{remark}
The isomorphism~\eqref{eqn:two presentations} can be upgraded to a generic $C$, if one introduces the corresponding central
element in the definition of the quantum loop group of Definition~\ref{def:quantum loop} (which we choose not to do).
\end{remark}

\medskip

\noindent
Notably, the isomorphism \eqref{eqn:two presentations} does not intertwine the triangular decompositions
\eqref{eqn:triangular loop} and \eqref{eqn:triangular affine}. In fact, if we think of $\UU$ and $\VV/(C-1)$ as one
and the same algebra, then these two decompositions are ``orthogonal" to each other, as the following picture suggests.

\medskip

\begin{picture}(100,230)(-110,-75)
\label{pic:par}

\put(0,0){\circle*{2}}\put(20,0){\circle*{2}}\put(40,0){\circle*{2}}\put(60,0){\circle*{2}}\put(80,0){\circle*{2}}\put(100,0){\circle*{2}}\put(120,0){\circle*{2}}\put(0,20){\circle*{2}}\put(20,20){\circle*{2}}\put(40,20){\circle*{2}}\put(60,20){\circle*{2}}\put(80,20){\circle*{2}}\put(100,20){\circle*{2}}\put(120,20){\circle*{2}}\put(0,40){\circle*{2}}\put(20,40){\circle*{2}}\put(40,40){\circle*{2}}\put(60,40){\circle*{2}}\put(80,40){\circle*{2}}\put(100,40){\circle*{2}}\put(120,40){\circle*{2}}\put(0,60){\circle*{2}}\put(20,60){\circle*{2}}\put(40,60){\circle*{2}}\put(60,60){\circle*{2}}\put(80,60){\circle*{2}}\put(100,60){\circle*{2}}\put(120,60){\circle*{2}}\put(0,80){\circle*{2}}\put(20,80){\circle*{2}}\put(40,80){\circle*{2}}\put(60,80){\circle*{2}}\put(80,80){\circle*{2}}\put(100,80){\circle*{2}}\put(120,80){\circle*{2}}\put(0,100){\circle*{2}}\put(20,100){\circle*{2}}\put(40,100){\circle*{2}}\put(60,100){\circle*{2}}\put(80,100){\circle*{2}}\put(100,100){\circle*{2}}\put(120,100){\circle*{2}}\put(0,120){\circle*{2}}\put(20,120){\circle*{2}}\put(40,120){\circle*{2}}\put(60,120){\circle*{2}}\put(80,120){\circle*{2}}\put(100,120){\circle*{2}}\put(120,120){\circle*{2}}

\put(60,-10){\vector(0,1){140}}
\put(-10,60){\vector(1,0){140}}

\put(-25,10){\scalebox{4}{$\{$}}
\put(-70,17){$\VVm$}
\put(-25,90){\scalebox{4}{$\{$}}
\put(-70,97){$\VVp$}
\put(0,-15){\scalebox{4}{\rotatebox{270}{$\}$}}}
\put(0,-43){$\UUm$}
\put(80,-15){\scalebox{4}{\rotatebox{270}{$\}$}}}
\put(81,-43){$\UUp$}
\put(56,133){$\BN$}
\put(49,-22){$-\BN$}
\put(135,58){$Q^+$}
\put(-25,58){$Q^-$}

\put(-110,-65){F{\scriptsize IGURE} 1. The grading of $\UU \simeq \VV/(C-1)$ and its various subalgebras}
\end{picture}

\medskip

\noindent
The axes in the picture above describe the two components of $Q \times \BZ$, and the four subalgebras marked by parentheses
indicate the degrees in which elements of these subalgebras lie (although $\VVp$ and $\VVm$ also include elements on the
positive and negative horizontal axes, respectively, namely products of $e_i$ and $f_i$ for $i \in I$).

\medskip


\subsection{}

The picture in the previous Subsection suggests that we can further obtain triangular decompositions of the ``half'' subalgebras
in terms of ``quarter" subalgebras. Specifically, it was shown in~\cite[Lemma 5, Proof of Lemma 6]{B}
that:\footnote{We note that~\cite[Lemma 5]{B} treats $\VVp,\VVm$ in place of $\VVg,\VVl$ and $A_>,A_<$ in place of $\UUp,\UUm$,
respectively. However, this does not affect the equalities~(\ref{eqn:quarter +-},~\ref{eqn:quarter ++}), since $A_>$ (resp.\ $A_<$)
is obtained from $\UUp$ (resp.\ $\UUm$) by adding negative (resp.\ positive) imaginary root vectors as well as $\uuo$, as follows
from~\cite[Proof of Lemma 6]{B}.}
\begin{align}
  & U_q^+(L\fn^-) := \UUm \cap \VVg = \Big \{\text{subalgebra generated by } \sfe_{\tbeta_k}, k > 0 \Big\}
    \label{eqn:quarter +-}\\
  & U_q^+(L\fn^+) := \UUp \cap \VVg = \Big \{\text{subalgebra generated by } \sfe_{\tbeta_k}, k \leq 0 \Big\}
    \label{eqn:quarter ++}
\end{align}
where we define $\sfe_{\tbeta_k}$ in accordance with~\eqref{eqn:Beck isomorphism} via:
\begin{equation}
\label{eqn:twisted affine root generators}
  \sfe_{\tbeta_k} =
  \begin{cases}
    \ph_{-\hdeg(\tbeta_k)} E_{\tbeta_k} &\text{if } k > 0 \\
    E_{\tbeta_k} &\text{if } k \leq 0
\end{cases}
\end{equation}
Henceforth, given a homogeneous element $z$ of degree $\left(\sum_{i\in I} r_i\alpha_i,d\right)\in Q\times \BZ$, set
\begin{equation}
\label{eqn:h-Cartan phi}
  \ph_{\pm \hdeg(z)} := \ph_{\pm \sum_{i\in I} r_i\alpha_i} = \prod_{i\in I} \ph^{\pm r_i}_i \, \in \UUo
\end{equation}
As $C=1$, we note that~\eqref{eqn:Cartan phi} and~\eqref{eqn:h-Cartan phi} agree under
the identification~\eqref{eqn:two presentations}:
\begin{equation}
\label{eqn:deg vs hdeg}
  \ph_{\hdeg(z)}=\ph_{\deg(z)}
\end{equation}

\noindent
Formulas~(\ref{eqn:q comm general affine},~\ref{eqn:q comm affine}) still hold when the $E_{\tbeta_k}$ are replaced
with the $\sfe_{\tbeta_k}$, since commuting $\ph$'s simply produces powers of $q$. Likewise, the PBW
decompositions~(\ref{eqn:quarter 01},~\ref{eqn:quarter 02}) imply that the subalgebras above have the following PBW bases:
\begin{align}
  & U_q^+(L\fn^-) \ = \mathop{\bigoplus_{n_1,n_2,\dots \in \BZ_{\geq 0}}}_{n_1+n_2+\dots < \infty}
    \BQ(q) \cdot \dots \sfe_{\tbeta_2}^{n_2} \sfe_{\tbeta_1}^{n_1}
    \label{eqn:quarter 1} \\
  & U_q^+(L\fn^+) \ = \mathop{\bigoplus_{n_0,n_{-1},\dots \in \BZ_{\geq 0}}}_{n_0+n_{-1}+\dots < \infty}
    \BQ(q) \cdot \dots \sfe_{\tbeta_{-1}}^{n_{-1}} \sfe_{\tbeta_0}^{n_0}
    \label{eqn:quarter 2}
\end{align}
Applying the anti-involutions~(\ref{eqn:anti-involution},~\ref{eqn:loop antiinvolution}), compatible via
Remark~\ref{rmk:compatibility of antiinvolutions}, we see that the analogous ``quarter'' subalgebras of $\VVl$
have similar PBW decompositions:
\begin{align}
  & U_q^-(L\fn^-) := \UUm \cap \VVl \ =
    \mathop{\bigoplus_{n_0,n_{-1},\dots \in \BZ_{\geq 0}}}_{n_0+n_{-1}+\dots < \infty}
    \BQ(q) \cdot \sfe_{-\tbeta_{0}}^{n_{0}} \sfe_{-\tbeta_{-1}}^{n_{-1}}\dots
    \label{eqn:quarter 3} \\
  & U_q^-(L\fn^+) := \UUp \cap \VVl \ =
    \mathop{\bigoplus_{n_1,n_2,\dots \in \BZ_{\geq 0}}}_{n_1+n_2+\dots < \infty}
    \BQ(q) \cdot \sfe_{-\tbeta_1}^{n_1} \sfe_{-\tbeta_2}^{n_2} \dots \label{eqn:quarter 4}
\end{align}
where we define:
\begin{equation}
\label{eqn:twisted negative affine root generators}
  \sfe_{-\tbeta_k} = \Omega(\sfe_{\tbeta_k}) =
  \begin{cases}
    E_{-\tbeta_k} \ph_{\hdeg(\tbeta_k)}  &\text{if } k > 0 \\
    E_{-\tbeta_k} &\text{if } k \leq 0
\end{cases}
\end{equation}

\medskip

\noindent
The following result will finally allow us to construct the PBW bases of $\UUpm$.

\medskip

\begin{proposition}
\label{prop:quarter}
The multiplication map induces a vector space isomorphism:
\begin{equation}
\label{eqn:is isomorphism}
  U_q^+(L\fn^-) \otimes U_q^-(L\fn^-) \ \iso \ \UUm
\end{equation}
\end{proposition}

\medskip

\begin{proof}
The triangular decomposition~(\ref{eqn:triangular affine}) implies the injectivity of the map \eqref{eqn:is isomorphism}.
It thus suffices to show that:
$$
  U_q^+(L\fn^-) \otimes U_q^-(L\fn^-)
$$
is an algebra, since the fact that it contains all the generators of $\UUm$ (namely,  $f_{i,d}$'s with $i\in I,d\in \BZ$,
due to~\eqref{eqn:Beck isomorphism}) will imply the surjectivity of the map \eqref{eqn:is isomorphism}.
By definition, $U_q^+(L\fn^-)$ and $U_q^-(L\fn^-)$ are subalgebras, so it remains to show that:
\begin{equation}
\label{eqn:swap}
  ba \in U_q^+(L\fn^-) \otimes U_q^-(L\fn^-)
\end{equation}
for any $a \in U_q^+(L\fn^-)$ and $b \in U_q^-(L\fn^-)$. Products such as \eqref{eqn:swap} are governed by formula
\eqref{eqn:double}, with respect to the Drinfeld-Jimbo coproduct $\Delta$ of $\VV$.

\medskip

\noindent
In accordance with~(\ref{eqn:Damiani positive})--(\ref{eqn:Damiani negative restrictions}) and~\eqref{eqn:deg vs hdeg},
for $r\geq 1$ and $s\leq 0$, we thus obtain:
\begin{multline}
\label{eqn:twisted Damiani +}
  \Delta\left(\sfe_{\tbeta_r}\right) =
  1\otimes \sfe_{\tbeta_r} + \sfe_{\tbeta_r}\otimes \ph_{-\tbeta_r} +
  \sum \ph_{-\hdeg(x)}\, x\otimes \ph_{-\hdeg(\tbeta_r)}\, y \\
  \mathrm{with} \quad x \in U^+_q(r-1)  \ \text{ and } \ y \in \VVp
\end{multline}
\begin{multline}
\label{eqn:twisted Damiani -}
  \Delta\left(\sfe_{-\tbeta_s}\right) =
  1\otimes \sfe_{-\tbeta_s} + \sfe_{-\tbeta_s}\otimes \ph_{-\tbeta_s} +
  \sum y\otimes \ph_{\hdeg(y)}\, x  \\
  \mathrm{with}\quad x \in \VVm  \ \text{ and } \ y \in U^-_q(s+1)
\end{multline}
where the sums of~(\ref{eqn:twisted Damiani +}) and~(\ref{eqn:twisted Damiani -}) are vacuous for $r=1$ and $s=0$,
respectively. Therefore, for  $a \in U_q^+(L\fn^-)$ and $b \in U_q^-(L\fn^-)$ we have
\begin{align*}
  & \Delta(a) = a_1 \otimes a_2 \quad \text{where } a_1 \in U_q^+(L\fn^-)  \ \text{ and } \ a_2 \in \VVg \\
  & \Delta(b) = b_1 \otimes b_2 \quad \ \text{where } b_1 \in U_q^-(L\fn^-) \ \text{ and } \ b_2 \in \VVl
\end{align*}
Because the affine version of the pairing \eqref{eqn:bialg pair finite} pairs
trivially elements that do not sit in opposite $Q\times \BZ$-degrees, we have:
$$
  \langle a, b \rangle = \e(a) \e(b)
$$
for all $a \in U_q^+(L\fn^-)$ and $b\in U_q^-(L\fn^-)$ (here, $\e$ is the counit of the bialgebra $\VV$).
Therefore, \eqref{eqn:double} implies:
\begin{equation}
\label{eqn:double computation}
  a_1 b_1 \langle a_2,b_2\rangle = \langle a_1,b_1\rangle b_2a_2 = \e(a_1)\e(b_1) b_2a_2 = ba
\end{equation}
(the latter identity is part of the counit property of $\e$) thus implying \eqref{eqn:swap}.
\end{proof}

\medskip


\subsection{}

To make the presentation uniform, let us switch from $\tbeta_k$ of~\eqref{eqn:affine real roots} to $\beta_k$
of~\eqref{eqn:beta-roots} via~\eqref{eqn:beta vs tbeta}, so that the subalgebras $U_q^+(L\fn^-)$ and $U_q^-(L\fn^-)$
are generated by $\{\sfe_{-\beta_k}\}_{k\geq 1}$ and $\{\sfe_{-\beta_k}\}_{k\leq 0}$, respectively. Then, combining the
above Proposition with the PBW decompositions~(\ref{eqn:quarter 1},~\ref{eqn:quarter 3}), we obtain the PBW basis for $\UUm$:
\begin{equation}
\label{eqn:pbw basis loop}
  \UUm \ =
  \mathop{\bigoplus_{\dots, n_{-1}, n_0, n_1,n_2, \dots \in \BZ_{\geq 0}}}_{\dots + n_{-1} + n_0 + n_1+n_2+\dots < \infty}
    \BQ(q) \cdot \dots \sfe_{-\beta_{2}}^{n_{2}} \sfe_{-\beta_1}^{n_1} \sfe_{-\beta_0}^{n_0} \sfe_{-\beta_{-1}}^{n_{-1}} \dots
\end{equation}
We note that the set $\{-\beta_k\}_{k\in \BZ}$ exactly coincides with $\Delta^-\times \BZ$.

\medskip

\noindent
For any $r\geq 1$, let $U_q^+(L\fn^-)[r]$ be the subalgebra of $U_q^+(L\fn^-)$ generated by
$\sfe_{-\beta_1},\dots,\sfe_{-\beta_r}$. Likewise, for $s\leq 0$, let $U_q^-(L\fn^-)[s]$ be
the subalgebra of $U_q^-(L\fn^-)$ generated by $\sfe_{-\beta_0},\dots,\sfe_{-\beta_s}$,
that is, the subalgebra $U^-_q([s,0])$ in the previous notations.
These subalgebras have the following PBW bases, due to Lemma~\ref{lem:segmental pbw loop}:
\begin{align}
  & U_q^+(L\fn^-)[r] \ = \bigoplus_{n_1,n_2,\dots,n_r\in \BZ_{\geq 0}}
    \BQ(q) \cdot \sfe_{-\beta_r}^{n_r} \dots \sfe_{-\beta_2}^{n_2} \sfe_{-\beta_1}^{n_1} \label{eqn:truncated PBW +} \\
  & U_q^-(L\fn^-)[s] \ = \bigoplus_{n_0,n_{-1},\dots,n_s\in \BZ_{\geq 0}}
    \BQ(q) \cdot \sfe_{-\beta_{0}}^{n_{0}} \sfe_{-\beta_{-1}}^{n_{-1}} \dots \sfe_{-\beta_s}^{n_s}  \label{eqn:truncated PBW -}
\end{align}
Just as in the proof of Proposition~\ref{prop:quarter}, we obtain the following analogue of~\eqref{eqn:q comm general affine}.

\medskip

\begin{proposition}
\label{prop:orthogonal-to-Beck convexity}
For any $s\leq 0<r$, we have
\begin{equation*}
\label{eqn:orthogonal q comm general affine}
  \sfe_{-\beta_s} \sfe_{-\beta_r} - q^{(\beta_s, \beta_r)} \sfe_{-\beta_r} \sfe_{-\beta_s} \ \in
  \mathop{\bigoplus_{n_{r-1},\dots,n_{s+1}\in \BZ_{\geq 0}}}
  \BQ(q) \cdot \sfe_{-\beta_{r-1}}^{n_{r-1}} \dots \sfe_{-\beta_{1}}^{n_{1}}
               \sfe_{-\beta_{0}}^{n_{0}} \dots \sfe_{-\beta_{s+1}}^{n_{s+1}}
\end{equation*}
where the sum is finite as it is taken over all tuples $n_{r-1},\dots,n_{s+1}\in \BZ_{\geq 0}$ such that:
$$
  n_{r-1}\beta_{r-1}+\dots+n_1\beta_1+n_0\beta_0+\dots+n_{s+1}\beta_{s+1} = \beta_r+\beta_s
$$
\end{proposition}

\medskip

\begin{proof}
This follows from~\eqref{eqn:double computation} applied to the pair $a=\sfe_{-\beta_r}$ and $b=\sfe_{-\beta_s}$,
combined with~(\ref{eqn:twisted Damiani +},~\ref{eqn:twisted Damiani -}) and the above PBW
decompositions~(\ref{eqn:truncated PBW +},~\ref{eqn:truncated PBW -}).
\end{proof}

\medskip

\noindent
In the simplest case $\beta_r=(\alpha,-1),\, \beta_s=(\alpha_i,0)$, this can be further refined as follows.

\medskip

\begin{corollary}
\label{cor:answer-to-zhang}
$[\sfe_{(-\alpha_i,0)},\sfe_{(-\alpha,1)}]_q\in \BQ(q)^*\cdot \sfe_{(-\alpha-\alpha_i,1)}$
if $\alpha,\alpha+\alpha_i\in \Delta^+$.
\end{corollary}

\medskip

\begin{remark}
The above Corollary is in contrast with the better known approach which recovers $E_{(-\alpha,1)}$ via $q$-commutators
of $E_{(-\theta,1)}$ and $e_i$, see Corollary~\ref{cor:generating set for DJ root generators}(c).
\end{remark}

\medskip


\subsection{}

We shall now see that Theorem \ref{thm:PBW quantum loop} is equivalent to the PBW decomposition~(\ref{eqn:pbw basis loop})
applied to the reduced decomposition of $\widehat{\rho^\vee}$ produced by Theorem~\ref{thm:weyl to lyndon}.
The key feature of this reduced decomposition is that the ordered set of roots:
\begin{equation}
\label{eqn:beta-order revisited}
  \dots < \beta_2 < \beta_1 < \beta_0 < \beta_{-1} < \dots
\end{equation}
coincides with $\Delta^+ \times \BZ$ ordered in accordance with the bijection \eqref{eqn:associated word loop} via:
\begin{equation}
\label{eqn:o-compatiblility}
  \dots < \ell(\obeta_2) < \ell(\obeta_1) < \ell(\obeta_0) < \ell(\obeta_{-1}) < \dots
\end{equation}
where for any $(\alpha,d)\in \Delta^+\times \BZ$ we set:
\begin{equation}
\label{eqn:o-involution}
  \overline{(\alpha,d)} = (\alpha,-d)
\end{equation}

\medskip

\begin{proof}[Proof of Theorem \ref{thm:PBW quantum loop}]
Our proof will closely follow that of Theorem~\ref{thm:PBW quantum finite}. In particular, we shall need
the anti-involution $\varpi$ of $\UU$ defined via:
$$
  \varpi\colon e_{i,k}\mapsto f_{i,k},\ f_{i,k}\mapsto e_{i,k},\ \ph^\pm_{i,l}\mapsto \ph^{\pm}_{i,l}
$$
for any $i\in I$, $k\in \BZ$, $l\in \BN$, which should viewed as the loop version of $\varpi$ for $\uu$.
Applying $\varpi$ to~(\ref{eqn:pbw basis loop}), we obtain:
\begin{equation}
\label{eqn:pbw basis loop positive}
  \UUp \ =
  \bigoplus^{k\in \BN}_{\gamma_1\geq \dots \geq \gamma_k \in \Delta^+\times \BZ}
    \BQ(q) \cdot  \varpi(\sfe_{-\gamma_1}) \dots \varpi(\sfe_{-\gamma_k})
\end{equation}
with the above order on $\Delta^+ \times \BZ$ being~\eqref{eqn:beta-order revisited}.
On the other hand, combining the compatibility of~(\ref{eqn:beta-order revisited},~\ref{eqn:o-compatiblility}) with
Proposition~\ref{prop:orthogonal-to-Beck convexity} and formula~(\ref{eqn:q comm general affine}), we obtain:
$$
  [\sfe_{-\obeta},\sfe_{-\obeta'}]_q \ \in
  \mathop{\bigoplus^{k\in \BN}_{\ell(\beta')<\ell(\gamma_1)\leq \dots \leq \ell(\gamma_k) <\ell(\beta)}}_
    {\gamma_1 + \dots + \gamma_k = \beta+\beta'}
  \BQ(q) \cdot  \sfe_{-\ogamma_1} \dots \sfe_{-\ogamma_k}
$$
for any $\beta,\beta'\in \Delta^+\times \BZ$ such that $\obeta'<\obeta$, or equivalently $\ell(\beta')<\ell(\beta)$.
Applying the anti-involution $\varpi$ to the equation above, we get:
\begin{equation}
\label{eqn:q comm general affine twisted}
  [\varpi(\sfe_{-\obeta'}),\varpi(\sfe_{-\obeta})]_q \ \in
  \mathop{\bigoplus^{k\in \BN}_{\ell(\beta)>\ell(\gamma_1)\geq \dots \geq \ell(\gamma_k) >\ell(\beta')}}_
    {\gamma_1 + \dots + \gamma_k = \beta+\beta'}
  \BQ(q) \cdot  \varpi(\sfe_{-\ogamma_1}) \dots \varpi(\sfe_{-\ogamma_k})
\end{equation}
under the same restrictions on $\beta, \beta'\in \Delta^+\times \BZ$. In particular, if $\beta+\beta'\in \Delta^+\times \BZ$
and $\beta,\beta'$ are minimal in the sense:
\begin{equation}
\label{eqn:main minimal condition}
  \not \exists \ \alpha, \alpha' \in \Delta^+\times \BZ \quad \text{s.t.} \quad
  \obeta'<\oalpha'<\oalpha<\obeta \quad \text{and} \quad \alpha+\alpha'=\beta+\beta'
\end{equation}
we get (due to the convexity of Proposition~\ref{prop:convex loop}):
$$
  [\varpi(\sfe_{-\obeta'}),\varpi(\sfe_{-\obeta})]_q\in \BQ(q)\cdot \varpi(\sfe_{-\obeta-\obeta'})
$$
Using the arguments of Remark~\ref{rem:gavarini result}, the formula above can be further refined to:
\begin{equation}
\label{eqn:upsilon comm}
  [\varpi(\sfe_{-\obeta'}),\varpi(\sfe_{-\obeta})]_q\in \BZ[q,q^{-1}]^*\cdot \varpi(\sfe_{-\obeta-\obeta'})
\end{equation}
We claim that Theorem~\ref{thm:PBW quantum loop} follows from~(\ref{eqn:pbw basis loop positive}).
To this end, it suffices to show:
\begin{equation}
\label{eqn:loop coincidence}
  e_{\ell(\beta)} \in \BQ(q)^* \cdot \varpi(\sfe_{-\obeta})
\end{equation}
for any $\beta=(\alpha,d)\in \Delta^+\times \BZ$. We prove~(\ref{eqn:loop coincidence}) by induction on the height
of $\alpha$. The base case $\alpha=\alpha_i$ (with $i\in I$) is immediate, due
to~(\ref{eqn:Beck isomorphism},~\ref{eqn:twisted affine root generators},~\ref{eqn:twisted negative affine root generators}):
$$
  e_{\left[ i^{(d)} \right]} = e_{i,d} = \varpi(f_{i,d}) = \pm \varpi(\sfe_{(-\alpha_i,d)})
$$
For the induction step, consider the factorization \eqref{eqn:costandard factorization} of $\ell = \ell(\alpha,d)$:
$$
  \ell = \ell_1\ell_2
$$
Since factors of standard loop words are standard, we have $\ell_1 = \ell(\gamma_1,d_1)$ and $\ell_2 = \ell(\gamma_2,d_2)$
for some $(\gamma_1,d_1), (\gamma_2,d_2) \in \Delta^+\times \BZ$ such that $\alpha = \gamma_1 + \gamma_2, d=d_1+d_2$.
By the induction hypothesis, we have:
$$
  e_{\ell_k} \in \BQ(q)^* \cdot \varpi(\sfe_{(-\gamma_k,d_k)})
$$
for $k\in \{1,2\}$. However, we note that $(\gamma_1,d_1) < (\alpha,d) < (\gamma_2,d_2)$ is a minimal decomposition
in the sense of \eqref{eqn:main minimal condition}, according to Proposition \ref{prop:lyndon is minimal}.
Therefore, comparing~\eqref{eqn:quantum bracketing lyndon affine} with \eqref{eqn:upsilon comm}, we obtain:
\begin{equation*}
  e_\ell = [e_{\ell_1}, e_{\ell_2}]_q \in
  \BQ(q)^* \cdot \varpi([\sfe_{(-\gamma_2,d_2)}, \sfe_{(-\gamma_1,d_1)}]_q) =
  \BQ(q)^* \cdot \varpi(\sfe_{(-\alpha,d)})
\end{equation*}
as we needed to prove.
\end{proof}

\medskip


\section{Shuffle algebras of Feigin-Odesskii and Enriquez}
\label{sec:fo}

In the present Section, we will connect the loop shuffle algebra $\hCF$ with the trigonometric degeneration of the
Feigin-Odesskii shuffle algebra associated with $\fg$, with the goal of establishing Theorem~\ref{thm:main 2}.

\medskip


\subsection{}
\label{sub:FO-algebra}

We now recall the trigonometric degeneration (\cite{E1}) of the Feigin-Odesskii shuffle algebra (\cite{FO}) of type $\fg$.
Consider the vector space of color-symmetric rational functions:
\begin{equation}
\label{eqn:big}
  \CV \ =
  \bigoplus_{\bk = \sum_{i \in I} k_i \alpha_i \in Q^+}
    \BQ(q)(\dots,z_{i1},\dots,z_{ik_i},\dots)^{\sym}_{i \in I}
\end{equation}
The index $i \in I$ will be called the \underline{color} of the variables $z_{i1}, \dots, z_{ik_i}$.
The term \underline{color-symmetric} (as well as the superscript ``Sym" in the formula above) refers to
rational functions which are symmetric in the variables of each color separately.
We make the vector space $\CV$ into a $\BQ(q)$-algebra via the following \underline{shuffle product}:
\begin{equation}
\label{eqn:mult}
  F(\dots, z_{i1}, \dots, z_{i k_i}, \dots) * G(\dots, z_{i1}, \dots,z_{i l_i}, \dots) = \frac 1{\bk! \cdot \bl!}\,\cdot
\end{equation}
$$
  \textrm{Sym} \left[ F(\dots, z_{i1}, \dots, z_{ik_i}, \dots) G(\dots, z_{i,k_i+1}, \dots, z_{i,k_i+l_i}, \dots)
               \prod_{i,j \in I} \prod_{a \leq k_i, b > k_j} \zeta_{ij} \left( \frac {z_{ia}}{z_{jb}} \right) \right]
$$
In \eqref{eqn:mult}, $\sym$ denotes symmetrization with respect to the:
\begin{equation}
\label{eqn:deffactorial}
  (\bk+\bl)! := \prod_{i\in I} (k_i+l_i)!
\end{equation}
permutations that permute the variables $z_{i1}, \dots, z_{i,k_i+l_i}$ for each $i$ independently.

\medskip

\begin{definition}
\label{def:shuf}
(\cite{E1}, inspired by \cite{FO})
The \underline{positive shuffle algebra} $\CA^+$ is the subspace of $\CV$ consisting of rational functions of the form:
\begin{equation}
\label{eqn:shuf}
  R(\dots,z_{i1},\dots,z_{ik_i},\dots) =
  \frac {r(\dots,z_{i1},\dots,z_{ik_i},\dots)}
  {\prod^{\text{unordered}}_{\{i \neq i'\} \subset I} \prod_{1\leq a \leq k_i}^{1\leq a' \leq k_{i'}} (z_{ia} - z_{i'a'})}
\end{equation}
where $r$ is a symmetric Laurent polynomial that satisfies the \underline{wheel conditions}:
\begin{equation}
\label{eqn:wheel}
  r(\dots, z_{ia}, \dots)\Big|_
  {(z_{i1},z_{i2},z_{i3}, \dots, z_{i,1-a_{ij}}) \mapsto (w, w q_i^{2}, wq_i^{4}, \dots, w q_i^{-2a_{ij}}),\,
   z_{j1} \mapsto w q_i^{-a_{ij}}} =\, 0
\end{equation}
for any distinct $i, j \in I$.
\end{definition}

\medskip

\begin{remark}
Because of~(\ref{eqn:wheel}), any $r$ as in \eqref{eqn:shuf} is actually divisible by:
$$
  \prod^{\text{unordered}}_{\{i \neq i'\} \subset I:a_{ii'}=0} \prod_{1\leq b \leq k_i}^{1\leq b' \leq k_{i'}} (z_{ib} - z_{i'b'})
$$
Therefore, rational functions $R$ satisfying~(\ref{eqn:shuf},~\ref{eqn:wheel}) can only have simple poles on the diagonals
$z_{ib}=z_{i'b'}$ with \underline{adjacent} $i,i'\in I$, that is, such that $a_{ii'}<0$.
\end{remark}

\medskip

\noindent
The following is elementary, and we leave it to the interested reader.

\medskip

\begin{proposition}
$\CA^+$ is closed under the product \eqref{eqn:mult}, and is thus an algebra.
\end{proposition}

\medskip


\subsection{}
\label{sub:qaff vs FO}

The algebra $\CA^+$ is graded by $\bk = \sum_{i \in I} k_i \alpha_i \in Q^+$ that encodes the number of variables
of each color, and by the total homogeneous degree $d\in \BZ$. We write:
$$
  \deg R = (\bk,d)
$$
and say that $\CA^+$ is $Q^+ \times \BZ$-graded. We will denote the graded pieces by:
$$
  \CA^+ = \bigoplus_{\bk \in Q^+} \CA_{\bk} \ \ \qquad \text{and} \qquad \ \
  \CA_{\bk} = \bigoplus_{d \in \BZ} \CA_{\bk, d}
$$
We define the \underline{negative shuffle algebra} as $\CA^- = \left(\CA^+\right)^{\text{op}}$.
It is graded by $Q^- \times \BZ$, where a rational function in $\bk$ variables of homogeneous degree $d$
is assigned degree $(-\bk,d)$, when viewed as an element of $\CA^-$. We will denote the graded pieces~by:
$$
  \CA^- = \bigoplus_{-\bk \in Q^-} \CA_{-\bk} \quad \text{ and } \quad
  \CA_{-\bk} = \bigoplus_{d \in \BZ} \CA_{-\bk, d}
$$

\begin{proposition}
\label{prop:upsilon +}
(\cite{E1})
There exist unique algebra homomorphisms:
\begin{equation}
\label{eqn:upsilon +}
  \UUp \stackrel{\Upsilon}\longrightarrow \CA^+ \quad \text{and} \quad \UUm \stackrel{\Upsilon}\longrightarrow \CA^-
\end{equation}
determined by $\Upsilon(e_{i,d}) = z_{i1}^d \in \CA_{\bs_i,d}$ and $\Upsilon(f_{i,d}) = z_{i1}^d \in \CA_{-\bs_i,d}$, respectively.
\end{proposition}

\medskip

\begin{proposition}
\label{cor:injective}
The maps $\Upsilon$ of \eqref{eqn:upsilon +} are injective.
\end{proposition}

\medskip

\begin{proof}
We will prove the required statement for $\UUp$, as taking the opposite of both algebras yields the statement for $\UUm$.
Let us consider the ring $\BA = \BQ[[\hbar]]$, its fraction field $\BF = \BQ((\hbar))$, and define:
$$
  U_{\BA}(L\fn^+) \quad \text{and} \quad U_{\BF}(L\fn^+)
$$
by replacing $\BQ(q)$ in Definition~\ref{def:quantum loop} with $\BA$ and $\BF$, respectively. Similarly, let us define $\CA^+_\BA$
and $\CA^+_\BF$ by replacing $\BQ(q)$ with $\BA$ and $\BF$ in the definition of $\CA^+$, respectively (more precisely, by requiring $r$ of~\eqref{eqn:shuf} to have coefficients in $\BA$ or $\BF$, respectively). Then we have a commutative diagram:
$$
\begin{CD}
  U_{\BA}(L\fn^+) @>{\Upsilon_\BA}>> \CA^+_{\BA}\\
    @VV\jmath V @VVV\\
  U_{\BF}(L\fn^+) @>{\Upsilon_\BF}>> \CA^+_{\BF}
\end{CD}
$$
where the horizontal maps are defined by analogy with $\Upsilon$ (just over different coefficient rings). Note that the right-most map
is injective, but the left-most map is not necessarily so, due to the fact that $U_{\BA}(L\fn^+)$ might have $\BA$-torsion.

\medskip

\begin{claim}
\label{claim:upsilon}
The map $\Upsilon_\BF$ is injective.
\end{claim}

\medskip

\noindent
Let us first show how Claim \ref{claim:upsilon} allows us to complete the proof of the Proposition. The assignment $q = e^{\hbar}$
gives us vertical maps which make the following diagram commute:
$$
\begin{CD}
  \UUp @>{\Upsilon}>> \CA^+\\
    @VVV @VVV\\
  U_{\BF}(L\fn^+) @>{\Upsilon_\BF}>> \CA^+_{\BF}
\end{CD}
$$
We need to show that the top map is injective. Since the claim tells us that the bottom map is injective, then it suffices
to show that the left-most map is injective. The latter claim follows from the fact that $\UUp$ (respectively $U_{\BF}(L\fn^+)$)
is a free $\BQ(q)$ (respectively $\BF$) module with a basis given by ordered products of the root vectors
$\{\sfe_{(\alpha,d)}\}_{\alpha \in \Delta^+}^{d\in \BZ}$. In the case of $\UUp$, this follows from the corresponding result for
the affine quantum group via Theorem~\ref{thm:two presentations}, following our discussion from Section~\ref{sec:loop affine}.
Explicitly, it is obtained by applying the anti-involution $\Omega^L$ to the PBW decomposition~\eqref{eqn:pbw basis loop}, in view of
Remark~\ref{rmk:compatibility of antiinvolutions} and formula~\eqref{eqn:twisted negative affine root generators}.
In the case of $U_{\BF}(L\fn^+)$, one simply does the same proof, replacing the field $\BQ(q)$ by $\BF$ everywhere.

\medskip

\noindent
Let us now prove Claim~\ref{claim:upsilon}. Consider any $x \in U_{\BF}(L\fn^+)$ such that $\Upsilon_\BF(x) = 0$,
and our goal is to prove that $x = 0$. We may write:
$$
  x = \frac {\jmath(y)}{\hbar^k}
$$
for some $k \in \BN$ and $y \in U_{\BA}(L\fn^+)$, and assume for the purpose of contradiction that $\jmath(y) \neq 0$. The fact that
$\Upsilon_\BF(x) = 0$ and the injectivity of the map $\CA^+_{\BA} \rightarrow \CA^+_{\BF}$ implies that $\Upsilon_{\BA}(y) = 0$.
By~\cite[Corollary 1.4]{E2}, this implies that:
$$
  y \in \bigcap_{n=0}^\infty \hbar^n \cdot U_{\BA}(L\fn^+)
$$
Thus, for all $n \geq 0$, there exists $y_n \in U_{\BA}(L\fn^+)$ such that $y = \hbar^n y_n$. Passing this equality through
the map $\jmath$, we have for all $n \geq 0$:
\begin{equation}
\label{eqn:contradiction}
  \jmath(y) = \hbar^n \cdot \jmath(y_n)
\end{equation}
However, because $y$ and $y_n$'s lie in $U_{\BA}(L\fn^+)$, their images under $\jmath$ will lie in the free $\BA$-submodule of
$U_{\BF}(L\fn^+)$ spanned by ordered products of the root vectors $\sfe_{(\alpha,d)}$ (this statement uses the fact that the generators
$e_{i,d}$ of $U_{\BA}(L\fn^+)$ are among the $\sfe_{(\alpha,d)}$'s, due to~\eqref{eqn:Beck isomorphism}, together with the fact that
the structure constants of arbitrary products of $\sfe_{(\alpha,d)}$'s lie in $\BZ[q,q^{-1}] \subset \BA$,
cf. Remark~\ref{rem:gavarini result}). Therefore, there exist uniquely determined constants $c_{\alpha_1,\dots,\alpha_k}^{d_1,\dots,d_k} \in \BA$ such that:
$$
  \jmath(y) \ = \sum^{k \in \BN}_{(\alpha_1,d_1) \leq \dots \leq (\alpha_k,d_k)}
  c_{\alpha_1,\dots,\alpha_k}^{d_1,\dots,d_k} \cdot \sfe_{(\alpha_1,d_1)} \dots \sfe_{(\alpha_k,d_k)}
$$
But if in \eqref{eqn:contradiction} we take
$n$ larger than the leading power of $\hbar$ in all the $c_{\alpha_1,\dots,\alpha_k}^{d_1,\dots,d_k}$ which appear as coefficients of
$\jmath(y)$, we obtain a contradiction.
\end{proof}

\medskip

\begin{remark}
In type $A_{n-1}$ (and its affine version corresponding to quantum toroidal algebras of $\mathfrak{sl}_n$), a proof of Proposition~\ref{cor:injective}
was provided in~\cite[Theorem~1.1]{N2}. In simply laced finite types (as well as simply laced affine types), a proof of injectivity follows
from~\cite[Theorem 2.3.2(b) combined with formula (2.50)]{VV}, using the framework of $K$-theoretic Hall algebras of quivers, see~\cite{SV}.
In contrast, our proof of Proposition~\ref{cor:injective} for all finite types is based on~\cite{E2} and the PBW bases of Section~\ref{sec:quantum}.
\end{remark}

\medskip


\subsection{}
\label{sub:coproduct}

Define the \underline{extended shuffle algebras} as:
\begin{align}
  & \CA^\geq = \CA^+ \otimes \BQ(q) \left[(\ph_{i,0}^+)^{\pm 1}, \ph^+_{i,1},  \ph^+_{i,2}, \dots \right]_{i \in I}
    \label{eqn:a geq} \\
  & \CA^\leq = \CA^- \otimes \BQ(q) \left[(\ph_{i,0}^-)^{\pm 1}, \ph^-_{i,1}, \ph^-_{i,2}, \dots \right]_{i \in I}
    \label{eqn:a leq}
\end{align}
with pairwise commuting $\ph$'s, where the multiplication is governed by the rule:
\begin{equation}
\label{eqn:comm phi}
  \ph_j^\pm(w) * R^\pm(\dots, z_{ia}, \dots) = R^\pm(\dots, z_{ia}, \dots) * \ph_j^\pm(w) \cdot
  \prod_{i \in I} \prod_{a=1}^{k_{i}}
  \frac {\zeta_{ji} \left( w / z_{ia} \right)^{\pm 1}}{\zeta_{ij} \left( z_{ia}/w \right)^{\pm 1}}
\end{equation}
for any $R^\pm \in \CA_{\pm \bk}$, where the $\zeta$-factors are expanded as power series in non-negative powers of $w^{\mp 1}$.
Above, as before, we encode all $\ph$'s into the generating series:
\begin{equation}
\label{eqn:series phi}
  \ph^\pm_i(w) = \sum_{d = 0}^\infty \frac {\ph^\pm_{i,d}}{w^{\pm d}}
\end{equation}
Our reason for defining the extended shuffle algebras is that they admit coproducts.

\medskip

\begin{proposition}
\label{prop:coproduct}
(\cite{E2}, see also~\cite{N1, N2})
There exist bialgebra structures on $\CA^\geq$ and $\CA^\leq$, with coproduct determined by:
\begin{equation}
\label{eqn:coproduct0}
  \Delta(\ph^\pm_i(z)) = \ph^\pm_i(z) \otimes \ph^\pm_i(z)
\end{equation}
and the following assignments for all $R^\pm \in \CA_{\pm \bk}$:
\begin{align}
  & \Delta(R^+) \ = \sum_{\bl = \sum_{i \in I} l_i \alpha_i \in Q^+, \ l_i \leq k_i}
    \frac {\left[ \prod_{i \in I}^{a>l_i} \ph^+_i(z_{ia}) \right] * R^+(z_{i, a \leq l_i} \otimes z_{i, a>l_i})}
          {\prod_{i,i' \in I} \prod_{a \leq l_{i}}^{a' > l_{i'}} \zeta_{i'i}(z_{i'a'}/z_{ia})}
    \label{eqn:coproduct1} \\
  & \Delta(R^-) \ = \sum_{\bl = \sum_{i \in I} l_i \alpha_i \in Q^+, \ l_i \leq k_i}
    \frac {R^-(z_{i, a \leq l_i} \otimes z_{i, a>l_i}) * \left[ \prod_{i \in I}^{a \leq l_i} \ph^-_i(z_{ia}) \right]}
          {\prod_{i,i' \in I} \prod_{a \leq l_{i}}^{a' > l_{i'}} \zeta_{ii'}(z_{ia}/z_{i'a'})}
    \label{eqn:coproduct2}
\end{align}
\end{proposition}

\medskip

\begin{remark}
To think of~\eqref{eqn:coproduct1} as a well-defined tensor, we expand the right-hand side in non-negative powers of
$z_{ia} / z_{i'a'}$ for $a\leq l_i$ and $a'>l_{i'}$, thus obtaining an infinite sum of monomials. In each of these
monomials, we put the symbols $\ph^+_{i,d}$ to the very left of the expression, then all powers of $z_{ia}$ with
$a\leq l_i$, then the $\otimes$ sign, and finally all powers of $z_{ia}$ with $a>l_i$. The resulting expression
will be a power series, and therefore lies in a completion of $\CA^\geq \otimes \CA^+$. The same argument applies
to~\eqref{eqn:coproduct2}, still using non-negative powers of $z_{ia}/z_{i'a'}$ for $a\leq l_i$ and $a'>l_{i'}$,
and keeping all the $\ph_{i,d}^-$ to the very right.
\end{remark}

\medskip

\noindent
The following is straightforward.

\medskip

\begin{proposition}
\label{prop:upsilon geq}
The maps \eqref{eqn:upsilon +} extend to bialgebra homomorphisms:
\begin{equation}
\label{eqn:upsilon geq}
  \UUg \stackrel{\Upsilon}\longrightarrow \CA^\geq \quad \text{and} \quad
  \UUl \stackrel{\Upsilon}\longrightarrow \CA^\leq
\end{equation}
by sending $\ph_{i,d}^\pm \in \UUg, \UUl$ to the same-named $\ph_{i,d}^\pm \in \CA^\geq, \CA^\leq$.
\end{proposition}

\medskip


\subsection{}
\label{sub:full}

There exists a bialgebra pairing between $\CA^\geq$ and $\CA^\leq$.
As a first step toward defining it, we start with the following result. Let $Dz=\frac{dz}{2\pi i z}$.

\medskip

\begin{proposition}
\label{prop:pair shuf}
There exists a unique bialgebra pairing:
\begin{equation}
\label{eqn:pair shuf 1}
  \langle \cdot, \cdot \rangle \colon \ \CA^\geq \otimes \UUl \longrightarrow \BQ(q)
\end{equation}
satisfying~\eqref{eqn:aff pair 2} as well as:
\begin{equation}
\label{eqn:pair shuf generators}
\begin{split}
  & \Big \langle R, f_{i_1,-d_1} \dots f_{i_k,-d_k} \Big \rangle = \\
  & \qquad \prod_{a=1}^k (q_{i_a}^{-1} - q_{i_a})^{-1} \int_{|z_{1}| \ll \dots \ll |z_{k}|}
    \frac {R(z_{1}, \dots,z_{k}) z_{1}^{-d_1} \dots z_{k}^{-d_k}}
          {\prod_{1 \leq a < b \leq k} \zeta_{i_a i_b} (z_{a}/z_{b})} \prod_{a=1}^k D z_{a}
\end{split}
\end{equation}
for any
  $k\in \BN,\, i_1, \dots,i_k \in I,\,  d_1, \dots, d_k \in \BZ,\, R \in \CA_{\bs_{i_1}+\dots+\bs_{i_k},d_1+\dots+d_k}$
(all pairings between elements of non-opposite degrees are set to be $0$). In the right-hand side
of~\eqref{eqn:pair shuf generators}, we plug each variable $z_{a}$ into an argument of color $i_a$ of
the function $R$; since the latter is color-symmetric, the result is independent of any choices made.
\end{proposition}

\medskip

\begin{proof} We claim that formula~\eqref{eqn:pair shuf generators} yields a well-defined pairing of vector spaces
$$
\CA^+\otimes \UUm \longrightarrow \BQ(q)
$$
To see this, we need to acknowledge the fact that relations
\eqref{eqn:rel 0 affine} and \eqref{eqn:rel 1 affine} (or more precisely, the opposite of these relations,
since we are using $f$'s instead of $e$'s) imply linear relations between the various $f_{i_1,-d_1} \dots f_{i_k,-d_k}$,
and we need to check that these relations also hold in the right-hand side of \eqref{eqn:pair shuf generators}. Explicitly, we have:
\begin{equation}
\label{eqn:relation 1}
  f_{i,-r+1} f_{j,-s} q^{d_{ij}} - f_{i,-r} f_{j,-s+1} = f_{j,-s} f_{i,-r+1} - f_{j,-s+1} f_{i,-r} q^{d_{ij}}
\end{equation}
for all $i,j \in I$ and all $r, s \in \BZ$, and:
\begin{equation}
\label{eqn:relation 2}
\begin{split}
  & \sum_{\sigma \in S(1-a_{ij})} \sum_{k=0}^{1-a_{ij}} (-1)^k {1-a_{ij} \choose k}_i \cdot \\
  & \qquad \qquad \qquad \qquad f_{i,-r_{\sigma(1)}}\dots f_{i,-r_{\sigma(k)}} f_{j,-s} f_{i,-r_{\sigma(k+1)}}
    \dots f_{i,-r_{\sigma(1-a_{ij})}} = 0
\end{split}
\end{equation}
for all distinct $i, j \in I$ and all $r_1, \dots, r_{1-a_{ij}}, s\in \BZ$.
If we multiply the above formulas both on the left and the right with arbitrary products of $f$'s,
then we obtain various linear relations between products $f_{i_1,-d_1} \dots f_{i_k,-d_k}$.
The issue as to why these linear relations hold in the right-hand side of~\eqref{eqn:pair shuf generators}
is an interesting, but straightforward, exercise that we leave to the interested reader:
in the case of~\eqref{eqn:relation 1} it is because any rational function $R \in \CA^+$ can be written
as in~\eqref{eqn:shuf} with $r$ a Laurent polynomial, while in the case of~\eqref{eqn:relation 2}
it is because this $r$ satisfies the wheel conditions~\eqref{eqn:wheel}.
Details can be found in~\cite[\S2--3]{E1} and~\cite{DJ}, cf.\ our proof of Proposition~\ref{prop:image affine}.

\medskip

\noindent It remains to prove that formula \eqref{eqn:pair shuf generators} can be extended via \eqref{eqn:bialg pair 1} and \eqref{eqn:bialg pair 2} to a well-defined pairing on the larger algebras appearing in \eqref{eqn:pair shuf 1}. Let us start by noting that in series notation, formula \eqref{eqn:pair shuf generators} reads:
\begin{equation}
\label{eqn:pair shuf generators series}
 \Big \langle R, f_{i_1}(z_1) \dots f_{i_k}(z_k) \Big \rangle \sim \left[
    \frac {R(z_{1}, \dots,z_{k})}
          {\prod_{1 \leq s < t \leq k} \zeta_{i_s i_t} \left( \frac {z_{s}}{z_{t}} \right)} \right]_{|z_{1}| \ll \dots \ll |z_{k}|}
\end{equation}
where $[\dots]_{|z_{1}| \ll \dots \ll |z_{k}|}$ denotes the power series expansion in non-negative powers of $z_1/z_2, \dots, z_{k-1}/z_k$ and $\sim$ denotes proportionality up to the particular product of $(q_i^{-1}-q_i)^{-1}$ that features in the right-hand side of \eqref{eqn:pair shuf generators} (for brevity, we shall ignore this product for the remainder of the proof). Formula \eqref{eqn:pair shuf generators series} is actually imposed upon us by \eqref{eqn:bialg pair 1} and \eqref{eqn:aff pair 1}--\eqref{eqn:aff pair 2}, since:
$$
\Big \langle R, f_{i_1}(z_1) \dots f_{i_k}(z_k) \Big \rangle = \Big \langle \Delta^{(k-1)}(R), f_{i_1}(z_1) \otimes \dots \otimes  f_{i_k}(z_k) \Big \rangle
$$
If we write $R(z_1,\dots,z_k) = R_1(z_1) \dots R_k(z_k)$, with the implication that one needs a summation sign in front of the right-hand side, then \eqref{eqn:coproduct1} yields:
$$
\Big \langle \Delta^{(k-1)}(R), f_{i_1}(z_1) \otimes \dots \otimes  f_{i_k}(z_k) \Big \rangle = \frac {\prod_{s=1}^k \Big \langle \ph^+_{i_{s+1}}(z_{s+1}) \dots \ph^+_{i_{k}}(z_{k}) R_s(z_s), f_{i_s}(z_s) \Big\rangle}{\prod_{1\leq s < t\leq k} \zeta_{i_ti_s} \left( \frac {z_t}{z_s} \right)}
$$
expanded as $|z_1| \ll \dots \ll |z_k|$. In virtue of \eqref{eqn:bialg pair 2} and \eqref{eqn:aff pair 1}--\eqref{eqn:aff pair 2}, the numerator of the right-hand side above is computed as:
$$
\Big \langle \ph^+_{i_{s+1}}(z_{s+1}) \dots \ph^+_{i_{k}}(z_{k}) R_s(z_s), f_{i_s}(z_s) \Big \rangle \sim R_s(z_s) \frac {\zeta_{i_{s+1}i_s} \left( \frac {z_{s+1}}{z_s} \right)}{\zeta_{i_{s}i_{s+1}} \left( \frac {z_{s}}{z_{s+1}} \right)} \dots \frac {\zeta_{i_ki_s} \left( \frac {z_k}{z_s} \right)}{\zeta_{i_si_k} \left( \frac {z_{s}}{z_k} \right)}
$$
Putting the three displays above together explains why we need formula \eqref{eqn:pair shuf generators series}.

\medskip

\noindent A straightforward generalization of the above computation shows that the most general formula for the pairing \eqref{eqn:pair shuf 1} must be:
\begin{equation}
\label{eqn:pair shuf extended generators}
\left \langle R \prod_{s=1}^l \ph_{j_s}^+(w_s), f_{i_1}(z_1) \dots f_{i_k}(z_k) \prod_{s' = 1}^{l'} \ph_{j'_{s'}}^-(w'_{s'}) \right \rangle \sim
\end{equation}
$$
\sim \left[
    \frac {R(z_{1}, \dots,z_{k})}
          {\prod_{1 \leq s < t \leq k} \zeta_{i_si_t} \left( \frac {z_s}{z_t} \right)} \right]_{|z_{1}| \ll \dots \ll |z_{k}|} \left[\prod_{s=1}^l \prod_{s' = 1}^{l'} \frac {\zeta_{j_s j'_{s'}}\left(\frac {w_s}{w'_{s'}} \right)}{\zeta_{j'_{s'}j_s}\left(\frac {w'_{s'}}{w_s} \right)} \right]_{|w_1'|, \dots |w'_{l'}| \ll |w_1| , \dots |w_l|}
$$
for any indices $j_1,\dots,j_l, j_1', \dots, j'_{l'} \in I$ and formal variables $w_1,\dots,w_l, w'_1,\dots,w'_{l'}$. We need to show that formula \eqref{eqn:pair shuf extended generators} satisfies the bialgebra properties~(\ref{eqn:bialg pair 1}) and (\ref{eqn:bialg pair 2}) where $a,b,c$ are various elements of $\CA^\geq$ or $\UUl$. Because the bialgebra properties are multiplicative in $a,b,c$, one needs only to consider the case when either one of the three symbols $a,b,c$ is either in $\CA^+$, $\UUm$ or in $\CA^0$, $\UUo$. There are numerous cases one needs to consider, for instance:
\begin{equation}
\label{eqn:test pair}
\left\langle \prod_{t=1}^l \ph_{j_t}^+(w_t) \cdot R, f_{i_1}(z_1) \dots f_{i_k}(z_k) \right \rangle \stackrel{\eqref{eqn:bialg pair 2}}=
\end{equation}
$$
= \Big \langle R, f_{i_1}(z_1) \dots f_{i_k}(z_k) \Big \rangle  \left \langle  \prod_{t=1}^l  \ph^+_{j_t}(w_t), \prod_{s=1}^k \ph^-_{i_s}(z_s)  \right \rangle \sim
$$
$$
\sim \left[ \frac {R(z_{1}, \dots,z_{k})}{\prod_{1 \leq s < t \leq k}  \zeta_{i_si_t} \left( \frac {z_s}{z_t} \right)} \right]_{|z_{1}| \ll \dots \ll |z_{k}|} \prod_{s=1}^k \prod_{t=1}^l \frac {\zeta_{j_ti_s} \left(\frac {w_t}{z_s}\right)}{\zeta_{i_sj_t} \left(\frac {z_s}{w_t}\right)}
$$
is compatible with using relation \eqref{eqn:comm phi} to push all the $\ph^+_{j_t}(w_t)$ to the right of $R$ and then applying \eqref{eqn:pair shuf extended generators}. We will leave the analogous checks (when at least one of $a,b,c$ is among the series $\ph^\pm_j(w)$, $j \in I$) as exercises to the interested reader. Furthermore, note that we already implicitly established formula \eqref{eqn:bialg pair 1} for:
$$
a = R, \qquad b = f_{i_1}(z_1) \dots f_{i_k}(z_k), \qquad c = f_{i_{k+1}}(z_{k+1}) \dots f_{i_{k+l}}(z_{k+l})
$$
where $R \in \CA^+$, $i_1,\dots, i_{k+l} \in I$. Indeed, the formula above is a consequence of the fact that formula \eqref{eqn:pair shuf generators series} is deduced by iterating \eqref{eqn:bialg pair 1}, as we have already shown.

\medskip

\noindent It therefore remains to prove \eqref{eqn:bialg pair 2} for:
\begin{equation}
\label{eqn:setting 1}
  a = R, \qquad b = R', \qquad c = f_{i_1}(z_1) \dots f_{i_k}(z_k)
\end{equation}
where $R \in \CA_{\bl}$, $R' \in \CA_{\bl'}$, $i_1,\dots, i_{k} \in I$. We have:
\begin{multline}
\label{eqn:left hand side 1}
  \langle ab, c \rangle \sim
  \mathop{\sum^{S \subseteq \{1,\dots,k\}}_{\bl = \sum_{s \in S} \alpha_{i_s}}}_{\bl' = \sum_{t \notin S} \alpha_{i_t}} \\ \left[ \frac {R(\{z_s\}_{s \in S}) R'(\{z_t\}_{t \notin S}) \prod_{s \in S, t \notin S} \zeta_{i_si_t} \left(\frac {z_s}{z_t} \right)}
          {\prod_{1 \leq s < t \leq k} \zeta_{i_s i_t} \left( \frac {z_{s}}{z_{t}} \right)}  \right]_{|z_1| \ll \dots \ll |z_k|}
\end{multline}
on account of the shuffle product \eqref{eqn:mult} summing over all ways to distribute the variables $z_1,\dots,z_k$ among the arguments of the rational functions $R$ and $R'$. However, if $\{z_s\}_{s \in S}$ and $\{z_t\}_{t \notin S}$ (for some set $S \subseteq \{1,\dots, k\}$) are to be plugged into the arguments of $R$ and $R'$ respectively, we need $\bl = \sum_{s \in S} \alpha_s$ and $\bl' = \sum_{t \notin S} \alpha_t$ for degree reasons. Meanwhile, we have:
$$
  \langle b\otimes a, \Delta(c) \rangle \stackrel{\eqref{eqn:cop f}}= \sum_{S \subseteq \{1,\dots,k\}} \left \langle R, \prod_{s=1}^k \begin{cases} f_{i_s}(z_s) &\text{if } s \in S \\ \ph^-_{i_s}(z_s) &\text{if } s \notin S \end{cases} \right \rangle  \left \langle R', \prod_{t=1}^k \begin{cases} 1 &\text{if } t \in S \\ f_{i_t}(z_t) &\text{if } t \notin S \end{cases} \right \rangle
$$
$$
  \stackrel{\eqref{eqn:rel 2 affine}}= \sum_{S \subseteq \{1,\dots,k\}} \left \langle R, \prod_{s \in S}^{\rightarrow} f_{i_s}(z_s) \prod^\rightarrow_{t \notin S} \ph^-_{i_{t}}(z_{t}) \prod_{S \ni s > t \notin S} \frac {\zeta_{i_{s}i_t} \left(\frac {z_{s}}{z_t} \right)}{\zeta_{i_ti_{s}} \left(\frac {z_t}{z_{s}} \right)} \right \rangle \left \langle R', \prod^{\rightarrow}_{t \notin S} f_{i_t}(z_t) \right \rangle
$$
where $\prod^\rightarrow_s$ denotes the product in increasing order of $s$. The pairings in the preceding line are non-zero only if the set $S$ is such that $\bl = \sum_{s \in S} \alpha_{i_s}$ and $\bl' = \sum_{t \notin S} \alpha_{i_t}$. In this case, \eqref{eqn:pair shuf extended generators} implies that:
\begin{multline}
\label{eqn:left hand side 2}
\langle b\otimes a, \Delta(c) \rangle \sim \mathop{\sum^{S \subseteq \{1,\dots,k\}}_{\bl = \sum_{s \in S} \alpha_{i_s}}}_{\bl' = \sum_{t \notin S} \alpha_{i_t}}  \\
\left[  \frac {R(\{z_s\}_{s \in S}) R'(\{z_t\}_{t \notin S}) \prod_{S \ni s > t \notin S} \zeta_{i_si_{t}} \left(\frac {z_s}{z_{t}} \right)}{\prod_{S \ni s < s' \in S} \zeta_{i_s i_{s'}} \left( \frac {z_{s}}{z_{s'}}\right)\prod_{S \not \ni t < t' \notin S} \zeta_{i_t i_{t'}} \left( \frac {z_{t}}{z_{t'}}\right) \prod_{S \ni s > t \notin S} \zeta_{i_ti_s} \left(\frac {z_{t}}{z_s} \right)}   \right]_{\substack{\dots \ll |z_t| \ll |z_{t'}| \ll \dots \ll |z_s| \ll |z_{s'}| \ll \dots \\ \forall s<s' \in S, t < t' \notin S}}
\end{multline}
Note that the right-hand sides of \eqref{eqn:left hand side 1} and \eqref{eqn:left hand side 2} are actually equal, despite the apparent difference in the order of the variables in the power series expansions; this is because the only denominators which might have produced a genuine difference would have been $z_s - q^*z_t$ for various $S \ni s < t \notin S$ and $* \in \BZ$, and these do not arise. Thus, we conclude the required identity $\langle ab, c \rangle = \langle b \otimes a, \Delta(c) \rangle$.
\end{proof}

\medskip


\subsection{}
\label{sub:injectivity}

We note the following immediate consequence of formula \eqref{eqn:pair shuf generators}.

\medskip

\begin{proposition}
\label{prop:non-degenerate shuf}
The pairing \eqref{eqn:pair shuf 1} is non-degenerate in the first argument:
$$
  \Big\langle R, - \Big\rangle = 0 \quad \Rightarrow \quad R = 0
$$
for any $R \in \CA^\geq$.
\end{proposition}

\medskip

\begin{proof}
Because of \eqref{eqn:a geq}, elements of $\CA^\geq$ are linear combinations of $R \cdot \ph^+$, where:
$$
  R \in \CA^+ \qquad \text{and} \qquad
  \ph^+ \in \BQ(q) \left[(\ph_{i,0}^+)^{\pm 1}, \ph^+_{i,1},  \ph^+_{i,2}, \dots \right]_{i \in I}
$$
As a consequence of the bialgebra pairing properties~(\ref{eqn:bialg pair 1},~\ref{eqn:bialg pair 2}),
it is easy to see that:
$$
  \Big\langle R\ph^+, x \ph^- \Big\rangle = \Big\langle R, x \Big\rangle \cdot \Big\langle \ph^+, \ph^- \Big\rangle
$$
for any $x \in \UUm$ and $\ph^-$ a product of $\ph^-_{i,d}$'s. Thus the non-degeneracy of the
pairing~\eqref{eqn:pair shuf 1} is a consequence of the non-degeneracy of its restriction:
\begin{equation}
\label{eqn:restricted pairing}
  \langle \cdot, \cdot \rangle \colon \ \CA^+ \otimes \UUm \longrightarrow \BQ(q)
\end{equation}
(indeed, the pairing between $\ph$'s is easily seen to be non-degenerate, due to the explicit formula
\eqref{eqn:aff pair 2}). However, the non-degeneracy of \eqref{eqn:restricted pairing} in the first argument is
an immediate consequence of formula \eqref{eqn:pair shuf generators}: if $R$ is a non-zero rational function,
then we simply choose an arbitrary order of its variables $|z_{1}| \ll \dots \ll |z_{k}|$, and consider the leading
order term of $R$ when expanded as a power series in this particular order. On one hand, the coefficient of this leading order term must
be non-zero, but on the other hand, it is of the form in the right-hand side of \eqref{eqn:pair shuf generators}.
\end{proof}

\medskip

\noindent
We note that the pairings~(\ref{eqn:bialg pair affine}) and~(\ref{eqn:pair shuf 1}) are compatible, in the sense that:
\begin{equation}
\label{eqn:compatibility of pairings}
  \Big  \langle x, y \Big \rangle = \Big \langle \Upsilon(x), y \Big \rangle
\end{equation}
for all $x\in \UUg$ and $y\in \UUl$. Indeed, both sides of~\eqref{eqn:compatibility of pairings} define bialgebra pairings:
$$
 \UUg\otimes\UUl \longrightarrow \BQ(q)
$$
which coincide on the generators, thus must be  equal as a consequence of~(\ref{eqn:bialg pair 1},~\ref{eqn:bialg pair 2}).

\medskip

\noindent
Combining~(\ref{eqn:compatibility of pairings}) with Propositions~\ref{cor:injective},~\ref{prop:non-degenerate shuf},
we thus obtain the non-degeneracy statement of Proposition~\ref{prop:non-degenerate qaff pairing} (strictly speaking,
we obtain the aforementioned non-degeneracy statement only in the first argument, but the case of the second argument
is treated by simply switching the roles of $+$ and $-$ everywhere).

\medskip


\subsection{}
\label{sub:shuffle pairing}

Once Theorem~\ref{thm:main 2} will be proved, Proposition~\ref{prop:pair shuf} can be construed as the existence
of a bialgebra pairing (which is non-degenerate by Proposition \ref{prop:non-degenerate qaff pairing}):
$$
  \langle \cdot , \cdot \rangle \colon \CA^\geq \otimes \CA^\leq \longrightarrow \BQ(q)
$$
Hence, we may construct the Drinfeld double:
\begin{equation}
\label{eqn:triangular a}
  \CA \ := \ \CA^\geq \otimes \CA^\leq/ (\ph_{i,0}^+ \otimes \ph_{i,0}^- - 1 \otimes 1)
\end{equation}
Since all the structures (product, coproduct, and pairing) are preserved by $\Upsilon$, we conclude that
Proposition~\ref{prop:upsilon geq} and Theorem~\ref{thm:main 2} imply the following result.

\medskip

\begin{theorem}
\label{thm:upsilon}
There exists a bialgebra isomorphism:
\begin{equation}
\label{eqn:upsilon}
  \UU \stackrel{\Upsilon}\longrightarrow \CA
\end{equation}
which maps
$$
  e_{i,d} \mapsto z_{i1}^d \in \CA^+,\ f_{i,d} \mapsto z_{i1}^d \in \CA^-,\ \ph_{i,r}^\pm \mapsto \ph_{i,r}^\pm
$$
\end{theorem}

\medskip


\subsection{}
\label{sub:iota}

Let us consider the linear map:
\begin{equation}
\label{eqn:iota}
  \CA^+ \stackrel{\iota}\longrightarrow \hCF
\end{equation}
given by the following formula:
\begin{equation}
\label{eqn:def iota}
  \iota(R) \ =
  \mathop{\sum_{i_1, \dots, i_k \in I}}_{d_1, \dots, d_k \in \BZ}
  \left[ \prod_{a=1}^k (q_{i_a}^{-1}-q_{i_a}) \right]
  \Big \langle R, f_{i_1,-d_1} \dots f_{i_k,-d_k} \Big \rangle \cdot
  \left[i_1^{(d_1)} \dots\, i_k^{(d_k)} \right]
\end{equation}
for all $R \in \CA^+_{\bk}$, where $k = |\bk|$. Because of \eqref{eqn:pair shuf generators}, we have the explicit formula:
\begin{equation}
\begin{split}
  & \iota(R) \ =
    \mathop{\sum_{i_1, \dots, i_k \in I}}_{d_1, \dots, d_k \in \BZ}
    \left[i_1^{(d_1)} \dots\, i_k^{(d_k)} \right] \cdot \\
  & \qquad \qquad \qquad \qquad  \int_{|z_{1}| \ll \dots \ll |z_{k}|}
    \frac {R(z_{1}, \dots, z_{k}) z_{1}^{-d_1} \dots z_{k}^{-d_k}}
          {\prod_{1 \leq a < b \leq k} \zeta_{i_a i_b} (z_{a}/z_{b})} \prod_{a=1}^k D z_{a}
   \label{eqn:iota explicit}
\end{split}
\end{equation}
where all sequences $i_1,\dots,i_k\in I$ that appear in the formula above satisfy:
$$
  \bs_{i_1}+\dots+\bs_{i_k}=\bk
$$
and each variable $z_{a}$ is plugged into an argument of color $i_a$ of the function $R$ (since the latter is
color-symmetric, the result is independent of any choices made). It is easy to see that $\iota(R)$ indeed lands
in the completion~(\ref{eqn:completion fix},~\ref{eqn:infinite sums}).

\medskip

\noindent
As a consequence of the non-degeneracy of the pairing \eqref{eqn:restricted pairing} in the first argument,
we conclude that $\iota$ is injective. Comparing~(\ref{eqn:comm phi}) with~(\ref{eqn:affine shuffle comm phi}),
we can further extend~(\ref{eqn:iota}) to an algebra homomorphism:
\begin{equation}
\label{eqn:iota ext}
  \CA^\geq \stackrel{\iota}\hooklongrightarrow \hCFext
\end{equation}
sending $\ph^+_{i,r}\mapsto \ph^+_{i,r}$.

\medskip

\begin{proposition}
\label{prop:shuf hom}
The map $\iota$ of \eqref{eqn:iota ext} is a bialgebra homomorphism.
\end{proposition}

\medskip

\begin{proof}
The first thing we need to prove is that $\iota$ is an algebra homomorphism. Since the multiplicative relations
involving the $\ph_{i,r}^+$'s are the same for the domain and target of~\eqref{eqn:iota ext} (this is so by design),
then it suffices to show that the map~\eqref{eqn:iota} is an algebra homomorphism. In other words, we must show that
$\iota$ intertwines the product~\eqref{eqn:mult} on $\CA^+$ with the product~\eqref{eqn:shuf affine} on $\hCF$.
To this end, consider any $F\in \CA_{\bk}$, $G\in \CA_{\bl}$ and let $k = |\bk|$, $l = |\bl|$.
According to~\eqref{eqn:iota explicit}, $\iota(F * G)$ equals:
$$
  \mathop{\sum_{s_1, \dots, s_{k+l} \in I}}_{t_1, \dots, t_{k+l} \in \BZ}
  \left[s_1^{(t_1)} \dots\, s_{k+l}^{(t_{k+l})} \right] \int_{|z_{1}| \ll \dots \ll |z_{k+l}|}
  \frac {(F * G)(z_{1}, \dots , z_{k+l}) z_{1}^{-t_1} \dots z_{k+l}^{-t_{k+l}}}
        {\prod_{1 \leq a < b \leq k+l} \zeta_{s_a s_b} (z_{a}/z_{b})} \prod_{a=1}^{k+l} D z_{a}
$$
where we implicitly assume that $s_1, \dots, s_{k+l}\in I$ are acceptable in the sense that:
$$
  \bs_{s_1}+\dots+\bs_{s_{k+l}} = \bk + \bl
$$
According to the definition of the shuffle product in \eqref{eqn:mult}, we have:
$$
  (F * G)(z_{1},\dots,z_{k+l}) \ = \sum^{\text{acceptable partitions}}_{A \sqcup B = \{1,\dots,k+l\}}
  F \left( \{ z_{a} \}_{a \in A} \right) G \left( \{ z_{b} \}_{b \in B} \right)
  \prod_{a \in A, b \in B} \zeta_{s_a s_b} \left( \frac {z_{a}}{z_{b}} \right)
$$
where a partition $A \sqcup B = \{1,\dots,k+l\}$ is called acceptable if the number of variables of each color
in the set $A$ (resp.\ $B$) is equal to the number of variables of that color of the rational function $F$ (resp.\ $G$).
With this in mind, we conclude:
\begin{multline}
\label{eqn:gos}
  \iota(F * G) \ = \mathop{\sum_{s_1, \dots, s_{k+l} \in I}}_{t_1, \dots, t_{k+l} \in \BZ}
  \left[s_1^{(t_1)} \dots\, s_{k+l}^{(t_{k+l})} \right] \sum^{\text{acceptable partitions}}_{A \sqcup B = \{1,\dots,k+l\}}
  \int_{|z_{1}| \ll \dots \ll |z_{k+l}|} \\
  \frac {F \left( \{ z_{a} \}_{a \in A} \right) \prod_{a \in A} z_a^{-t_a}}
        {\prod_{a < a' \in A} \zeta_{s_{a} s_{a'}} (z_{a}/z_{a'})} \cdot
  \frac {G \left( \{ z_{b} \}_{b \in B} \right) \prod_{b \in B} z_b^{-t_b}}
        {\prod_{b < b' \in B} \zeta_{s_{b} s_{b'}} (z_{b}/z_{b'})}
  \prod_{A \ni a > b \in B} \frac {\zeta_{s_a s_b} \left( \frac {z_{a}}{z_{b}} \right)}
                                  {\zeta_{s_b s_a} \left( \frac {z_{b}}{z_{a}} \right)}
  \prod_{a=1}^{k+l} D z_{a}
\end{multline}
For various $a \in A$ and $b \in B$, the expression above has poles involving $z_a$ and $z_b$ only if $a>b$.
This implies that the value of the integral above is unchanged if we move the variables in such a way that
all the $z_{a}$'s with $a \in A$ are much greater than all the $z_b$'s with $b \in B$. In other words we may replace:
$$
  \int_{|z_{1}| \ll \dots \ll |z_{k+l}|} \quad \text{by} \quad \int_{|x_1| \ll \dots \ll |x_l| \ll |y_1| \ll \dots \ll |y_k|}
$$
where $x_1, \dots, x_l$ (resp.\ $y_1, \dots, y_k$) are simply relabelings of the variables $\{z_b\}_{b\in B}$
in the increasing order of $b$ (resp.\ $\{z_a\}_{a\in A}$ in the increasing order of $a$). Moreover, let
$i_1, \dots, i_k, d_1, \dots, d_k$ (resp.\ $j_1, \dots, j_l,e_1,\dots, e_l$) refer to those of the elements
$s_c\in I$ and $t_c \in \BZ$ for $c \in A$ (resp.\ $c \in B$), as in formula~\eqref{eqn:cases affine}.
It is straightforward to see that applying the shuffle product \eqref{eqn:shuf affine} to $\iota(F)$ and $\iota(G)$
gives us precisely \eqref{eqn:gos}. Therefore, $\iota(F * G)=\iota(F) * \iota(G)$, as claimed.

\medskip

\noindent
The second thing we need to prove is that the map $\iota$ is a coalgebra homomorphism, i.e.\ that it intertwines the
coproduct~\eqref{eqn:coproduct1} on $\CA^\geq$ with the coproduct~\eqref{eqn:cop affine shuffle} on $\hCFext$.
To this end, consider any $R \in \CA_{\bk}$ and note that~\eqref{eqn:coproduct1} reads:
$$
  \Delta(R) \ = \sum_{\bl=\sum_{i\in I} l_i\alpha_i \in Q^+}^{l_i\leq k_i} \sum_{\pi_{ia} \geq 0}
  \frac { \prod_{i \in I}^{a>l_i} \ph^+_{i,\pi_{ia}}  * R(z_{i, a \leq l_i} \otimes z_{i, a>l_i})
          \prod_{i\in I}^{a > l_i} z_{ia}^{-\pi_{ia}}}
        {\prod_{i,i' \in I} \prod_{a \leq l_{i}}^{a' > l_{i'}} \zeta_{i'i}(z_{i'a'}/z_{ia})}
$$
where the second sum is over all collections of non-negative integers $\{\pi_{ia}\}_{i\in I}^{l_i<a\leq k_i}$.
Applying the map $\iota \otimes \iota$ to the above expression, we obtain by \eqref{eqn:iota explicit}:
\begin{equation*}
\begin{split}
  &  (\iota \otimes \iota)(\Delta(R)) = \\
  & \mathop{\mathop{\mathop{\sum_{i_1, \dots, i_{k} \in I}}_{d_1, \dots, d_{k} \in \BZ}}_
    {\pi_{l+1}, \dots, \pi_{k} \geq 0}}^{0\leq l\leq k}
    \ph^+_{i_{l+1},\pi_{l+1}} \dots \ph^+_{i_{k}, \pi_{k}} \left[i_1^{(d_1)} \dots\, i_l^{(d_l)} \right] \otimes
    \left[i_{l+1}^{(d_{l+1})} \dots\, i_{k}^{(d_{k})} \right] \, \cdot \\
  & \int_{|z_1| \ll \dots \ll |z_{k}|}
    \frac { R(z_1, \dots, z_{k})\prod_{a=l+1}^{k} z_a^{-\pi_a} \prod_{a=1}^{k} z_a^{-d_a} D z_a }
          { \prod_{1\leq a < b \leq l} \zeta_{i_ai_b}(z_a/z_b)\prod_{l < a < b \leq k} \zeta_{i_ai_b}(z_a/z_b)
            \prod_{a \leq l < b} \zeta_{i_bi_a}(z_b/z_a)}
\end{split}
\end{equation*}
If we substitute $d_a \mapsto d_a - \pi_a$ for $a \in \{l+1, \dots, k\}$ in the above relation,
and use~\eqref{eqn:affine shuffle comm phi} to commute the product of $\ph$'s to the right of the word
$[i_1^{(d_1)} \dots\, i_l^{(d_l)}]$, then we obtain precisely formula \eqref{eqn:cop affine shuffle}
for $\Delta(\iota(R))$, as required.
\end{proof}

\medskip


\subsection{}
\label{sub:iota-image}

As:
$$
  \iota(\Upsilon(e_{i,d})) = \left[ i^{(d)} \right] =\, \wPhi(e_{i,d})
$$
for any $i\in I$ and $d\in \BZ$, the composition of the maps~\eqref{eqn:upsilon +} and~\eqref{eqn:iota}
recovers~\eqref{eqn:loop to shuffle}:
\begin{equation}
\label{eqn:composition}
  \wPhi \colon \ \UUp \stackrel{\Upsilon}\longrightarrow \CA^+ \stackrel{\iota}\longrightarrow \hCF
\end{equation}
The main result of this Section, Theorem~\ref{thm:main 2}, states that the map $\Upsilon$ is an isomorphism,
so it would naturally imply that the image of $\wPhi$ is equal to the image of~$\iota$.
Therefore, let us characterize the latter, by analogy with \eqref{eqn:image}--\eqref{eqn:leclerc image}.

\medskip

\begin{proposition}
\label{prop:image affine}
We have:
\begin{equation}
\label{eqn:image affine}
  \emph{Im}\, \iota = \left\{ \mathop{\mathop{\sum_{i_1, \dots ,i_k \in I}}_{d_1, \dots ,d_k \in \BZ}}^{k\in \BN}
  \gamma \begin{pmatrix} i_1 & \dots & i_k \\ d_1 & \dots & d_k \end{pmatrix} \cdot
  \left[i^{(d_1)}_1 \dots\, i^{(d_k)}_k \right] \right\}
\end{equation}
where the scalars $\gamma \begin{pmatrix} i_1 & \dots & i_k \\ d_1 & \dots & d_k \end{pmatrix} \in \BQ(q)$ vanish for all
but finitely many values of $(\bk,d)=(\alpha_{i_1}+\dots+\alpha_{i_k},d_1+\dots+d_k)\in Q^+\times \BZ$ and
satisfy equations~(\ref{eqn:gamma boundedness})--(\ref{eqn:constraint 3}):
\begin{equation}\label{eqn:gamma boundedness}
  \exists M \ \mathrm{s.t.}\ \gamma\begin{pmatrix} i_1 & \dots & i_k \\ d_1 & \dots & d_k \end{pmatrix} = 0
  \ \mathrm{if}\ d_1+\dots+d_a<M \ \mathrm{for\ some}\ 1\leq a < k
\end{equation}
\begin{multline}
\label{eqn:constraint 1}
  \gamma \begin{pmatrix} w & i & j & w' \\ \chi & r-1 & s & \chi' \end{pmatrix} -
  \gamma \begin{pmatrix} w & i & j & w' \\ \chi & r & s-1 & \chi' \end{pmatrix} q^{-d_{ij}} = \\
  \qquad \qquad \qquad
  \gamma \begin{pmatrix} w & j & i & w' \\ \chi & s & r-1 & \chi' \end{pmatrix} q^{-d_{ij}} -
  \gamma \begin{pmatrix} w & j & i & w' \\ \chi & s-1 & r & \chi' \end{pmatrix}
\end{multline}
for all $i,j \in I$ and $r,s \in \BZ$. Moreover:
\begin{multline}
\label{eqn:constraint 2}
  \sum_{\sigma \in S(1-a_{ij})} \sum_{k=0}^{1-a_{ij}} (-1)^k {1-a_{ij} \choose k}_i \cdot \\ \
  \gamma \begin{pmatrix} w & i & \dots & i & j & i & \dots & i & w'  \\
                        \chi & p_{\sigma(1)} & \dots & p_{\sigma(k)} & t & p_{\sigma(k+1)} & \dots &
  p_{\sigma(1-a_{ij})} & \chi' \end{pmatrix} = 0
\end{multline}
for all distinct $i,j \in I$ and $p_1,\dots,p_{1-a_{ij}},t\in \BZ$. In the formulas above, $w,w'$ denote
arbitrary finite words and $\chi,\chi'$ denote arbitrary collections of integers, so that $(w,\chi), (w',\chi')$
encode a pair of arbitrary loop words. Finally, we require:
\begin{multline}
\label{eqn:constraint 3}
  \mathop{\sum_{\e_{ab} \in \{0,1\},}}_{\forall\, 1 \leq a < b \leq k}
  \prod_{a<b}^{\e_{ab} = 1} (-q^{-d_{i_ai_b}})\cdot \\
  \gamma\begin{pmatrix}
    \dots & i_a & \dots \\ \dots & d_a - \#\{b>a|\e_{ab} = 0\} - \# \{b<a|\e_{ba} = 1\} & \dots
  \end{pmatrix} =\, 0
\end{multline}
for all but finitely many $(d_1, \dots, d_k) \in \BZ^k$ (note that there are only finitely many choices of
$i_1, \dots, i_k \in I$ in formula \eqref{eqn:constraint 3}, because $I$ is a finite set).
\end{proposition}

\medskip

\begin{proof}
Consider any $R \in \CA_{\bk,d}$ and set $k=|\bk|$. Since $\iota$ is injective, $\iota(R)$ is completely determined by
the collection of $\gamma(\dots) \in \BQ(q)$ that appear in \eqref{eqn:image affine}, which can be thought of as a function:
\begin{equation}
\label{eqn:function gamma}
  \gamma \colon
  \left\{ \begin{pmatrix} i_1 & \dots & i_k \\ d_1 & \dots & d_k \end{pmatrix}
          \text{ s.t. } \sum_{a=1}^k \bs_{i_a} = \bk,\, \sum_{a=1}^k d_a = d \right\}
  \longrightarrow \BQ(q)
\end{equation}
subject to the constraint~(\ref{eqn:gamma boundedness}).

\medskip

\noindent
For any $1 \leq a < b \leq k$, consider the following operator on the set of such functions:
\begin{equation}
\label{eqn:tau}
\begin{split}
  & \tau_{ab}(\gamma)
    \begin{pmatrix} \dots & i_a & \dots & i_{b} & \dots \\ \dots & d_a & \dots & d_{b} & \dots \end{pmatrix} = \\
  & \gamma \begin{pmatrix}
      \dots & i_a & \dots & i_{b} & \dots \\ \dots & d_a - 1 & \dots & d_{b} & \dots \end{pmatrix} -
    \gamma \begin{pmatrix}
      \dots & i_a & \dots & i_{b} & \dots \\ \dots & d_a & \dots & d_{b} - 1 & \dots \end{pmatrix} q^{- d_{i_a i_b}}
\end{split}
\end{equation}
It is easy to see that the various operators $\tau_{ab}$ commute with each other.
This notion is motivated by the obvious observation that if a function $\gamma$ encodes
the coefficients of $\iota(R)$:
\begin{equation}
\label{eqn:int0}
  \gamma \begin{pmatrix} i_1 & \dots & i_k \\ d_1 & \dots & d_k \end{pmatrix} =  \int_{|z_{1}| \ll \dots \ll |z_{k}|}
  \frac {R(z_{1},\dots,z_{k}) z_{1}^{-d_1} \dots z_{k}^{-d_k}}
        {\prod_{1 \leq a < b \leq k} \zeta_{i_a i_b} (z_{a}/z_{b})} \prod_{a=1}^k D z_{a}
\end{equation}
then:
\begin{multline}
\label{eqn:int1}
  \tau_{c,c+1}(\gamma) \begin{pmatrix} \dots & i_c & i_{c+1} & \dots \\ \dots & d_c & d_{c+1} & \dots \end{pmatrix} = \\
  \int_{ \dots \ll |z_c| \ll |z_{c+1}| \ll \dots}
  \frac { R(z_{1},\dots,z_{k}) (z_{c} - z_{c+1}) z_{1}^{-d_1} \dots z_{k}^{-d_k} }
        { \prod_{1 \leq a < b \leq k, (a,b) \neq (c,c+1)} \zeta_{i_a i_b} (z_{a}/z_{b}) } \prod_{a=1}^k D z_{a}
\end{multline}
Similarly, \eqref{eqn:int0} implies:
\begin{multline}
\label{eqn:int2}
  - \tau_{c,c+1}(\gamma)
  \begin{pmatrix} \dots & i_{c+1} & i_{c} & \dots \\ \dots & d_{c+1} & d_{c} & \dots \end{pmatrix} = \\
  \int_{ \dots \ll |z_{c+1}| \ll |z_{c}| \ll \dots}
  \frac { R(z_{1},\dots,z_{k}) (z_{c} - z_{c+1}) z_{1}^{-d_1} \dots z_{k}^{-d_k} }
        { \prod_{1 \leq a < b \leq k, (a,b) \neq (c,c+1)} \zeta_{i_a i_b} (z_{a}/z_{b}) } \prod_{a=1}^k D z_{a}
\end{multline}
The right-hand sides of~\eqref{eqn:int1} and~\eqref{eqn:int2} have the same integrand. Moreover, because
elements $R \in \CA^+$ only have poles as prescribed in~\eqref{eqn:shuf}, the integrand in question has no poles
involving $z_c$ and $z_{c+1}$. Therefore, one may change the order of variables in the integral from
$|z_c| \ll |z_{c+1}|$ to $|z_{c+1}| \ll |z_c|$ without changing the value of the integral,
which implies that the right-hand sides of~\eqref{eqn:int1} and~\eqref{eqn:int2} are equal.
Hence, we conclude that if a function $\gamma$ as in~\eqref{eqn:function gamma} encodes
the coefficients of $\iota(R)$ for some $R \in \CA^+$, then:
\begin{equation}
\label{eqn:constraint 1 explicit}
  \tau_{c,c+1}(\gamma) \begin{pmatrix} \dots & i_c & i_{c+1} & \dots \\ \dots & d_c & d_{c+1} & \dots \end{pmatrix} =
  - \tau_{c,c+1}(\gamma) \begin{pmatrix} \dots & i_{c+1} & i_{c} & \dots \\ \dots & d_{c+1} & d_{c} & \dots \end{pmatrix}
\end{equation}
for all $c$, which is precisely the linear constraint~\eqref{eqn:constraint  1}.

\medskip

\noindent
Going further, one may iterate the process of going from \eqref{eqn:int0} to \eqref{eqn:int1}
a number of $\frac{k(k-1)}{2}$ times, obtaining:
\begin{multline*}
  \left( \prod_{1 \leq a < b \leq k} \tau_{ab}\right) (\gamma)
  \begin{pmatrix} i_1 & \dots & i_k \\ d_1 & \dots & d_k \end{pmatrix} = \\
  \int_{|z_1| \ll \dots \ll |z_k|}
  R(z_{1},\dots,z_{k}) z_{1}^{-d_1} \dots z_{k}^{-d_k} \prod_{1 \leq a < b \leq k} (z_a - z_b) \prod_{a=1}^k D z_{a}
\end{multline*}
The product $R(z_{1},\dots,z_{k}) \prod_{1 \leq a < b \leq k} (z_a - z_b)$ is a Laurent polynomial,
due to~\eqref{eqn:shuf}, hence the integral above vanishes for all but finitely many values of $(d_1, \dots, d_k)$:
\begin{multline}
\label{eqn:constraint 2 explicit}
  \left( \prod_{1 \leq a < b \leq k} \tau_{ab}\right) (\gamma)
  \begin{pmatrix} i_1 & \dots & i_k \\ d_1 & \dots & d_k \end{pmatrix} = 0 \\
  \mathrm{for\ all\ but\ finitely\ many}\ (d_1, \dots, d_k)\in \BZ^k
\end{multline}
Unpacking the definition of $\tau$ in~\eqref{eqn:tau}, we see that identity~(\ref{eqn:constraint 2 explicit}) is
precisely equivalent to the linear constraint~\eqref{eqn:constraint 3}.

\medskip

\noindent
Finally, let us consider the linear combination in the left-hand side of~\eqref{eqn:constraint 2}
(to keep our notation simple, we will assume that the words $w$ and $w'$ are vacuous, as this
will not interfere with our argument) and replace all the $\gamma$'s therein
by the right-hand sides of \eqref{eqn:int0}. We obtain the following equality:
\begin{multline*}
  \text{Sym} \left[  \sum_{k=0}^{1-a_{ij}} (-1)^k {1-a_{ij} \choose k}_i\, \cdot
  \int_{|z_1| \ll \dots \ll |z_k| \ll |w| \ll |z_{k+1}| \ll \dots \ll |z_{1-a_{ij}}|} \right. \\
  \left.
  \frac {R(z_{1},\dots,z_{1-a_{ij}},w) z_{1}^{-p_1} \dots z_{1-a_{ij}}^{-p_{1-a_{ij}}} w^{-t}  Dz_1 \dots Dz_{1-a_{ij}} Dw}
        {\prod_{b=1}^{k} \zeta_{ij} (z_b/w) \prod_{b=k+1}^{1-a_{ij}} \zeta_{ji}(w/z_b)
         \prod_{1 \leq b < c \leq 1-a_{ij}} \zeta_{ii} (z_{b}/z_{c})} \right]= 0
\end{multline*}
where $\text{Sym} [\dots]$ denotes symmetrization with respect to the $z$-variables. In the formula above,
let us write the rational function $R$ in terms of the Laurent polynomial $r$ of \eqref{eqn:shuf}:
\begin{multline}
\label{eqn:constraint 3 explicit}
  \text{Sym} \left[  \sum_{k=0}^{1-a_{ij}} {1-a_{ij} \choose k}_i\, \cdot
  \int_{|z_1| \ll \dots \ll |z_k| \ll |w| \ll |z_{k+1}| \ll \dots \ll |z_{1-a_{ij}}|} \right. \\
  \left.
  \frac {r(z_{1},\dots,z_{1-a_{ij}},w) z_{1}^{-p_1} \dots z_{1-a_{ij}}^{-p_{1-a_{ij}}} w^{-t} Dz_1 \dots Dz_{1-a_{ij}} Dw}
        { \prod_{b=1}^k (z_b - w q_i^{-a_{ij}}) \prod_{b=k+1}^{1-a_{ij}} ( w - z_b q_i^{-a_{ij}} )
          \prod_{1 \leq b < c \leq 1-a_{ij}} \zeta_{ii} (z_{b}/z_{c})} \right]= 0
\end{multline}
We claim  that formula \eqref{eqn:constraint 3 explicit} is equivalent to~\eqref{eqn:wheel}, due to
the combinatorial identity between power series expansions of rational functions and certain formal $\delta$
functions established in~\cite[Proposition 4]{E1} (proved in full generality in~\cite[Theorem 1.1]{DJ}). Indeed,
the validity of~\eqref{eqn:constraint 3 explicit} for all $p_1,\dots,p_{1-a_{ij}},t\in \BZ$ is equivalent to the equality:
\begin{equation}
\label{eqn:constraint 3 explicit 2}
  0 \ = \ r(z_{1},\dots,z_{1-a_{ij}},w)\, \cdot
\end{equation}
$$
  \text{Sym} \left[  \sum_{k=0}^{1-a_{ij}} {1-a_{ij} \choose k}_i\, \cdot
  \prod_{b=1}^k \frac{1}{w - q_i^{a_{ij}} z_b }
  \prod_{b=k+1}^{1-a_{ij}} \frac{1}{z_b - q_i^{a_{ij}} w }
  \prod_{1 \leq b < c \leq 1-a_{ij}} \frac{z_c-z_b}{z_c-q_i^2z_b} \right]
$$
where all rational functions $\frac{1}{x-y}$ above are expanded as formal series $\sum_{r=0}^\infty \frac {y^r}{x^{r+1}}$.
According to~\cite[Theorem 1.1]{DJ}, where we set $m=-a_{ij}$ and $q=q_i^{-1}$, we have:
\begin{multline*}
  \text{Sym} \left[  \sum_{k=0}^{1-a_{ij}} {1-a_{ij} \choose k}_i\, \cdot
  \prod_{b=1}^k \frac{1}{ w - q_i^{a_{ij}} z_b }
  \prod_{b=k+1}^{1-a_{ij}} \frac{1}{ z_b - q_i^{a_{ij}} w }
  \prod_{b < c} \frac{z_c-z_b}{z_c - q_i^2z_b} \right] = \\
  q_i^{1+a_{ij}} \text{Sym} \left[ \delta(w,q_i^{-a_{ij}}z_1)
  \delta(z_1,q_i^{-2}z_2)\delta(z_2,q_i^{-2}z_3)\dots \delta(z_{-a_{ij}},q_i^{-2}z_{1-a_{ij}}) \right]
\end{multline*}
where the formal $\delta$-function $\delta(x,y)$ is defined via:
$$
  \delta(x,y)\ =\ \sum_{r\in \BZ} \frac {y^r}{x^{r+1}} \ =
  \underbrace{\frac{1}{x-y}}_{\text{expanded in } |x| \gg |y|} +
  \underbrace{\frac{1}{y-x}}_{\text{expanded in } |y| \gg |x|}
$$
Since $r$ is a Laurent polynomial, the fundamental property of $\delta$ implies that:
$$
  r(z_{1},\dots,z_{1-a_{ij}},w)\cdot \delta(w,q_i^{-a_{ij}}z_1)
  \delta(z_1,q_i^{-2}z_2)\dots \delta(z_{-a_{ij}},q_i^{-2}z_{1-a_{ij}}) =
$$
$$
   r(z_1,q_i^2z_1,\dots,q_i^{-2a_{ij}}z_1, q_i^{-a_{ij}}z_1) \cdot \delta(w,q_i^{-a_{ij}}z_1)
   \delta(z_1,q_i^{-2}z_2)\dots \delta(z_{-a_{ij}},q_i^{-2}z_{1-a_{ij}})
$$
Thus~\eqref{eqn:constraint 3 explicit 2}, and hence~(\ref{eqn:constraint 3 explicit}), is indeed equivalent to~\eqref{eqn:wheel}.

\medskip

\noindent
Conversely, suppose we have a function \eqref{eqn:function gamma} satisfying properties
\eqref{eqn:gamma boundedness}--\eqref{eqn:constraint 3}.
Our goal is to construct a rational function $R \in \CA_{\bk,d}$ such that \eqref{eqn:int0} holds.

\medskip

\noindent
For any ordered collection $\mathbf{i} = (i_1, \dots, i_k)\in I^k$ with $\bs_{i_1} + \dots + \bs_{i_k} = \bk$,
define a formal bi-infinite power series
  $F_{\mathbf{i}}(z_1,\dots,z_k)\in \BQ(q)[[z_1,z_1^{-1},\dots,z_k,z_k^{-1}]]$
via:
\begin{equation}\label{eqn:F-series}
  F_{\mathbf{i}}(z_1,\dots,z_k) \ = \sum_{d_1,\dots,d_k\in \BZ}
  \gamma \begin{pmatrix} i_1 & \dots & i_k \\ d_1 & \dots & d_k \end{pmatrix} z_1^{d_1}\dots z_k^{d_k}
\end{equation}
Here, we shall think of the variable $z_c$ being of color $i_c$ for all $1\leq c\leq k$.
However, due to~\eqref{eqn:gamma boundedness}, we actually have
\begin{equation}\label{eqn:F-nature}
  F_{\mathbf{i}}(z_1,\dots,z_k)\in \BQ(q)((z_k))\dots ((z_2))((z_1))
\end{equation}
Similarly to~\eqref{eqn:constraint 2 explicit}, property~\eqref{eqn:constraint 3} can be recast as:
\begin{equation*}
  \int_{|z_1| \ll \dots \ll |z_k|}
  F_{\mathbf{i}}(z_{1}, \dots, z_{k}) \ \cdot \prod_{1\leq a<b\leq k} (z_a - z_b q^{-d_{i_a i_b}})
  z_{1}^{-d_1} \dots z_{k}^{-d_k}\prod_{a=1}^k D z_{a} = 0
\end{equation*}
for all but finitely many $(d_1,\dots,d_k)\in \BZ^k$, which is equivalent to:
\begin{equation}\label{eqn:r-polynomials}
  r_{\mathbf{i}}(z_1,\dots,z_k) :=
  F_{\mathbf{i}}(z_1,\dots,z_k)\ \cdot \prod_{1\leq a<b\leq k} (z_a - z_b q^{-d_{i_a i_b}})
\end{equation}
being a Laurent polynomial. Invoking~(\ref{eqn:F-nature}), we conclude that:
\begin{equation}\label{eqn:F-vs-r}
  F_{\mathbf{i}}(z_1,\dots,z_k)=
  \frac{r_{\mathbf{i}}(z_1,\dots,z_k)}{\prod_{1\leq a<b\leq k} (z_a - z_b q^{-d_{i_a i_b}})}
\end{equation}
with the right-hand side expanded in $|z_1|\ll \dots \ll |z_k|$. If we let:
\begin{equation}\label{R-functions}
  R_{\mathbf{i}}(z_1,\dots,z_k) := \frac{r_{\mathbf{i}}(z_1,\dots,z_k)}{\prod_{1 \leq a < b \leq k} (z_a-z_b)}
\end{equation}
then we obtain:
\begin{equation}\label{R-integral}
  \gamma \begin{pmatrix} i_1 & \dots & i_k \\ d_1 & \dots & d_k \end{pmatrix} \, = \, \int_{|z_1| \ll \dots \ll |z_k|}
  \frac{R_{\mathbf{i}}(z_{1},\dots,z_{k}) z_{1}^{-d_1} \dots z_{k}^{-d_k}}
       {\prod_{1\leq a < b \leq k} \zeta_{i_a i_b}(z_a/z_b)} \prod_{a=1}^k D z_{a}
\end{equation}
for all $d_1,\dots, d_k\in \BZ$. Let us now prove that the rational functions $R_{\mathbf{i}}$ actually do not depend on
$\mathbf{i}$. To do so, note that property~\eqref{eqn:constraint 1 explicit} allows us to recast \eqref{eqn:constraint 1}
as:
\begin{multline*}
  \int_{ \dots \ll |z_c| \ll |z_{c+1}| \ll \dots}
  \frac { R_{\mathbf{i}}(z_{1},\dots,z_{k}) (z_{c} - z_{c+1}) z_{1}^{-d_1} \dots z_{k}^{-d_k} }
        { \prod_{1 \leq a < b \leq k, (a,b) \neq (c,c+1)} \zeta_{i_a i_b} (z_{a}/z_{b}) } \prod_{a=1}^k D z_{a} = \\
  \int_{ \dots \ll |z_{c+1}| \ll |z_{c}| \ll \dots}
  \frac { R_{\sigma_c({\mathbf{i}})}(z_{1},\dots,z_{k}) (z_{c} - z_{c+1}) z_{1}^{-d_1} \dots z_{k}^{-d_k} }
        { \prod_{1 \leq a < b \leq k, (a,b) \neq (c,c+1)} \zeta_{i_a i_b} (z_{a}/z_{b}) } \prod_{a=1}^k D z_{a}
\end{multline*}
for all $d_1,\dots,d_k\in \BZ$, where $\sigma_c({\mathbf{i}})=(i_1,\dots,i_{c-1},i_{c+1},i_c,i_{c+2},\dots,i_k)$.
As the integrands above have no poles involving $z_c$ and $z_{c+1}$, we conclude that
$R_{\mathbf{i}} = R_{\sigma_c({\mathbf{i}})}$. Since this holds for all $c \in \{1,\dots,k-1\}$, we conclude that
there exists a unique rational function $R = R_{\mathbf{i}}$, for all $\mathbf{i}$. Moreover, this rational function $R$
must be symmetric in the variables of each color separately, since $\gamma$ of \eqref{eqn:function gamma} is unchanged
if we permute $a$ and $b$ such that $i_a=i_b$ and $d_a=d_b$. Because a rational function which is symmetric in variables
$z$ and $w$ cannot have a simple pole at $z = w$, we conclude that the rational function $R$ thus constructed is
of the form \eqref{eqn:shuf}.

\medskip

\noindent
Finally, the fact that the numerator $r$ of $R$ satisfies the wheel conditions \eqref{eqn:wheel} is equivalent to
\eqref{eqn:constraint 3 explicit}, as we have already seen, which is in turn equivalent to \eqref{eqn:constraint 2}.

\medskip

\noindent
Thus, we have constructed $R\in \CA_{\bk,d}$ such that~\eqref{eqn:int0} holds, as needed.
\end{proof}

\medskip


\subsection{}
\label{sub:proof of Main Theorem 2}

We conclude the present Section with a proof of Theorem~\ref{thm:main 2}.

\medskip

\begin{proof}[Proof of Theorem \ref{thm:main 2}]
According to Proposition~\ref{cor:injective}, the map $\Upsilon\colon \UUp\to \CA^+$ is injective,
hence it remains to prove that it is also surjective. To this end, recall the filtration~\eqref{eqn:filtration},
and consider the following vector subspaces for any loop word~$w$:
$$
  \CA^+_{\leq w} \subset \CA^+
$$
consisting of rational functions $R$ such that the leading order term of $\iota(R)$ is $\leq w$.
It is clear, due to~(\ref{eqn:composition}), that the map $\Upsilon$ restricts to an injection:
\begin{equation}
\label{eqn:upsilon restrict}
  \UUp_{\leq w} \stackrel{\Upsilon}\hooklongrightarrow \CA^+_{\leq w}
\end{equation}
Recall the vector subspace~\eqref{eqn:subspace 1} and consider
the restriction of the pairing \eqref{eqn:restricted pairing}:
\begin{equation}
\label{eqn:pairing restrict}
  \CA^+_{\leq w} \otimes \UUm^{\leq w} \longrightarrow \BQ(q)
\end{equation}
With Proposition \ref{prop:non-degenerate shuf} in mind, we claim that the pairing \eqref{eqn:pairing restrict}
is non-degenerate in the first argument, cf.~Proposition~\ref{prop:non-degenerate}. This claim holds because elements:
$$
  R \in \CA^+_{\leq w}
$$
pair trivially with the basis elements $\{_vf\}_{v>w}$ of~\eqref{eqn:vf elements}, due to~\eqref{eqn:def iota},
and hence also with $\{f_v\}_{v>w}$ of \eqref{eqn:quantum bracketing arbitrary affine}, due to~\eqref{eqn:fv vs vf}.
The non-degeneracy of \eqref{eqn:pairing restrict} implies that:
\begin{equation}
\label{eqn:bounds for thm2}
  \dim \CA^+_{\leq w} \leq \dim \UUm^{\leq w} = \# \Big\{\text{standard loop words } \leq w \Big\}
\end{equation}
(although the dimensions above are technically speaking infinite, they become finite
when we restrict to each $Q^+ \times \BZ$-graded component, see Corollary \ref{cor:finitely many}).
However, the domain of the map~\eqref{eqn:upsilon restrict} has dimension equal to the number of
standard loop words $\leq w$, see Subsection~\ref{sub:dimension for affine pieces}, which together
with~\eqref{eqn:bounds for thm2} implies that the map~\eqref{eqn:upsilon restrict} is an isomorphism.
As $\CA^+=\cup_w \CA^+_{\leq w}$, the surjectivity of $\Upsilon$ follows.
\end{proof}

\medskip


\section{Appendix}
\label{sec:appendix}

In this Appendix, we will provide without proof combinatorial data pertaining to standard Lyndon loop words
in the classical types, associated to the order:
$$
  1 < \dots < n
$$
of the simple roots (we indicate the labeling of the simple roots in all cases below). The analogous computations
can be easily performed for all other orders, and for the exceptional types, using straightforward computer computations
(whose run-time is less than an hour for the most complicated root system, $E_8$).

\medskip

\noindent
By Proposition~\ref{prop:peridocitiy}, the bijection \eqref{eqn:associated word loop}  is completely determined by
$\ell(\alpha,d)$ for $\alpha \in \Delta^+$ and $1 \leq d \leq |\alpha|$. Furthermore, \eqref{eqn:explicit lyndon} states
that all the exponents of the letters of such $\ell(\alpha,d)$ will be $0$ and $1$, so in what follows we will denote them by:
\begin{align*}
  i \dots j  &\text{ instead of } i^{(0)} \dots\, j^{(0)} \\
  \squiggly{i \dots j} &\text{ instead of } i^{(1)} \dots j^{(1)}
\end{align*}
to keep the notation legible. For any letters $a \leq b \geq c$ in the set $\{1,\dots,n\}$,
we will use the following notation in our loop words:
\begin{align}
  & \Big| a \nearrow b \Big| = a, a+1, \dots, b-1, b \label{eqn:not1} \\
  & \Big| b \searrow c \Big| = b, b-1, \dots, c+1, c \label{eqn:not2} \\
  & \Big| a \nearrow b \searrow c \Big| = a, a+1, \dots, b-1, b, b-1, \dots, c+1, c \label{eqn:not3}
\end{align}
and the analogous notations with squiggly underlines under the letters.
If $a = b+1$ or $b = c-1$, the notations \eqref{eqn:not1} and \eqref{eqn:not2} will denote the empty sequences.

\medskip

\noindent
Beside the standard Lyndon loop words corresponding to each positive root, that will be explicitly given for all classical types
in the following Subsections, we will also give ``rooted tree" presentations for the set of all such words following \cite{LR}.

\medskip

\begin{definition}
Consider a rooted tree, whose vertices are either hollow or full, and are labeled by the letters of some alphabet (in our case
$\{a^{(d)}\}_{a \in I, d \in \{0,1\}}$). To any path from the root to a hollow vertex, we associate the word obtained by reading
the labels of all the vertices (be they hollow or full) encountered along the way. The \underline{dictionary} of the tree is the set
of all words associated to all such paths.
\end{definition}

\medskip

\noindent
It is elementary to see that any set of words starting with a given letter can be uniquely represented as the dictionary
of some rooted tree. Then, following \cite{LR}, we will show the rooted trees that produce the standard Lyndon loop words
starting with the letter $a^{(1)}$ for all $a \in I$. This gives a complete list of the set of words
$\{\ell(\alpha,d)\}_{\alpha \in \Delta^+, d \in \{1,\dots,|\alpha|\}}$, although it is not easy to extract from
this presentation the word corresponding to a specific $(\alpha,d)$.

\medskip

\begin{remark}
Extracting the subtrees whose only labels are $\{a^{(1)}\}_{a \in I}$ from the trees associated to the four classical types below,
gives precisely the trees of finite types constructed in \cite[Figure 1]{LR}. In the exceptional types, it's clear why presenting
our trees is unfeasible: while the forest that corresponds to type $E_8$ in \cite{LR} has 120 hollow vertices, our forest would need
to have 1240 hollow vertices.
\end{remark}

\medskip


\subsection{Type $A_n$}
\label{sub:an}

Consider the vertices of the Dynkin diagram as below:
\begin{figure}[H]
\centering
\includegraphics[scale=0.35]{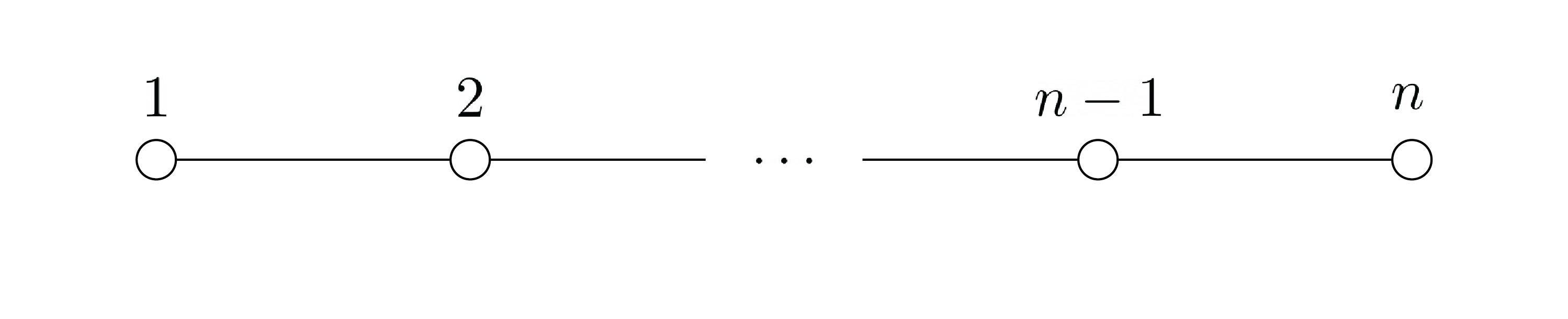}
\end{figure}

\noindent
The set of positive roots $\Delta^+$ consists of the elements:
\begin{equation}
\label{eqn:an}
  \alpha_{ij} = \alpha_i + \dots + \alpha_j
\end{equation}
for all $1 \leq i \leq j \leq n$. Then the bijection \eqref{eqn:associated word loop} is given by:
\begin{equation}
\label{eqn:type A 1}
  \ell(\alpha_{ij} , d) = \left[\squiggly{j-d+1} \Big| j-d \searrow i \Big| \squiggly{j-d+2 \nearrow j} \right]
\end{equation}
for all $1 \leq d \leq j-i+1 = |\alpha_{ij}|$. The set of words~\eqref{eqn:type A 1} can also be described as the dictionaries
of the following collection of rooted trees (for all $a \in \{1,\dots,n\}$). The root is in the top left of the picture, and
all the horizontal branches take the same form as the one displayed.
\begin{figure}[H]
\centering
\includegraphics[scale=0.34]{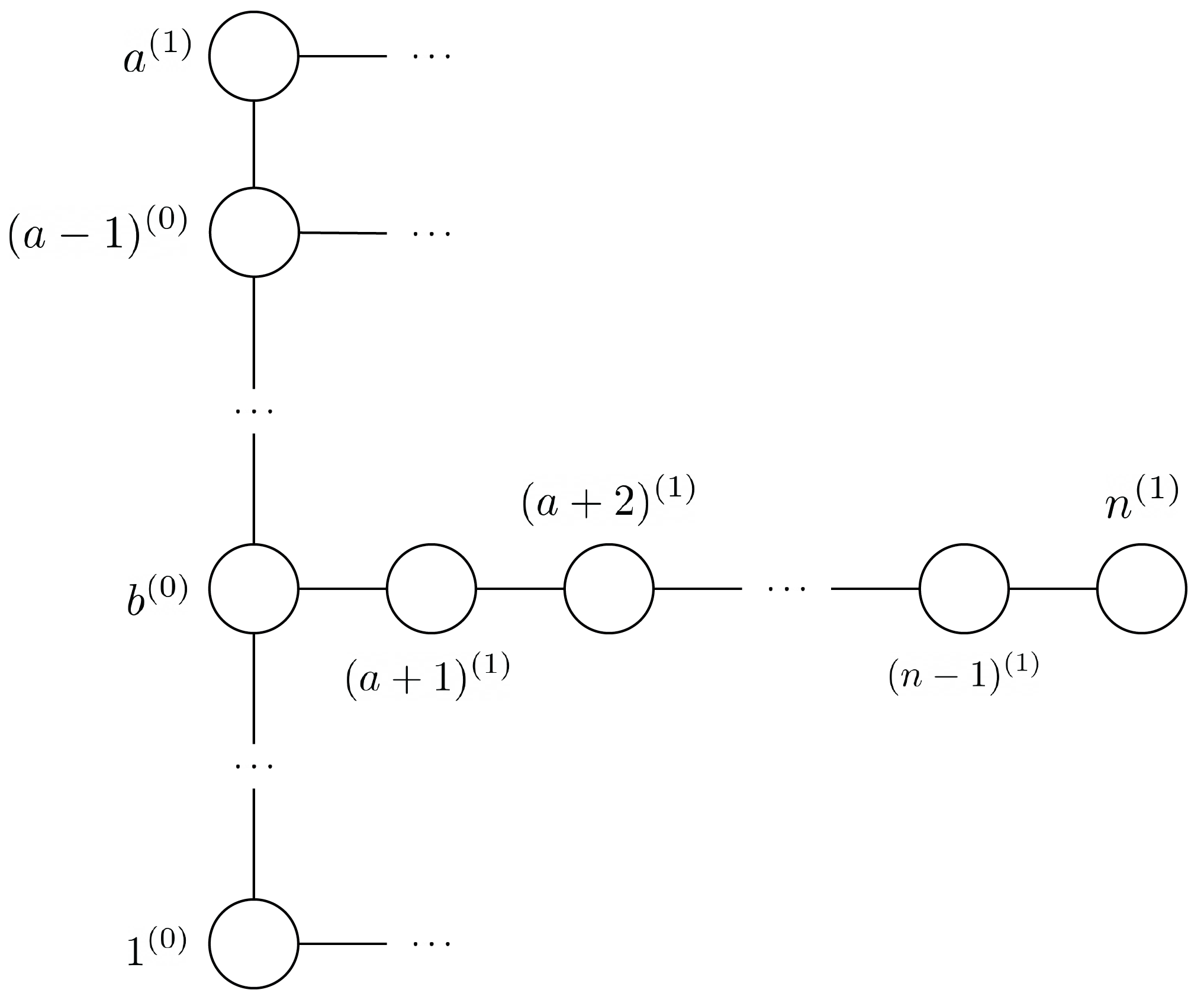}
\end{figure}

\medskip


\subsection{Type $B_n$}
\label{sub:Bn}

Consider the vertices of the Dynkin diagram as below:
\begin{figure}[H]
\centering
\includegraphics[scale=0.35]{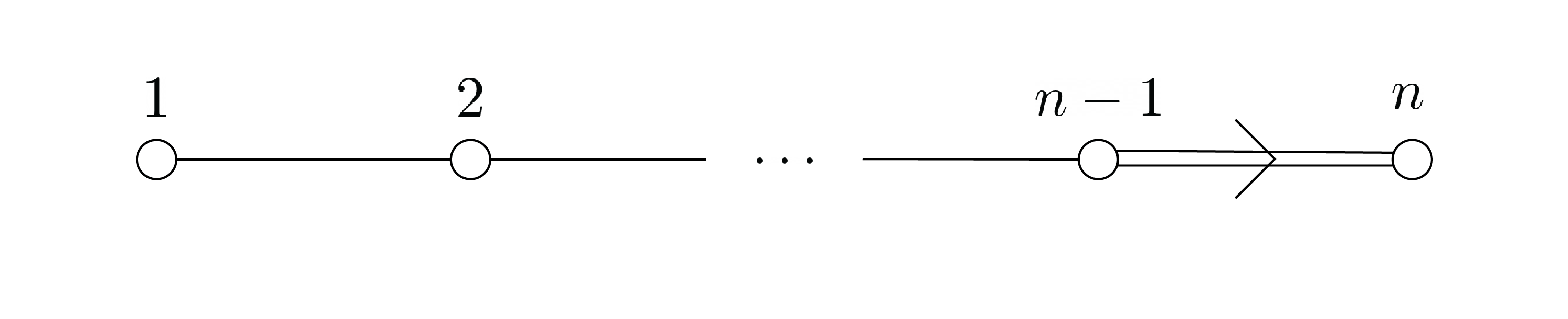}
\end{figure}

\noindent
Besides the positive roots \eqref{eqn:an}, we also have the following ones:
$$
  \beta_{ij} = \alpha_i + \dots + \alpha_{j-1} + 2\alpha_{j} + \dots + 2\alpha_n
$$
for all $1 \leq i < j \leq n$. Then the bijection \eqref{eqn:associated word loop} is given by the formulas of
Subsection \ref{sub:an} for the roots of the form \eqref{eqn:an}, together with the following formulas:
\begin{align*}
  & \ell(\beta_{ij} , 1 )= \left[ \squiggly{n}, n-1 , n , n-2, n-1 , \dots , j-1,j \Big| j-2 \searrow i \right] \\
  & \ell(\beta_{ij} , 2k) =
    \begin{cases}
      \left[
        \squiggly{a} \Big| a-1 \searrow j \Big| \squiggly{a+1 \nearrow n} \Big| \squiggly{n \searrow a}\Big|a-1 \searrow i
      \right]
        &\text{if } j \leq n-k+1 \\ \\
      \left[\squiggly{a} \Big| a-1 \searrow i \Big| \squiggly{a+1 \nearrow n} \Big| \squiggly{n \searrow j} \right]
        &\text{if } j \geq n-k+2 \end{cases} \\
  & \ell(\beta_{ij} , 2k + 1) =
    \begin{cases}
      \left[
        \squiggly{a} \Big| a-1 \searrow i \Big| \squiggly{a+1 \nearrow n} \Big| \squiggly{n \searrow a+1}\Big|a \searrow j
      \right]
        &\text{if } j \leq n-k+1 \\ \\
     \left[\squiggly{a} \Big| a-1 \searrow i \Big| \squiggly{a+1 \nearrow n} \Big| \squiggly{n \searrow j} \right]
       &\text{if } j \geq n-k+2
    \end{cases}
\end{align*}
for all $k \geq 1$ such that $2k$ (resp.\ $2k+1$) is less than or equal to $|\beta_{ij}|=2n-i-j+2$. In all formulas above,
the natural number $a$ is uniquely determined by the fact that the total number of letters with squiggly underlines is $2k$
(resp.\ $2k+1$) and can be easily expressed in terms of $i,j,k,n$.

\medskip

\noindent
Equivalently, this set of standard Lyndon loop words in type $B_n$ is also the collection of dictionaries of the following
rooted trees (for all $a \in \{1,\dots,n\}$). The root is in the top left of the picture, and all the horizontal branches take
the same form as the one displayed. The hollow circles marked with $\times$'s must be removed from the tree if $a = n$.
\begin{figure}[H]
\centering
\includegraphics[scale=0.34]{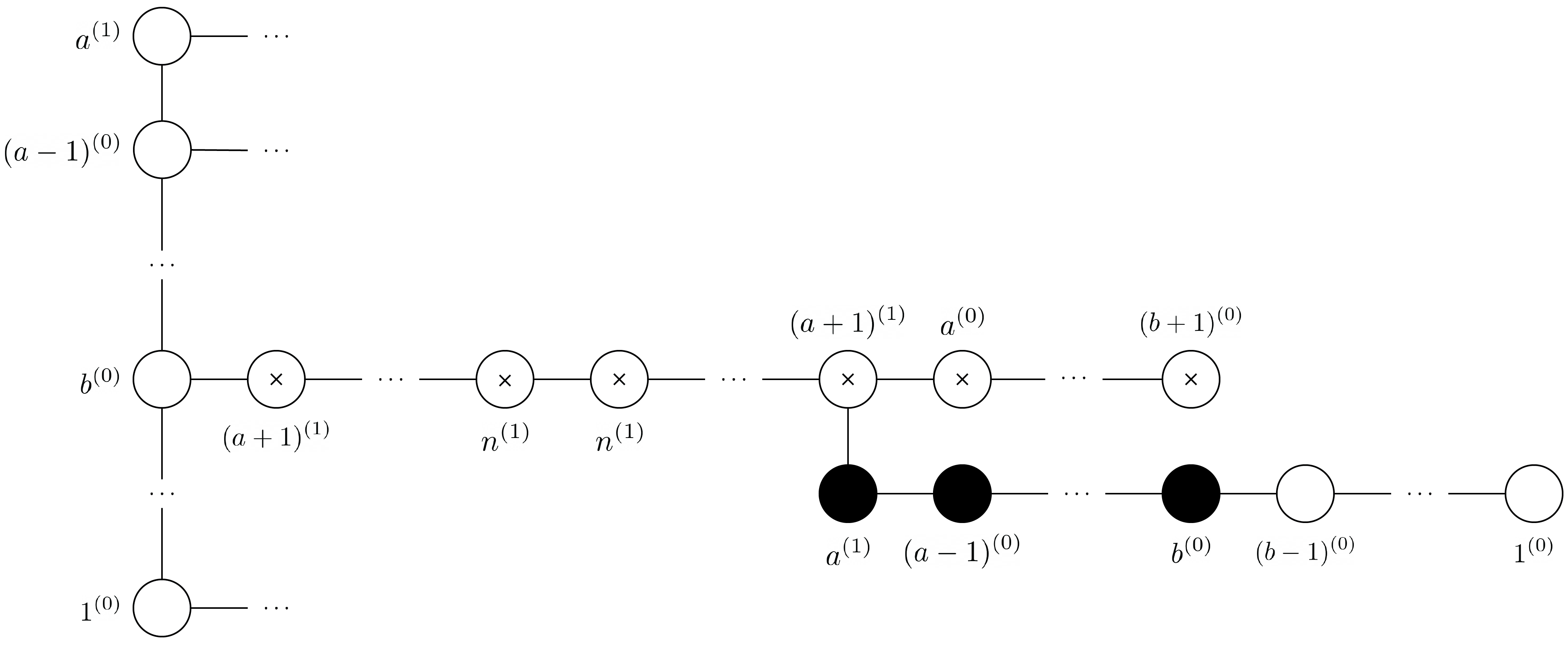}
\end{figure}

\noindent
and:
\begin{figure}[H]
\centering
\includegraphics[scale=0.34]{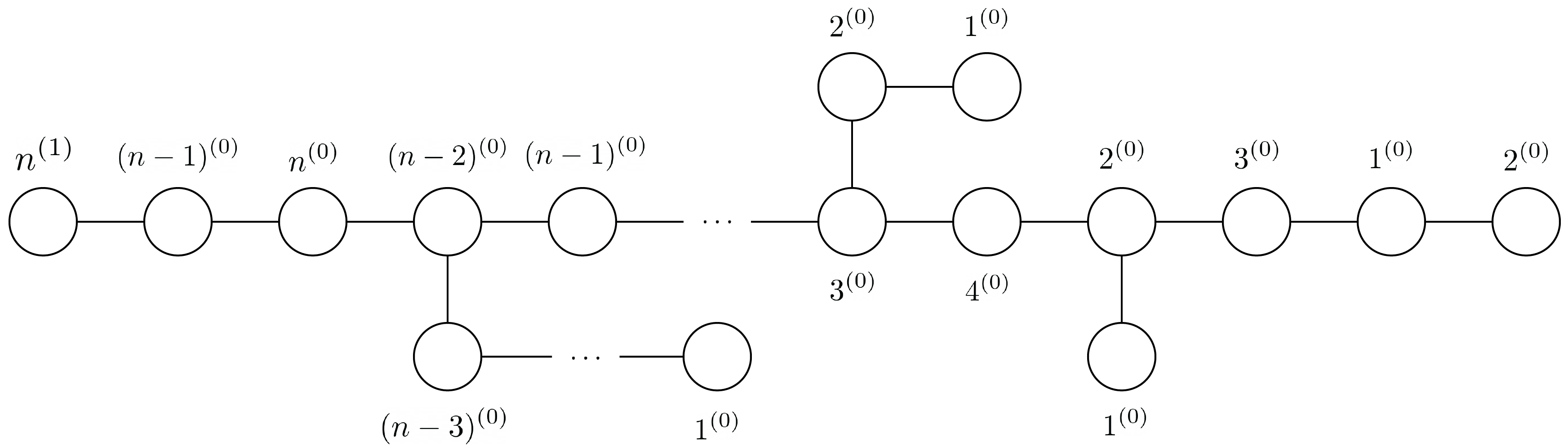}
\end{figure}

\medskip

\noindent
Note that the first picture also appears when $a = n$, so there are two trees describing words that start with the letter $n^{(1)}$.
Strictly speaking, we would need to glue these two trees along the first two hollow circles (namely the ones with labels $n^{(1)}$
and $(n-1)^{(0)}$), but we chose not to display them as such lest the picture be too unwieldy.

\medskip


\subsection{Type $C_n$}
\label{sub:cn}

Consider the vertices of the Dynkin diagram as below:
\begin{figure}[H]
\centering
\includegraphics[scale=0.35]{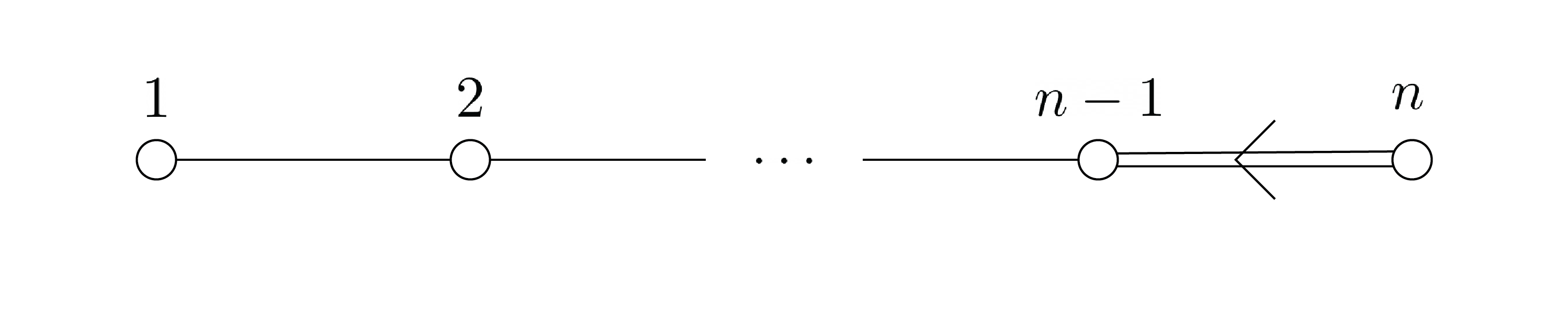}
\end{figure}

\noindent
Besides the positive roots \eqref{eqn:an}, we also have the following ones:
$$
  \gamma_{ij} = \alpha_i + \dots + \alpha_{j-1} + 2\alpha_{j} + \dots + 2\alpha_{n-1} + \alpha_n
$$
for all $1 \leq i \leq j < n$. Then the bijection \eqref{eqn:associated word loop} is given by the formulas
of Subsection \ref{sub:an} for the roots of the form \eqref{eqn:an}, together with the following formulas:
\begin{align*}
  & \ell(\gamma_{ij} , 1 ) =
    \left[ \squiggly{n}, n-1, n-1, n-2,n-2, \dots , j, j \Big| j-1 \searrow i \right] \\
  & \ell(\gamma_{ij} , 2k ) =
    \begin{cases}
      \left[ \squiggly{a} \Big| a-1 \searrow i \Big|\squiggly{a+1 \nearrow n \searrow a+1} \Big| a \searrow j \right]
        &\text{if } j \leq n-k \\ \\
      \left[ \squiggly{a} \Big| a-1 \searrow i \Big|\squiggly{a+1 \nearrow n \searrow j}\right]
        &\text{if } j \geq n-k+1
    \end{cases} \\
  & \ell(\gamma_{ij} , 2k + 1) =
    \begin{cases}
      \left[ \squiggly{a} \Big| a-1 \searrow i \Big| \squiggly{a+1 \nearrow n-1} \Big | \squiggly{a} \Big|
             a-1 \searrow i \Big| \squiggly{a+1 \nearrow n} \right] \ \ \ \ \ \text{if }i = j \\ \\
      \left[ \squiggly{a} \Big| a-1 \searrow j \Big|\squiggly{a+1 \nearrow n \searrow a} \Big| a-1 \searrow i \right]
         \qquad \ \ \text{if } i < j \leq n-k \\ \\
      \left[ \squiggly{a} \Big| a-1 \searrow i \Big|\squiggly{a+1 \nearrow n \searrow j}\right]
         \qquad \qquad \qquad \text{if } i < j \geq n-k+1
    \end{cases}
\end{align*}
for all $k \geq 1$ such that $2k$ (resp.\ $2k+1$) is less than or equal to $|\gamma_{ij}|=2n-i-j+1$.
In all formulas above, the natural number $a$ is uniquely determined by the fact that the total number of letters
with squiggly underlines is $2k$ (resp.\ $2k+1$) and can be easily expressed in terms of $i,j,k,n$.

\medskip

\noindent
Equivalently, this set of standard Lyndon loop words in type $C_n$ is also the collection of dictionaries of the following
rooted trees (for all $a \in \{1,\dots,n-1\}$). The root is in the top left of the picture, and all the horizontal branches take
the same form as the one displayed.
\begin{figure}[H]
\centering
\includegraphics[scale=0.34]{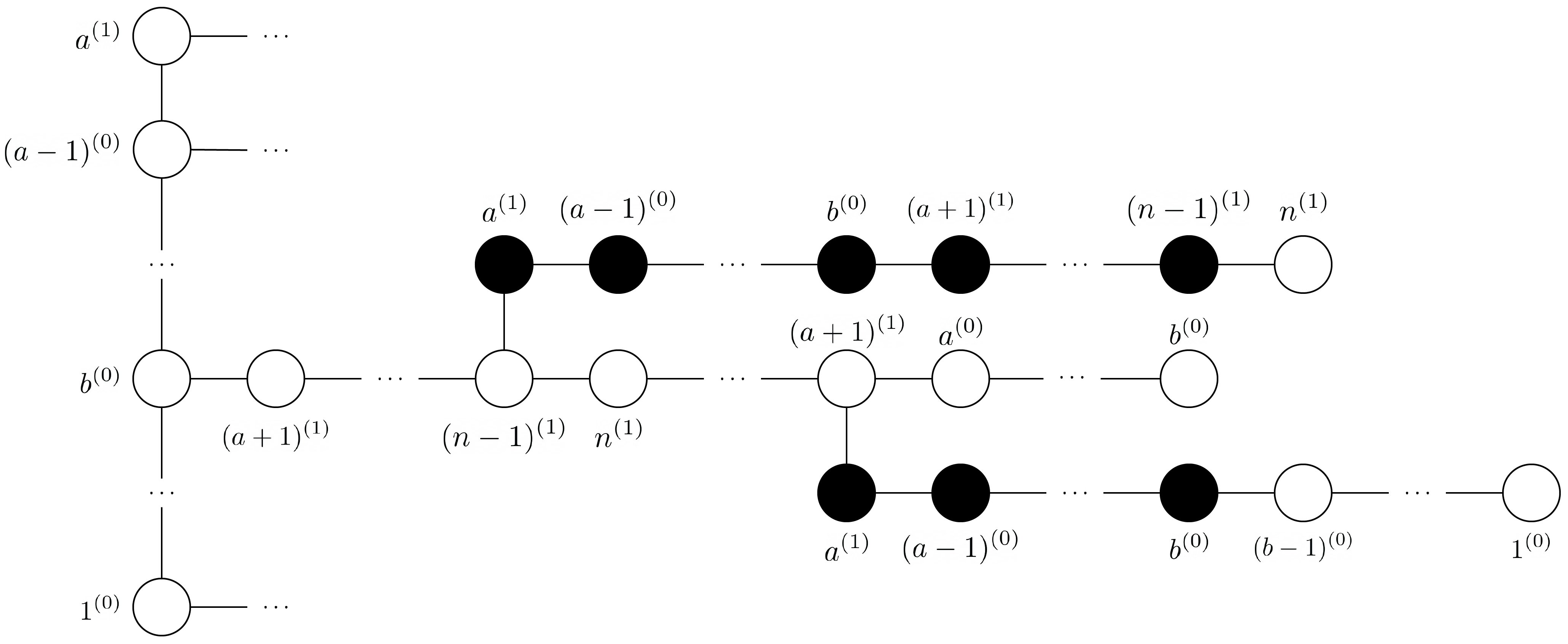}
\end{figure}

\noindent
and:
\begin{figure}[H]
\centering
\includegraphics[scale=0.34]{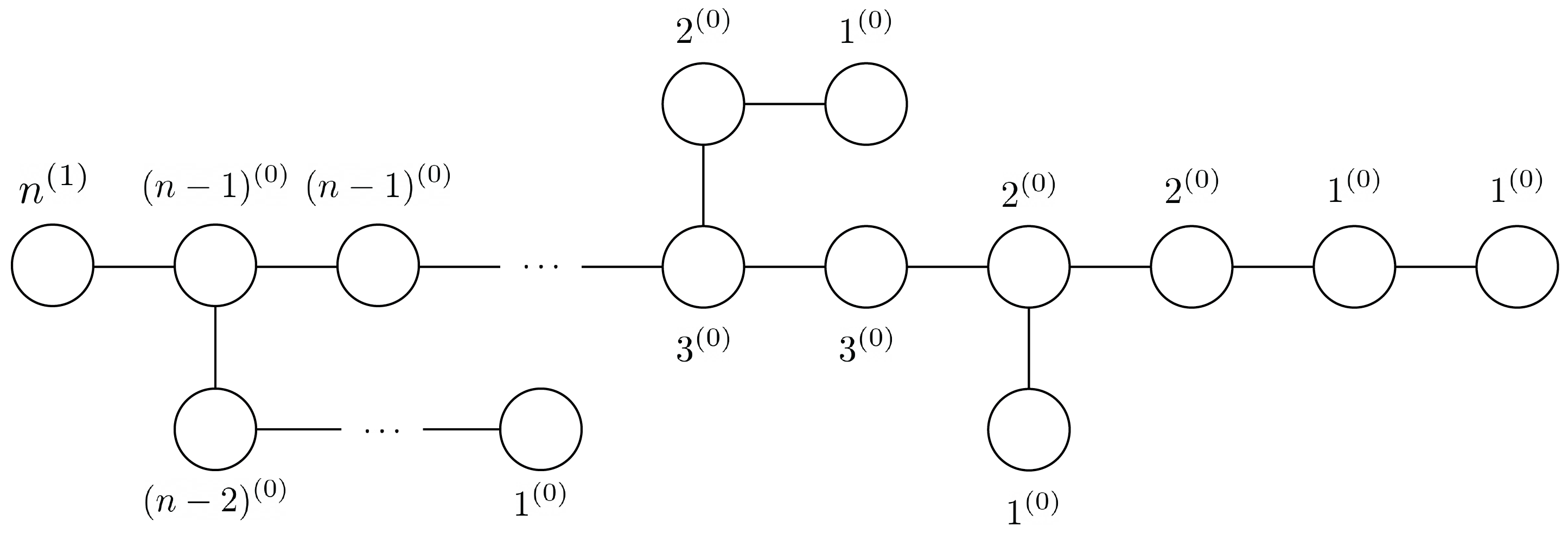}
\end{figure}

\medskip


\subsection{Type $D_n$}
\label{sub:dn}

Consider the vertices of the Dynkin diagram as below:
\begin{figure}[H]
\centering
\includegraphics[scale=0.35]{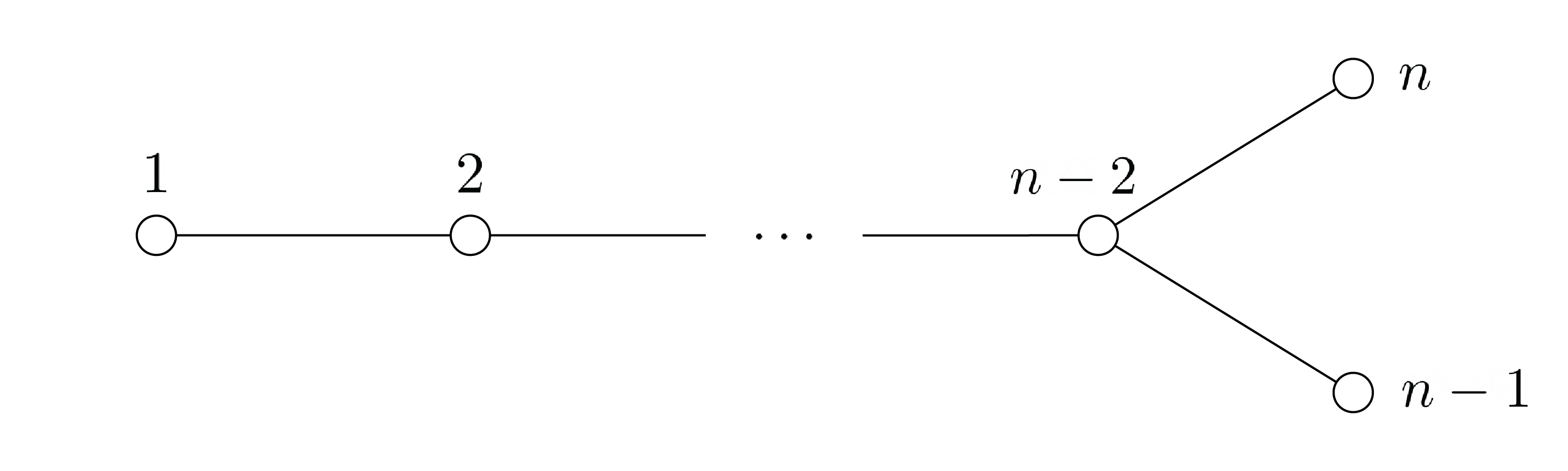}
\end{figure}

\noindent
Besides the positive roots \eqref{eqn:an} (for $n$ replaced by $n-1$, as well as for $n$ replaced by $n-1$
and the simple root $\alpha_{n-1}$ replaced by $\alpha_n$), we also have the following ones:
\begin{align*}
  & \sigma_j = \alpha_j + \dots + \alpha_{n-2} + \alpha_{n-1} + \alpha_n \\
  & \tau_{ij} = \alpha_i+ \dots + \alpha_{j-1} + 2\alpha_j + \dots + 2 \alpha_{n-2} + \alpha_{n-1} + \alpha_n
\end{align*}
for all $1 \leq i < j \leq n-2$. Then the bijection \eqref{eqn:associated word loop} is given by the formulas
of Subsection \ref{sub:an} for the roots of the form \eqref{eqn:an}, together with the following formulas:
\begin{align*}
  & \ell(\sigma_j,1) = \left[\squiggly{n}, n-2, n-1 \Big| n-3 \searrow j \right] \\
  & \ell(\sigma_j,2) = \left[\squiggly{n-1} \Big| n-2 \searrow j \Big| \squiggly{n} \right] \\
  & \ell(\sigma_j,d) =
    \left[\squiggly{n-d+1} \Big| n-d \searrow j \Big| \squiggly{n-d+2 \nearrow n-2} \Big| \squiggly{n, n-1} \right]
\end{align*}
for all $d \in \{3,\dots,n-j+1\}$, as well as:
\begin{align*}
  & \ell(\tau_{ij},1) = \left[ \squiggly{n}, n-2 , n-1 , n-3, n-2 , \dots ,
    j, j+1, j-1, j \Big| j-2 \searrow i \right] \\
  & \ell(\tau_{ij},2) = \left[ \squiggly{n-1} \Big| n-2 \searrow i \Big| \squiggly{n} \Big| n-2 \searrow j \right]  \\
  & \ell(\tau_{ij},3) = \left[ \squiggly{n-2} \Big| n-3 \searrow i \Big| \squiggly{n,n-1} \Big| n-2 \searrow j \right] \\
  & \ell(\tau_{ij},2k) =
    \begin{cases}
      \left[\squiggly{a} \Big| a-1 \searrow j \Big| \squiggly{a+1 \nearrow n-2} \Big|
            \squiggly{n \searrow a} \Big| a-1 \searrow i \right] &\text{if } j \leq n-k \\ \\
      \left[\squiggly{a} \Big| a-1 \searrow i \Big| \squiggly{a+1 \nearrow n-2} \Big| \squiggly{n \searrow j} \right]
        &\text{if } j \geq n-k + 1\end{cases} \\
  & \ell(\tau_{ij},2k+1) =
    \begin{cases}
      \left[\squiggly{a} \Big| a-1 \searrow i \Big| \squiggly{a+1 \nearrow n-2} \Big| \squiggly{n \searrow a+1} \Big|
        a \searrow j \right] &\text{if } j \leq n-k-1 \\ \\
      \left[\squiggly{a} \Big| a-1 \searrow i \Big| \squiggly{a+1 \nearrow n-2} \Big| \squiggly{n \searrow j} \right]
        &\text{if } j \geq n-k
  \end{cases}
\end{align*}
for all $k \geq 2$ such that $2k$ (resp.\ $2k+1$) is less than or equal to $|\tau_{ij}|=2n-i-j$.
In all formulas above, the natural number $a$ is uniquely determined by the fact that the total number of letters
with squiggly underlines is $2k$ (resp.\ $2k+1$) and can be easily expressed in terms of $i,j,k,n$.

\medskip

\noindent
Equivalently, this set of standard Lyndon loop words in type $D_n$ is also the collection of dictionaries of the following
rooted trees (for all $a \in \{1,\dots,n-1\}$). The root is in the top left of the picture, and all the horizontal branches take
the same form as the one displayed. The hollow circles marked with $\times$'s must be removed from the tree if $a = n-1$.
\begin{figure}[H]
\centering
\includegraphics[scale=0.34]{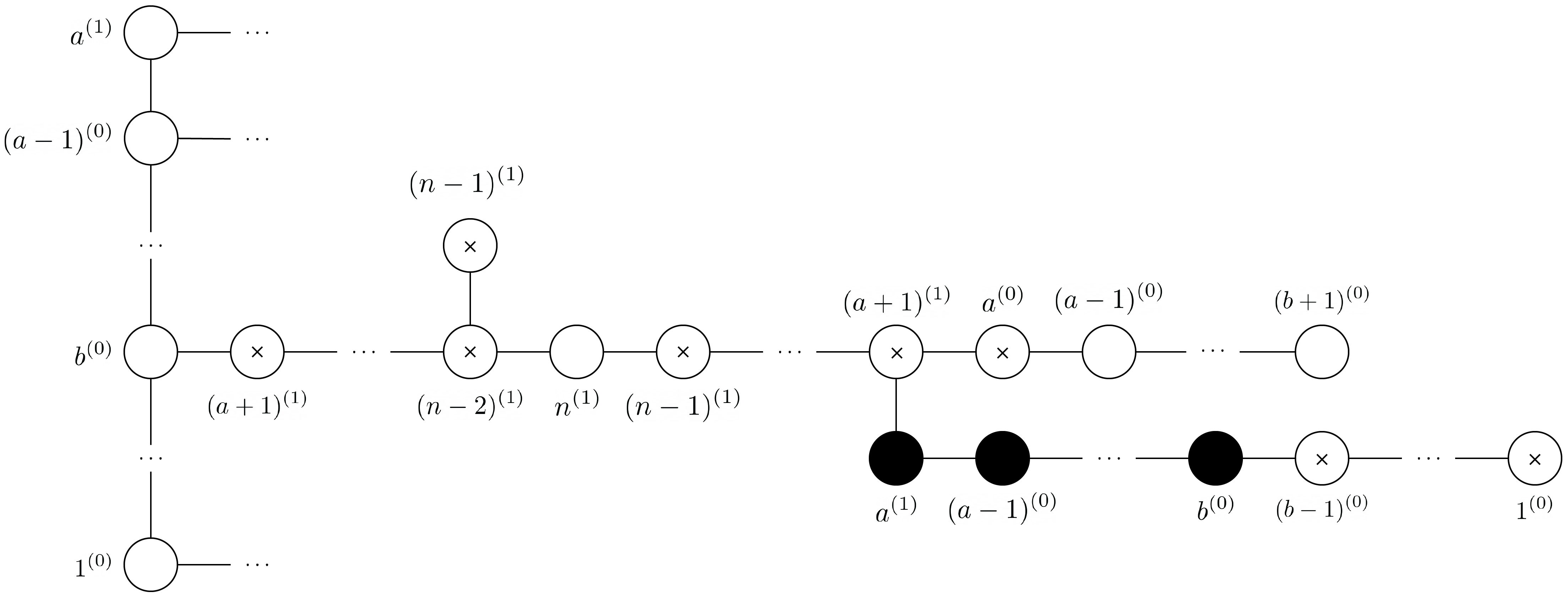}
\end{figure}

\noindent
and:
\begin{figure}[H]
\centering
\includegraphics[scale=0.34]{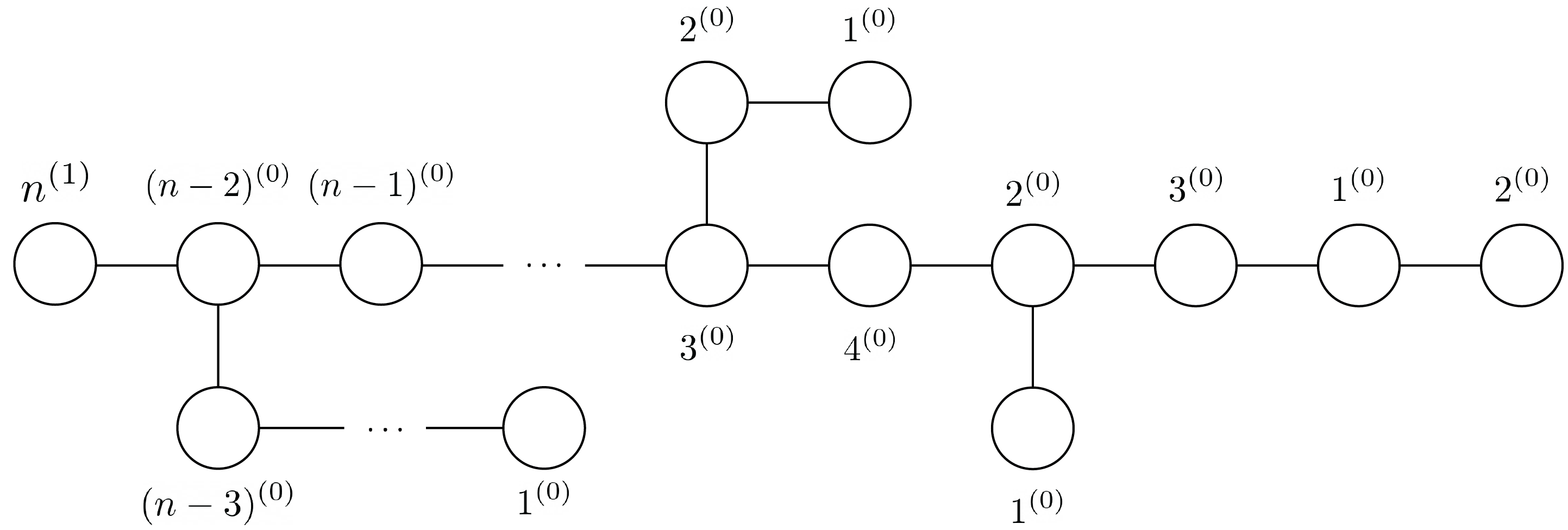}
\end{figure}

\medskip


\end{document}